\documentclass{amsart}

\usepackage{ adjustbox }
\usepackage{ amsthm }
\usepackage{ amssymb }
\usepackage{ amsmath }
\usepackage{ array }
\usepackage{ booktabs }
\usepackage{ cite }
\usepackage{ color }
\usepackage{ enumitem }
\usepackage{ latexsym }
\usepackage{ qtree }
\usepackage{ scalerel }
\usepackage{ url }
\usepackage[all]{ xy }

\newcommand{\bigplus} {{\vstretch{1.5}{\hstretch{1.5}{+}}}}
\newcommand{\bigminus}{{\vstretch{1.5}{\hstretch{1.5}{-}}}}
\newcommand{\oa}{m}

\newcommand{\trees}{\mathcal{T}}
\newcommand{\freeternary}{\mathcal{LT}}
\newcommand{\tpa}{\mathcal{TP\!A}}

\theoremstyle{definition}
\newtheorem{definition}{Definition}[section]

\newtheorem{remark}[definition]{Remark}

\theoremstyle{plain}
\newtheorem{lemma}[definition]{Lemma}

\newtheorem{theorem}[definition]{Theorem}

\allowdisplaybreaks

\begin{document}


\title[Gr\"obner bases for ternary partially associative operads]
{Gr\"obner bases and dimension formulas for ternary partially associative operads}

\author{Fatemeh Bagherzadeh}

\address{Department of Mathematics and Statistics, University of Saskatchewan, Canada}

\email{bagherzadeh@math.usask.ca}

\author{Murray Bremner}

\address{Department of Mathematics and Statistics, University of Saskatchewan, Canada}

\email{bremner@math.usask.ca}



\thanks{This research was supported by the Discovery Grant \emph{Algebraic Operads} (2016) 
from the Natural Sciences and Engineering Research Council of Canada (NSERC).
The authors thank Vladimir Dotsenko for sending us his results without proofs on Gr\"obner bases 
and dimension formulas for ternary partially associative operads (emails of 16 \& 26 February 2018) 
and for helpful comments on an earlier version of this paper}

\begin{abstract}
Dotsenko and Vallette discovered an extension to nonsymmetric operads
of Buchberger's algorithm for Gr\"obner bases of polynomial ideals.
In the free nonsymmetric operad with one ternary operation $({\ast}{\ast}{\ast})$,
we compute a Gr\"obner basis for the ideal generated by 
partial associativity $((abc)de) + (a(bcd)e) + (ab(cde)$. 
In the category of $\mathbb{Z}$-graded vector spaces with
Koszul signs, the (homological) degree of $({\ast}{\ast}{\ast})$ may be even or odd.
We use the Gr\"obner bases to calculate the dimension formulas for these operads.
\end{abstract}

\maketitle



\section{Introduction}

We consider nonsymmetric operads in the category of $\mathbb{Z}$-graded vector
spaces over a field of characteristic 0.
The product is the tensor product (with Koszul signs) and the coproduct is the direct sum.
Gr\"obner bases for operads were introduced by Dotsenko, Khoroshkin \& Vallette \cite{DK,DV};
see also \cite{BD}.

Let $\freeternary$ be the free nonsymmetric operad with one ternary operation 
$t = ({\ast}{\ast}{\ast})$.
Let $\alpha$ denote ternary partial associativity, which may be written 
as a tree polynomial,
using partial compositions, or
as a nonassociative polynomial:
\begin{equation}
\label{tpa}
\quad
\alpha = \!\!
\adjustbox{valign=m}{
\begin{xy}
(  0,  0 )*{} = "root";
( -3, -4 )*{} = "l";
(  0, -4 )*{} = "m";
(  3, -4 )*{} = "r";
{ \ar@{-} "root"; "l" };
{ \ar@{-} "root"; "m" };
{ \ar@{-} "root"; "r" };
( -6, -8 )*{} = "ll";
( -3, -8 )*{} = "lm";
(  0, -8 )*{} = "lr";
{ \ar@{-} "l"; "ll" };
{ \ar@{-} "l"; "lm" };
{ \ar@{-} "l"; "lr" };
\end{xy}
}
\bigplus
\adjustbox{valign=m}{
\begin{xy}
(  0,  0 )*{} = "root";
( -3, -4 )*{} = "l";
(  0, -4 )*{} = "m";
(  3, -4 )*{} = "r";
{ \ar@{-} "root"; "l" };
{ \ar@{-} "root"; "m" };
{ \ar@{-} "root"; "r" };
( -3, -8 )*{} = "ml";
(  0, -8 )*{} = "mm";
(  3, -8 )*{} = "mr";
{ \ar@{-} "m"; "ml" };
{ \ar@{-} "m"; "mm" };
{ \ar@{-} "m"; "mr" };
\end{xy}
}
\bigplus
\adjustbox{valign=m}{
\begin{xy}
(  0,  0 )*{} = "root";
( -3, -4 )*{} = "l";
(  0, -4 )*{} = "m";
(  3, -4 )*{} = "r";
{ \ar@{-} "root"; "l" };
{ \ar@{-} "root"; "m" };
{ \ar@{-} "root"; "r" };
(  0, -8 )*{} = "rl";
(  3, -8 )*{} = "rm";
(  6, -8 )*{} = "rr";
{ \ar@{-} "r"; "rl" };
{ \ar@{-} "r"; "rm" };
{ \ar@{-} "r"; "rr" };
\end{xy}
}
\qquad
\adjustbox{valign=m}{$
\begin{array}{l}
t \circ_1 t + t \circ_2 t + t \circ_3 t,
\\[1pt]
(({\ast}{\ast}{\ast}){\ast}{\ast}) +
({\ast}({\ast}{\ast}{\ast}){\ast}) +
({\ast}{\ast}({\ast}{\ast}{\ast})).
\end{array}
$}
\quad
\end{equation}
We compute a Gr\"obner basis for the ideal $\langle \alpha \rangle$
when $t$ has even (homological) degree so that Koszul signs are irrelevant,
and when $t$ has odd degree so that Koszul signs are essential.
We include details of the calculations to clarify the Gr\"obner basis algorithm
for nonsymmetric operads.
As an application, we calculate dimension formulas for the quotient operads.
Similar results have been obtained independently in unpublished work
of Vladimir Dotsenko.

For earlier work on partial associativity and its applications, 
see \cite{B,C,DMR,G1,G2,GR,MR1,MR2,R}.
Recent results of Dotsenko, Shadrin \& Vallette \cite{DSV} have shown that 
the ternary partially associative operad with an odd generator 
arises naturally in the homology of the poset of interval partitions into intervals 
of odd length and in certain De Concini-Procesi models of subspace arrangements \cite{DP} 
over the real numbers.


\section{Preliminaries}

\begin{definition}
An \textit{$\oa$-ary tree} is a rooted plane tree $p$ in which every node
has no children (\textit{leaf}) or $\oa$ children (\textit{internal node}).
The \textit{weight} $w(p)$ counts internal nodes;
the \textit{arity} $\ell(p) = 1 + w(p) (\oa{-}1)$ counts leaves
indexed $1, \dots, \ell(p)$ from left to right.
The \emph{basic tree} $t$ is the $m$-ary tree of weight 1.
Set $[n] = \{ 1, \dots, n \}$.
\end{definition}

\begin{definition}
If $n \equiv 1$ (mod $\oa{-}1$) then $\trees(n)$ is the set of $\oa$-ary trees of arity~$n$,
and $\trees$ is the disjoint union of the $\trees(n)$ for $n \ge 1$.
\end{definition}

\begin{definition}
If $p$, $q \in \trees$ then for $i \in [\ell(p)]$ the \textit{partial composition}
$p \circ_i q \in \trees$ is obtained by identifying leaf $i$ of $p$ with the root of $q$.
\end{definition}

\begin{lemma}
Starting with $t$, every $\oa$-ary tree of weight $w$ can be obtained by a sequence of $w{-}1$ partial compositions.
\end{lemma}

\begin{lemma}
\label{pcrules}
Let $p$, $q$, $r$ be $\oa$-ary trees.
Partial composition satisfies \emph{\cite[p.~72]{BD}}:
\[
( p \circ_i q ) \circ_j r
=
\begin{cases}
\; p \circ_i ( q \circ_{j-i+1} r ), &\;\; i \le j \le i{+}\ell(q){-}1;
\\
\; ( p \circ_{j-\ell(q)+1} r ) \circ_i q, &\;\; i{+}\ell(q) \le j \le \ell(p){+}\ell(q){-}1;
\\
\; ( p \circ_j r ) \circ_{i+\ell(r)-1} q, &\;\; 1 \le j \le i{-}1.
\end{cases}
\]
\end{lemma}

\begin{lemma}
The set $\trees$ with partial compositions is isomorphic to
the free nonsymmetric (set) operad with one $\oa$-ary operation $t$.
\end{lemma}

\begin{definition}
If $n \equiv 1$ (mod $\oa{-}1$) then 
$\freeternary(n)$ is the vector space with basis $\trees(n)$,
and $\freeternary$ is the direct sum of $\freeternary(n)$ for $n \ge 1$.
A \emph{tree polynomial} of arity $n$ is an element of $\freeternary(n)$.
Partial composition in $\trees$ extends bilinearly to $\freeternary$.
\end{definition}

\begin{lemma}
The vector space $\freeternary$ with partial compositions is isomorphic to
the free nonsymmetric (vector) operad with one $\oa$-ary operation $t$.
\end{lemma}

\begin{definition}
A \textit{relation} of arity $n$ is an element of $\freeternary(n) \setminus 0$.
For relations $f_1, \dots, f_k$ let
$\mathcal{I} = \langle f_1, \dots, f_k \rangle = \bigcap \mathcal{S}$
over homogeneous subspaces $\mathcal{S} \subseteq \freeternary$ with
$f_1, \dots, f_k \in \mathcal{S}$,
and
$f \in \mathcal{S}(m), g \in \freeternary(n) \implies f \circ_i g, g \circ_j f \in \mathcal{S}$ 
($i \in [m]$, $j \in [n]$). 
\end{definition}

The following results come from \cite[\S3.4]{BD} and \cite[\S\S2.4,~3.1]{DV} with minor changes.

\begin{definition}
The \textit{path sequence} of $p \in \trees(n)$ is $\mathrm{path}(p) = ( a_1, \dots, a_n )$
where $a_i$ is the length of the path from the root to the leaf $i$.
\end{definition}

\begin{lemma}
If $p, q \in \trees$ then $p = q$ if and only if $\mathrm{path}(p) = \mathrm{path}(q)$.
\end{lemma}

\begin{definition}
For $p, q \in \trees(n)$ we write $p \prec q$ and say $p$ \textit{precedes} $q$ in \emph{path-lex order}
if and only if $\mathrm{path}(p) \prec \mathrm{path}(q)$ in lex order on $n$-tuples of positive integers.
If $f \in \freeternary(n)$ then its \emph{leading monomial} $\ell m(f) \in \trees(n) $ is the greatest monomial
in path-lex order, and its \emph{leading coefficient} $\ell c(f)$ is the coefficient of $\ell m(f)$.
\end{definition}

\begin{definition}
\label{Mpqf}
If $p, q \in \trees$ then $q$ is \emph{divisible} by $p$ (written $p \mid q$)
if $p$ is a subtree of $q$:
that is,
$q = \cdots p \cdots$ where
the dots denote sequences of partial compositions with parentheses.
If $p \in \trees(m)$, $q \in \trees(n)$, $p \mid q$, and $f \in \freeternary(m)$
then we may replace $p$ by $f$ in $q$ and use linearity and the same partial compositions to obtain
\[
M(q,p,f) = \cdots f \cdots \in \freeternary(n).
\]
\end{definition}

\begin{definition}
If $f$, $g \in \freeternary$ and $\ell m(g) \mid \ell m(f)$ then
the \emph{reduction of $f$ by $g$} (which eliminates the leading term of $f$)
is
\[
R(f,g) = f - \frac{\ell c(f)}{\ell c(g)} M\big( \ell m(f), \ell m(g), g \big).
\]
This extends to reduction of $f$ by $g_1, \dots, g_k$; see \cite[Algorithm 3.4.2.16]{BD}.
\end{definition}

\begin{definition}
If $p, q, r \in \trees$ then $p$ is a \emph{small common multiple} (SCM) of $q$, $r$ if
\[
q \mid p, \;\;\;
r \mid p, \;\;\;
\text{every node of $p$ is a node of $q$ or $r$ (or both)}, \;\;\;
\ell(p) < \ell(q) + \ell(r).
\]
\end{definition}

\begin{definition}
If $f$, $g$, $h$ are monic tree polynomials and $\ell m(f)$ is an SCM of $\ell m(g)$, $\ell m(h)$
then the resulting \emph{S-polynomial} is
\[
S( f, g, h ) =
M\big( \ell m(f), \ell m(g), g \big)
-
M\big( \ell m(f), \ell m(h), h \big).
\]
\end{definition}

\begin{definition}
Let $G$ be a finite set of relations and let $I = \langle G \rangle$.
If for all $f \in I$ there exists $g \in G$ such that $\ell m(g) \mid \ell m(f)$
then we call $G$ a \emph{Gr\"obner basis} for $I$.
We say $G$ is \emph{reduced} if $\ell m(g)$ is not divisible by $\ell m(h)$ for all $g$, $h \in G$.
\end{definition}

\begin{lemma}
Every operad ideal has a unique reduced Gr\"obner basis.
\end{lemma}

\begin{theorem}
If $I = \langle G \rangle$ then
$G$ is a Gr\"obner basis for $I$ if and only if for every SCM $f$ of elements $g$, $h \in G$
the reduction of $S(f,g,h)$ by $G$ is 0.
\end{theorem}


\section{Gr\"obner bases and dimension formulas}
\label{evensection}

In the rest of this paper we consider a ternary operation ($\oa = 3$).
We usually indicate the leading monomial of a tree polynomial by a bullet at the root, and
write the terms of a tree polynomial from left to right in reverse path-lex order.
The partially associative relation $\alpha$ corresponds to this rewrite rule:
\begin{equation}
\label{frr}
t \circ_1 t
\; =
\adjustbox{valign=m}{
\begin{xy}
(  0,  0 )*{} = "root";
( -3, -4 )*{} = "l";
(  0, -4 )*{} = "m";
(  3, -4 )*{} = "r";
{ \ar@{-} "root"; "l" };
{ \ar@{-} "root"; "m" };
{ \ar@{-} "root"; "r" };
( -6, -8 )*{} = "ll";
( -3, -8 )*{} = "lm";
(  0, -8 )*{} = "lr";
{ \ar@{-} "l"; "ll" };
{ \ar@{-} "l"; "lm" };
{ \ar@{-} "l"; "lr" };
\end{xy}
}
\;
\xrightarrow{\qquad}
\;
\bigminus
\adjustbox{valign=m}{
\begin{xy}
(  0,  0 )*{} = "root";
( -3, -4 )*{} = "l";
(  0, -4 )*{} = "m";
(  3, -4 )*{} = "r";
{ \ar@{-} "root"; "l" };
{ \ar@{-} "root"; "m" };
{ \ar@{-} "root"; "r" };
( -3, -8 )*{} = "ml";
(  0, -8 )*{} = "mm";
(  3, -8 )*{} = "mr";
{ \ar@{-} "m"; "ml" };
{ \ar@{-} "m"; "mm" };
{ \ar@{-} "m"; "mr" };
\end{xy}
}
\bigminus
\adjustbox{valign=m}{
\begin{xy}
(  0,  0 )*{} = "root";
( -3, -4 )*{} = "l";
(  0, -4 )*{} = "m";
(  3, -4 )*{} = "r";
{ \ar@{-} "root"; "l" };
{ \ar@{-} "root"; "m" };
{ \ar@{-} "root"; "r" };
(  0, -8 )*{} = "rl";
(  3, -8 )*{} = "rm";
(  6, -8 )*{} = "rr";
{ \ar@{-} "r"; "rl" };
{ \ar@{-} "r"; "rm" };
{ \ar@{-} "r"; "rr" };
\end{xy}
}
=
{} - t \circ_2 t - t \circ_3 t
\end{equation}

\begin{theorem}
\label{eventheorem}
For the path-lex monomial order, the following tree polynomials
form the reduced Gr\"obner basis for $\langle \alpha \rangle$
with an operation of even degree:
\begin{alignat*}{2}
&
\adjustbox{valign=t}{$
\alpha
=
\adjustbox{valign=m}{
\begin{xy}
(  0,  0 )*{\bullet} = "root";
( -3, -4 )*{} = "l";
(  0, -4 )*{} = "m";
(  3, -4 )*{} = "r";
{ \ar@{-} "root"; "l" };
{ \ar@{-} "root"; "m" };
{ \ar@{-} "root"; "r" };
( -6, -8 )*{} = "ll";
( -3, -8 )*{} = "lm";
(  0, -8 )*{} = "lr";
{ \ar@{-} "l"; "ll" };
{ \ar@{-} "l"; "lm" };
{ \ar@{-} "l"; "lr" };
\end{xy}
}
\bigplus
\adjustbox{valign=m}{
\begin{xy}
(  0,  0 )*{} = "root";
( -3, -4 )*{} = "l";
(  0, -4 )*{} = "m";
(  3, -4 )*{} = "r";
{ \ar@{-} "root"; "l" };
{ \ar@{-} "root"; "m" };
{ \ar@{-} "root"; "r" };
( -3, -8 )*{} = "ml";
(  0, -8 )*{} = "mm";
(  3, -8 )*{} = "mr";
{ \ar@{-} "m"; "ml" };
{ \ar@{-} "m"; "mm" };
{ \ar@{-} "m"; "mr" };
\end{xy}
}
\bigplus
\adjustbox{valign=m}{
\begin{xy}
(  0,  0 )*{} = "root";
( -3, -4 )*{} = "l";
(  0, -4 )*{} = "m";
(  3, -4 )*{} = "r";
{ \ar@{-} "root"; "l" };
{ \ar@{-} "root"; "m" };
{ \ar@{-} "root"; "r" };
(  0, -8 )*{} = "rl";
(  3, -8 )*{} = "rm";
(  6, -8 )*{} = "rr";
{ \ar@{-} "r"; "rl" };
{ \ar@{-} "r"; "rm" };
{ \ar@{-} "r"; "rr" };
\end{xy}
}
$}
&\qquad
&
\adjustbox{valign=t}{$
\beta
=
\adjustbox{valign=m}{
\begin{xy}
(  0,  0 )*{\bullet} = "root";
( -3, -4 )*{} = "l";
(  0, -4 )*{} = "m";
(  3, -4 )*{} = "r";
{ \ar@{-} "root"; "l" };
{ \ar@{-} "root"; "m" };
{ \ar@{-} "root"; "r" };
( -3, -8 )*{} = "ml";
(  0, -8 )*{} = "mm";
(  3, -8 )*{} = "mr";
{ \ar@{-} "m"; "ml" };
{ \ar@{-} "m"; "mm" };
{ \ar@{-} "m"; "mr" };
(  0, -12 )*{} = "mrl";
(  3, -12 )*{} = "mrm";
(  6, -12 )*{} = "mrr";
{ \ar@{-} "mr"; "mrl" };
{ \ar@{-} "mr"; "mrm" };
{ \ar@{-} "mr"; "mrr" };
\end{xy}
}
\bigplus
\adjustbox{valign=m}{
\begin{xy}
(  0,  0 )*{} = "root";
( -3, -4 )*{} = "l";
(  0, -4 )*{} = "m";
(  3, -4 )*{} = "r";
{ \ar@{-} "root"; "l" };
{ \ar@{-} "root"; "m" };
{ \ar@{-} "root"; "r" };
(  0, -8 )*{} = "rl";
(  3, -8 )*{} = "rm";
(  6, -8 )*{} = "rr";
{ \ar@{-} "r"; "rl" };
{ \ar@{-} "r"; "rm" };
{ \ar@{-} "r"; "rr" };
(  0, -12 )*{} = "rml";
(  3, -12 )*{} = "rmm";
(  6, -12 )*{} = "rmr";
{ \ar@{-} "rm"; "rml" };
{ \ar@{-} "rm"; "rmm" };
{ \ar@{-} "rm"; "rmr" };
\end{xy}
}
\bigplus
\adjustbox{valign=m}{
\begin{xy}
(  0,  0 )*{} = "root";
( -3, -4 )*{} = "l";
(  0, -4 )*{} = "m";
(  3, -4 )*{} = "r";
{ \ar@{-} "root"; "l" };
{ \ar@{-} "root"; "m" };
{ \ar@{-} "root"; "r" };
(  0, -8 )*{} = "rl";
(  3, -8 )*{} = "rm";
(  6, -8 )*{} = "rr";
{ \ar@{-} "r"; "rl" };
{ \ar@{-} "r"; "rm" };
{ \ar@{-} "r"; "rr" };
(  3, -12 )*{} = "rrl";
(  6, -12 )*{} = "rrm";
(  9, -12 )*{} = "rrr";
{ \ar@{-} "rr"; "rrl" };
{ \ar@{-} "rr"; "rrm" };
{ \ar@{-} "rr"; "rrr" };
\end{xy}
}
$}
\\
&
\adjustbox{valign=t}{$
\eta
=
\adjustbox{valign=m}{
\begin{xy}
(  3,   0 )*{\bullet} = "root";
(  0,  -4 )*{} = "l";
(  3,  -4 )*{} = "m";
( 12,  -4 )*{} = "r";
{ \ar@{-} "root"; "l" };
{ \ar@{-} "root"; "m" };
{ \ar@{-} "root"; "r" };
( 0, -8 )*{} = "ll";
(  3, -8 )*{} = "lm";
(  6, -8 )*{} = "lr";
{ \ar@{-} "m"; "ll" };
{ \ar@{-} "m"; "lm" };
{ \ar@{-} "m"; "lr" };
( 9, -8 )*{} = "ml";
(  12, -8 )*{} = "mm";
(  15, -8 )*{} = "mr";
{ \ar@{-} "r"; "mr" };
{ \ar@{-} "r"; "mm" };
{ \ar@{-} "r"; "ml" };
(  9, -12 )*{} = "mll";
(  12, -12 )*{} = "mlm";
(  15, -12 )*{} = "mlr";
{ \ar@{-} "mm"; "mll" };
{ \ar@{-} "mm"; "mlm" };
{ \ar@{-} "mm"; "mlr" };
\end{xy}
}
\bigplus
\adjustbox{valign=m}{
\begin{xy}
(  3,   0 )*{} = "root";
(  0,  -4 )*{} = "l";
(  3,  -4 )*{} = "m";
( 12,  -4 )*{} = "r";
{ \ar@{-} "root"; "l" };
{ \ar@{-} "root"; "m" };
{ \ar@{-} "root"; "r" };
( 0, -8 )*{} = "ml";
(  3, -8 )*{} = "mm";
(  6, -8 )*{} = "mr";
{ \ar@{-} "m"; "ml" };
{ \ar@{-} "m"; "mm" };
{ \ar@{-} "m"; "mr" };
( 9, -8 )*{} = "ml";
(  12, -8 )*{} = "mm";
(  15, -8 )*{} = "mr";
{ \ar@{-} "r"; "mr" };
{ \ar@{-} "r"; "mm" };
{ \ar@{-} "r"; "ml" };
(  12, -12 )*{} = "mll";
(  15, -12 )*{} = "mlm";
(  18, -12 )*{} = "mlr";
{ \ar@{-} "mr"; "mll" };
{ \ar@{-} "mr"; "mlm" };
{ \ar@{-} "mr"; "mlr" };
\end{xy}
}
$}
&\qquad
&
\adjustbox{valign=t}{$
\theta
=
\adjustbox{valign=m}{
\begin{xy}
(  0,   0 )*{\bullet} = "root";
( -3,  -4 )*{} = "l";
(  0,  -4 )*{} = "m";
(  3,  -4 )*{} = "r";
{ \ar@{-} "root"; "l" };
{ \ar@{-} "root"; "m" };
{ \ar@{-} "root"; "r" };
(  0,  -8 )*{} = "rl";
(  3,  -8 )*{} = "rm";
( 12,  -8 )*{} = "rr";
{ \ar@{-} "r"; "rl" };
{ \ar@{-} "r"; "rm" };
{ \ar@{-} "r"; "rr" };
(  0, -12 )*{} = "rml";
(  3, -12 )*{} = "rmm";
(  6, -12 )*{} = "rmr";
{ \ar@{-} "rm"; "rml" };
{ \ar@{-} "rm"; "rmm" };
{ \ar@{-} "rm"; "rmr" };
(  9, -12 )*{} = "rrl";
( 12, -12 )*{} = "rrm";
( 15, -12 )*{} = "rrr";
{ \ar@{-} "rr"; "rrl" };
{ \ar@{-} "rr"; "rrm" };
{ \ar@{-} "rr"; "rrr" };
\end{xy}
}
$}
\qquad
\adjustbox{valign=t}{$
\nu
=
\adjustbox{valign=m}{
\begin{xy}
(  3,   0 )*{\bullet} = "root";
(  0,  -4 )*{} = "l";
(  3,  -4 )*{} = "m";
(  6,  -4 )*{} = "r";
{ \ar@{-} "root"; "l" };
{ \ar@{-} "root"; "m" };
{ \ar@{-} "root"; "r" };
( 3, -8 )*{} = "rl";
(  6, -8 )*{} = "rm";
(  9, -8 )*{} = "rr";
{ \ar@{-} "r"; "rl" };
{ \ar@{-} "r"; "rm" };
{ \ar@{-} "r"; "rr" };
( 6, -12 )*{} = "rrl";
(  9, -12 )*{} = "rrm";
(  12, -12 )*{} = "rrr";
{ \ar@{-} "rr"; "rrr" };
{ \ar@{-} "rr"; "rrm" };
{ \ar@{-} "rr"; "rrl" };
(  9, -16 )*{} = "rrml";
(  12, -16 )*{} = "rrmm";
(  15, -16 )*{} = "rrmr";
{ \ar@{-} "rrr"; "rrml" };
{ \ar@{-} "rrr"; "rrmm" };
{ \ar@{-} "rrr"; "rrmr" };
\end{xy}
}
$}
\end{alignat*}
\end{theorem}

\begin{proof}
The proof consists of Lemmas \ref{degree7} to \ref{degree9finished}.
\end{proof}

\begin{remark}
As nonassociative polynomials, the relations of Theorem \ref{eventheorem} are:
\begin{align*}
&
(({\ast}{\ast}{\ast}){\ast}{\ast}) +
({\ast}({\ast}{\ast}{\ast}){\ast}) +
({\ast}{\ast}({\ast}{\ast}{\ast})),
\;
({\ast}({\ast}{\ast}({\ast}{\ast}{\ast})){\ast}) +
({\ast}{\ast}({\ast}({\ast}{\ast}{\ast}){\ast})) +
({\ast}{\ast}({\ast}{\ast}({\ast}{\ast}{\ast}))),
\\
&
({\ast}({\ast}{\ast}{\ast})({\ast}({\ast}{\ast}{\ast}){\ast})) +
({\ast}({\ast}{\ast}{\ast})({\ast}{\ast}({\ast}{\ast}{\ast}))),
\;
({\ast}{\ast}({\ast}({\ast}{\ast}{\ast})({\ast}{\ast}{\ast}))),
\;
({\ast}{\ast}({\ast}{\ast}({\ast}{\ast}({\ast}{\ast}{\ast})))).
\end{align*}
\end{remark}

\begin{lemma}
\label{degree7}
There is only one SCM of $\ell m(\alpha)$ with itself;
this produces reduced S-polynomial $\beta$, and the set $\{ \alpha, \beta \}$ is self-reduced:
\[
\beta
=
\adjustbox{valign=m}{
\begin{xy}
(  0,  0 )*{\bullet} = "root";
( -3, -4 )*{} = "l";
(  0, -4 )*{} = "m";
(  3, -4 )*{} = "r";
{ \ar@{-} "root"; "l" };
{ \ar@{-} "root"; "m" };
{ \ar@{-} "root"; "r" };
( -3, -8 )*{} = "ml";
(  0, -8 )*{} = "mm";
(  3, -8 )*{} = "mr";
{ \ar@{-} "m"; "ml" };
{ \ar@{-} "m"; "mm" };
{ \ar@{-} "m"; "mr" };
(  0, -12 )*{} = "mrl";
(  3, -12 )*{} = "mrm";
(  6, -12 )*{} = "mrr";
{ \ar@{-} "mr"; "mrl" };
{ \ar@{-} "mr"; "mrm" };
{ \ar@{-} "mr"; "mrr" };
\end{xy}
}
\bigplus
\adjustbox{valign=m}{
\begin{xy}
(  0,  0 )*{} = "root";
( -3, -4 )*{} = "l";
(  0, -4 )*{} = "m";
(  3, -4 )*{} = "r";
{ \ar@{-} "root"; "l" };
{ \ar@{-} "root"; "m" };
{ \ar@{-} "root"; "r" };
(  0, -8 )*{} = "rl";
(  3, -8 )*{} = "rm";
(  6, -8 )*{} = "rr";
{ \ar@{-} "r"; "rl" };
{ \ar@{-} "r"; "rm" };
{ \ar@{-} "r"; "rr" };
(  0, -12 )*{} = "rml";
(  3, -12 )*{} = "rmm";
(  6, -12 )*{} = "rmr";
{ \ar@{-} "rm"; "rml" };
{ \ar@{-} "rm"; "rmm" };
{ \ar@{-} "rm"; "rmr" };
\end{xy}
}
\bigplus
\adjustbox{valign=m}{
\begin{xy}
(  0,  0 )*{} = "root";
( -3, -4 )*{} = "l";
(  0, -4 )*{} = "m";
(  3, -4 )*{} = "r";
{ \ar@{-} "root"; "l" };
{ \ar@{-} "root"; "m" };
{ \ar@{-} "root"; "r" };
(  0, -8 )*{} = "rl";
(  3, -8 )*{} = "rm";
(  6, -8 )*{} = "rr";
{ \ar@{-} "r"; "rl" };
{ \ar@{-} "r"; "rm" };
{ \ar@{-} "r"; "rr" };
(  3, -12 )*{} = "rrl";
(  6, -12 )*{} = "rrm";
(  9, -12 )*{} = "rrr";
{ \ar@{-} "rr"; "rrl" };
{ \ar@{-} "rr"; "rrm" };
{ \ar@{-} "rr"; "rrr" };
\end{xy}
}
=
t \circ_2 ( t \circ_3 t )
+ t \circ_3 ( t \circ_2 t )
+ t \circ_3 ( t \circ_3 t ).
\]
\end{lemma}

\begin{proof}
We have $\ell m(\alpha) = t \circ_1 t$ and hence
\[
\ell m(\alpha) \circ_1 t
=
\adjustbox{valign=m}{
\begin{xy}
(  0,  0 )*{} = "root";
( -3, -4 )*{} = "l";
(  0, -4 )*{} = "m";
(  3, -4 )*{} = "r";
{ \ar@{-} "root"; "l" };
{ \ar@{-} "root"; "m" };
{ \ar@{-} "root"; "r" };
( -6, -8 )*{} = "ll";
( -3, -8 )*{} = "lm";
(  0, -8 )*{} = "lr";
{ \ar@{-} "l"; "ll" };
{ \ar@{-} "l"; "lm" };
{ \ar@{-} "l"; "lr" };
( -9, -12 )*{} = "lll";
( -6, -12 )*{} = "llm";
( -3, -12 )*{} = "llr";
{ \ar@{-} "ll"; "lll" };
{ \ar@{-} "ll"; "llm" };
{ \ar@{-} "ll"; "llr" };
\end{xy}
}
=
t \circ_1 \ell m(\alpha).
\]
From this we obtain these tree polynomials (Definition \ref{Mpqf}):
\begin{align*}
\alpha \circ_1 t
&=
( t \circ_1 t ) \circ_1 t + ( t \circ_2 t ) \circ_1 t + ( t \circ_3 t ) \circ_1 t
=
\!\!\!\!\!\!\!\!
\adjustbox{valign=m}{
\begin{xy}
(  0,  0 )*{\bullet} = "root";
( -3, -4 )*{} = "l";
(  0, -4 )*{} = "m";
(  3, -4 )*{} = "r";
{ \ar@{-} "root"; "l" };
{ \ar@{-} "root"; "m" };
{ \ar@{-} "root"; "r" };
( -6, -8 )*{} = "ll";
( -3, -8 )*{} = "lm";
(  0, -8 )*{} = "lr";
{ \ar@{-} "l"; "ll" };
{ \ar@{-} "l"; "lm" };
{ \ar@{-} "l"; "lr" };
( -9, -12 )*{} = "lll";
( -6, -12 )*{} = "llm";
( -3, -12 )*{} = "llr";
{ \ar@{-} "ll"; "lll" };
{ \ar@{-} "ll"; "llm" };
{ \ar@{-} "ll"; "llr" };
\end{xy}
}
\bigplus
\!\!
\adjustbox{valign=m}{
\begin{xy}
(  0,  0 )*{} = "root";
( -9, -4 )*{} = "l";
(  0, -4 )*{} = "m";
(  3, -4 )*{} = "r";
{ \ar@{-} "root"; "l" };
{ \ar@{-} "root"; "m" };
{ \ar@{-} "root"; "r" };
( -12, -8 )*{} = "ll";
(  -9, -8 )*{} = "lm";
(  -6, -8 )*{} = "lr";
{ \ar@{-} "l"; "ll" };
{ \ar@{-} "l"; "lm" };
{ \ar@{-} "l"; "lr" };
( -3, -8 )*{} = "ml";
(  0, -8 )*{} = "mm";
(  3, -8 )*{} = "mr";
{ \ar@{-} "m"; "ml" };
{ \ar@{-} "m"; "mm" };
{ \ar@{-} "m"; "mr" };
\end{xy}
}
\bigplus
\!\!
\adjustbox{valign=m}{
\begin{xy}
(  0,  0 )*{} = "root";
( -4.5, -4 )*{} = "l";
(  0, -4 )*{} = "m";
(  4.5, -4 )*{} = "r";
{ \ar@{-} "root"; "l" };
{ \ar@{-} "root"; "m" };
{ \ar@{-} "root"; "r" };
( -7.5, -8 )*{} = "ll";
( -4.5, -8 )*{} = "lm";
( -1.5, -8 )*{} = "lr";
{ \ar@{-} "l"; "ll" };
{ \ar@{-} "l"; "lm" };
{ \ar@{-} "l"; "lr" };
( 1.5, -8 )*{} = "rl";
( 4.5, -8 )*{} = "rm";
( 7.5, -8 )*{} = "rr";
{ \ar@{-} "r"; "rl" };
{ \ar@{-} "r"; "rm" };
{ \ar@{-} "r"; "rr" };
\end{xy}
}
\\[1mm]
t \circ_1 \alpha
&=
t \circ_1 ( t \circ_1 t ) + t \circ_1 ( t \circ_2 t ) + t \circ_1 ( t \circ_3 t )
=
\!\!\!\!\!\!\!\!
\adjustbox{valign=m}{
\begin{xy}
(  0,  0 )*{\bullet} = "root";
( -3, -4 )*{} = "l";
(  0, -4 )*{} = "m";
(  3, -4 )*{} = "r";
{ \ar@{-} "root"; "l" };
{ \ar@{-} "root"; "m" };
{ \ar@{-} "root"; "r" };
( -6, -8 )*{} = "ll";
( -3, -8 )*{} = "lm";
(  0, -8 )*{} = "lr";
{ \ar@{-} "l"; "ll" };
{ \ar@{-} "l"; "lm" };
{ \ar@{-} "l"; "lr" };
( -9, -12 )*{} = "lll";
( -6, -12 )*{} = "llm";
( -3, -12 )*{} = "llr";
{ \ar@{-} "ll"; "lll" };
{ \ar@{-} "ll"; "llm" };
{ \ar@{-} "ll"; "llr" };
\end{xy}
}
\bigplus
\adjustbox{valign=m}{
\begin{xy}
(  0,  0 )*{} = "root";
( -3, -4 )*{} = "l";
(  0, -4 )*{} = "m";
(  3, -4 )*{} = "r";
{ \ar@{-} "root"; "l" };
{ \ar@{-} "root"; "m" };
{ \ar@{-} "root"; "r" };
( -6, -8 )*{} = "ll";
( -3, -8 )*{} = "lm";
(  0, -8 )*{} = "lr";
{ \ar@{-} "l"; "ll" };
{ \ar@{-} "l"; "lm" };
{ \ar@{-} "l"; "lr" };
( -6, -12 )*{} = "lml";
( -3, -12 )*{} = "lmm";
(  0, -12 )*{} = "lmr";
{ \ar@{-} "lm"; "lml" };
{ \ar@{-} "lm"; "lmm" };
{ \ar@{-} "lm"; "lmr" };
\end{xy}
}
\bigplus
\adjustbox{valign=m}{
\begin{xy}
(  0,  0 )*{} = "root";
( -3, -4 )*{} = "l";
(  0, -4 )*{} = "m";
(  3, -4 )*{} = "r";
{ \ar@{-} "root"; "l" };
{ \ar@{-} "root"; "m" };
{ \ar@{-} "root"; "r" };
( -6, -8 )*{} = "ll";
( -3, -8 )*{} = "lm";
(  0, -8 )*{} = "lr";
{ \ar@{-} "l"; "ll" };
{ \ar@{-} "l"; "lm" };
{ \ar@{-} "l"; "lr" };
( -3, -12 )*{} = "lrl";
(  0, -12 )*{} = "lrm";
(  3, -12 )*{} = "lrr";
{ \ar@{-} "lr"; "lrl" };
{ \ar@{-} "lr"; "lrm" };
{ \ar@{-} "lr"; "lrr" };
\end{xy}
}
\end{align*}
The difference is this (non-reduced) S-polynomial:
\begin{align*}
\alpha \circ_1 t - t \circ_1 \alpha
&=
( t \circ_2 t ) \circ_1 t + ( t \circ_3 t ) \circ_1 t
- t \circ_1 ( t \circ_2 t ) - t \circ_1 ( t \circ_3 t )
\\
&= \;
\adjustbox{valign=m}{
\begin{xy}
(  0,  0 )*{} = "root";
( -9, -4 )*{} = "l";
(  0, -4 )*{} = "m";
(  6, -4 )*{} = "r";
{ \ar@{-} "root"; "l" };
{ \ar@{-} "root"; "m" };
{ \ar@{-} "root"; "r" };
( -12, -8 )*{} = "ll";
(  -9, -8 )*{} = "lm";
(  -6, -8 )*{} = "lr";
{ \ar@{-} "l"; "ll" };
{ \ar@{-} "l"; "lm" };
{ \ar@{-} "l"; "lr" };
( -3, -8 )*{} = "ml";
(  0, -8 )*{} = "mm";
(  3, -8 )*{} = "mr";
{ \ar@{-} "m"; "ml" };
{ \ar@{-} "m"; "mm" };
{ \ar@{-} "m"; "mr" };
\end{xy}
}
\bigplus
\adjustbox{valign=m}{
\begin{xy}
(  0,  0 )*{} = "root";
( -6, -4 )*{} = "l";
(  0, -4 )*{} = "m";
(  6, -4 )*{} = "r";
{ \ar@{-} "root"; "l" };
{ \ar@{-} "root"; "m" };
{ \ar@{-} "root"; "r" };
( -9, -8 )*{} = "ll";
( -6, -8 )*{} = "lm";
( -3, -8 )*{} = "lr";
{ \ar@{-} "l"; "ll" };
{ \ar@{-} "l"; "lm" };
{ \ar@{-} "l"; "lr" };
(  3, -8 )*{} = "rl";
(  6, -8 )*{} = "rm";
(  9, -8 )*{} = "rr";
{ \ar@{-} "r"; "rl" };
{ \ar@{-} "r"; "rm" };
{ \ar@{-} "r"; "rr" };
\end{xy}
}
\bigminus
\adjustbox{valign=m}{
\begin{xy}
(  0,  0 )*{\bullet} = "root";
( -3, -4 )*{} = "l";
(  0, -4 )*{} = "m";
(  3, -4 )*{} = "r";
{ \ar@{-} "root"; "l" };
{ \ar@{-} "root"; "m" };
{ \ar@{-} "root"; "r" };
( -6, -8 )*{} = "ll";
( -3, -8 )*{} = "lm";
(  0, -8 )*{} = "lr";
{ \ar@{-} "l"; "ll" };
{ \ar@{-} "l"; "lm" };
{ \ar@{-} "l"; "lr" };
( -6, -12 )*{} = "lml";
( -3, -12 )*{} = "lmm";
(  0, -12 )*{} = "lmr";
{ \ar@{-} "lm"; "lml" };
{ \ar@{-} "lm"; "lmm" };
{ \ar@{-} "lm"; "lmr" };
\end{xy}
}
\bigminus
\adjustbox{valign=m}{
\begin{xy}
(  0,  0 )*{} = "root";
( -3, -4 )*{} = "l";
(  0, -4 )*{} = "m";
(  3, -4 )*{} = "r";
{ \ar@{-} "root"; "l" };
{ \ar@{-} "root"; "m" };
{ \ar@{-} "root"; "r" };
( -6, -8 )*{} = "ll";
( -3, -8 )*{} = "lm";
(  0, -8 )*{} = "lr";
{ \ar@{-} "l"; "ll" };
{ \ar@{-} "l"; "lm" };
{ \ar@{-} "l"; "lr" };
( -3, -12 )*{} = "lrl";
(  0, -12 )*{} = "lrm";
(  3, -12 )*{} = "lrr";
{ \ar@{-} "lr"; "lrl" };
{ \ar@{-} "lr"; "lrm" };
{ \ar@{-} "lr"; "lrr" };
\end{xy}
}
\\
&=
( t \circ_1 t ) \circ_4 t
+
( t \circ_1 t ) \circ_5 t
-
( t \circ_1 t ) \circ_2 t
-
( t \circ_1 t ) \circ_3 t.
\end{align*}
We have rewritten the partial compositions (Lemma \ref{pcrules}).
We apply rewrite rule \eqref{frr} to the top subtree $\ell m(\alpha) = t \circ_1 t$
of each monomial (reduce using $\alpha$):
\[
\begin{array}{l}
{}
- ( t \circ_2 t ) \circ_4 t
- ( t \circ_3 t ) \circ_4 t
- ( t \circ_2 t ) \circ_5 t
- ( t \circ_3 t ) \circ_5 t
\\[1mm]
{}
+ ( t \circ_2 t ) \circ_2 t
+ ( t \circ_3 t ) \circ_2 t
+ ( t \circ_2 t ) \circ_3 t
+ ( t \circ_3 t ) \circ_3 t.
\end{array}
\]
Terms 3 and 6 cancel since both monomials represent the same tree:
\[
( t \circ_2 t ) \circ_5 t
\; = \;
( t \circ_3 t ) \circ_2 t
\; = \;
\adjustbox{valign=m}{
\begin{xy}
(  0,  0 )*{} = "root";
( -6, -4 )*{} = "l";
(  0, -4 )*{} = "m";
(  9, -4 )*{} = "r";
{ \ar@{-} "root"; "l" };
{ \ar@{-} "root"; "m" };
{ \ar@{-} "root"; "r" };
( -3, -8 )*{} = "ml";
(  0, -8 )*{} = "mm";
(  3, -8 )*{} = "mr";
{ \ar@{-} "m"; "ml" };
{ \ar@{-} "m"; "mm" };
{ \ar@{-} "m"; "mr" };
(  6, -8 )*{} = "rl";
(  9, -8 )*{} = "rm";
( 12, -8 )*{} = "rr";
{ \ar@{-} "r"; "rl" };
{ \ar@{-} "r"; "rm" };
{ \ar@{-} "r"; "rr" };
\end{xy}
}
\]
Six terms remain:
\begin{align*}
&
{}
- ( t \circ_2 t ) \circ_4 t
- ( t \circ_3 t ) \circ_4 t
- ( t \circ_3 t ) \circ_5 t
+ ( t \circ_2 t ) \circ_2 t
+ ( t \circ_2 t ) \circ_3 t
+ ( t \circ_3 t ) \circ_3 t
\\
&= \,
\bigminus
\adjustbox{valign=m}{
\begin{xy}
(  0,  0 )*{} = "root";
( -3, -4 )*{} = "l";
(  0, -4 )*{} = "m";
(  3, -4 )*{} = "r";
{ \ar@{-} "root"; "l" };
{ \ar@{-} "root"; "m" };
{ \ar@{-} "root"; "r" };
( -3, -8 )*{} = "ml";
(  0, -8 )*{} = "mm";
(  3, -8 )*{} = "mr";
{ \ar@{-} "m"; "ml" };
{ \ar@{-} "m"; "mm" };
{ \ar@{-} "m"; "mr" };
(  0, -12 )*{} = "mrl";
(  3, -12 )*{} = "mrm";
(  6, -12 )*{} = "mrr";
{ \ar@{-} "mr"; "mrl" };
{ \ar@{-} "mr"; "mrm" };
{ \ar@{-} "mr"; "mrr" };
\end{xy}
}
\bigminus
\adjustbox{valign=m}{
\begin{xy}
(  0,  0 )*{} = "root";
( -3, -4 )*{} = "l";
(  0, -4 )*{} = "m";
(  3, -4 )*{} = "r";
{ \ar@{-} "root"; "l" };
{ \ar@{-} "root"; "m" };
{ \ar@{-} "root"; "r" };
(  0, -8 )*{} = "rl";
(  3, -8 )*{} = "rm";
(  6, -8 )*{} = "rr";
{ \ar@{-} "r"; "rl" };
{ \ar@{-} "r"; "rm" };
{ \ar@{-} "r"; "rr" };
(  0, -12 )*{} = "rml";
(  3, -12 )*{} = "rmm";
(  6, -12 )*{} = "rmr";
{ \ar@{-} "rm"; "rml" };
{ \ar@{-} "rm"; "rmm" };
{ \ar@{-} "rm"; "rmr" };
\end{xy}
}
\bigminus
\adjustbox{valign=m}{
\begin{xy}
(  0,  0 )*{} = "root";
( -3, -4 )*{} = "l";
(  0, -4 )*{} = "m";
(  3, -4 )*{} = "r";
{ \ar@{-} "root"; "l" };
{ \ar@{-} "root"; "m" };
{ \ar@{-} "root"; "r" };
(  0, -8 )*{} = "rl";
(  3, -8 )*{} = "rm";
(  6, -8 )*{} = "rr";
{ \ar@{-} "r"; "rl" };
{ \ar@{-} "r"; "rm" };
{ \ar@{-} "r"; "rr" };
(  3, -12 )*{} = "rrl";
(  6, -12 )*{} = "rrm";
(  9, -12 )*{} = "rrr";
{ \ar@{-} "rr"; "rrl" };
{ \ar@{-} "rr"; "rrm" };
{ \ar@{-} "rr"; "rrr" };
\end{xy}
}
\bigplus
\adjustbox{valign=m}{
\begin{xy}
(  0,  0 )*{\bullet} = "root";
( -3, -4 )*{} = "l";
(  0, -4 )*{} = "m";
(  3, -4 )*{} = "r";
{ \ar@{-} "root"; "l" };
{ \ar@{-} "root"; "m" };
{ \ar@{-} "root"; "r" };
( -3, -8 )*{} = "ml";
(  0, -8 )*{} = "mm";
(  3, -8 )*{} = "mr";
{ \ar@{-} "m"; "ml" };
{ \ar@{-} "m"; "mm" };
{ \ar@{-} "m"; "mr" };
( -6, -12 )*{} = "mll";
( -3, -12 )*{} = "mlm";
(  0, -12 )*{} = "mlr";
{ \ar@{-} "ml"; "mll" };
{ \ar@{-} "ml"; "mlm" };
{ \ar@{-} "ml"; "mlr" };
\end{xy}
}
\bigplus
\adjustbox{valign=m}{
\begin{xy}
(  0,  0 )*{} = "root";
( -3, -4 )*{} = "l";
(  0, -4 )*{} = "m";
(  3, -4 )*{} = "r";
{ \ar@{-} "root"; "l" };
{ \ar@{-} "root"; "m" };
{ \ar@{-} "root"; "r" };
( -3, -8 )*{} = "ml";
(  0, -8 )*{} = "mm";
(  3, -8 )*{} = "mr";
{ \ar@{-} "m"; "ml" };
{ \ar@{-} "m"; "mm" };
{ \ar@{-} "m"; "mr" };
( -3, -12 )*{} = "mml";
(  0, -12 )*{} = "mmm";
(  3, -12 )*{} = "mmr";
{ \ar@{-} "mm"; "mml" };
{ \ar@{-} "mm"; "mmm" };
{ \ar@{-} "mm"; "mmr" };
\end{xy}
}
\bigplus
\adjustbox{valign=m}{
\begin{xy}
(  0,  0 )*{} = "root";
( -3, -4 )*{} = "l";
(  0, -4 )*{} = "m";
(  3, -4 )*{} = "r";
{ \ar@{-} "root"; "l" };
{ \ar@{-} "root"; "m" };
{ \ar@{-} "root"; "r" };
(  0, -8 )*{} = "rl";
(  3, -8 )*{} = "rm";
(  6, -8 )*{} = "rr";
{ \ar@{-} "r"; "rl" };
{ \ar@{-} "r"; "rm" };
{ \ar@{-} "r"; "rr" };
( -3, -12 )*{} = "rll";
(  0, -12 )*{} = "rlm";
(  3, -12 )*{} = "rlr";
{ \ar@{-} "rl"; "rll" };
{ \ar@{-} "rl"; "rlm" };
{ \ar@{-} "rl"; "rlr" };
\end{xy}
}
=
\\
&
{}
- t \circ_2 ( t \circ_3 t )
- t \circ_3 ( t \circ_2 t )
- t \circ_3 ( t \circ_3 t )
+ t \circ_2 ( t \circ_1 t )
+ t \circ_2 ( t \circ_2 t )
+ t \circ_3 ( t \circ_1 t ).
\end{align*}
In terms 4 and 6, we reduce the bottom subtree $\ell m(\alpha) = t \circ_1 t$ using $\alpha$:
\[
\begin{array}{l}
{}
- t \circ_2 ( t \circ_3 t )
- t \circ_3 ( t \circ_2 t )
- t \circ_3 ( t \circ_3 t )
- t \circ_2 ( t \circ_2 t )
\\[1mm]
{}
- t \circ_2 ( t \circ_3 t )
+ t \circ_2 ( t \circ_2 t )
- t \circ_3 ( t \circ_2 t )
- t \circ_3 ( t \circ_3 t ).
\end{array}
\]
Terms 4 and 6 cancel and the others combine in pairs:
\[
-2 \big[
  t \circ_2 ( t \circ_3 t )
+ t \circ_3 ( t \circ_2 t )
+ t \circ_3 ( t \circ_3 t )
\big]
=
-2 \! \left[
\adjustbox{valign=m}{
\begin{xy}
(  0,  0 )*{\bullet} = "root";
( -3, -4 )*{} = "l";
(  0, -4 )*{} = "m";
(  3, -4 )*{} = "r";
{ \ar@{-} "root"; "l" };
{ \ar@{-} "root"; "m" };
{ \ar@{-} "root"; "r" };
( -3, -8 )*{} = "ml";
(  0, -8 )*{} = "mm";
(  3, -8 )*{} = "mr";
{ \ar@{-} "m"; "ml" };
{ \ar@{-} "m"; "mm" };
{ \ar@{-} "m"; "mr" };
(  0, -12 )*{} = "mrl";
(  3, -12 )*{} = "mrm";
(  6, -12 )*{} = "mrr";
{ \ar@{-} "mr"; "mrl" };
{ \ar@{-} "mr"; "mrm" };
{ \ar@{-} "mr"; "mrr" };
\end{xy}
}
\bigplus
\adjustbox{valign=m}{
\begin{xy}
(  0,  0 )*{} = "root";
( -3, -4 )*{} = "l";
(  0, -4 )*{} = "m";
(  3, -4 )*{} = "r";
{ \ar@{-} "root"; "l" };
{ \ar@{-} "root"; "m" };
{ \ar@{-} "root"; "r" };
(  0, -8 )*{} = "rl";
(  3, -8 )*{} = "rm";
(  6, -8 )*{} = "rr";
{ \ar@{-} "r"; "rl" };
{ \ar@{-} "r"; "rm" };
{ \ar@{-} "r"; "rr" };
(  0, -12 )*{} = "rml";
(  3, -12 )*{} = "rmm";
(  6, -12 )*{} = "rmr";
{ \ar@{-} "rm"; "rml" };
{ \ar@{-} "rm"; "rmm" };
{ \ar@{-} "rm"; "rmr" };
\end{xy}
}
\bigplus
\adjustbox{valign=m}{
\begin{xy}
(  0,  0 )*{} = "root";
( -3, -4 )*{} = "l";
(  0, -4 )*{} = "m";
(  3, -4 )*{} = "r";
{ \ar@{-} "root"; "l" };
{ \ar@{-} "root"; "m" };
{ \ar@{-} "root"; "r" };
(  0, -8 )*{} = "rl";
(  3, -8 )*{} = "rm";
(  6, -8 )*{} = "rr";
{ \ar@{-} "r"; "rl" };
{ \ar@{-} "r"; "rm" };
{ \ar@{-} "r"; "rr" };
(  3, -12 )*{} = "rrl";
(  6, -12 )*{} = "rrm";
(  9, -12 )*{} = "rrr";
{ \ar@{-} "rr"; "rrl" };
{ \ar@{-} "rr"; "rrm" };
{ \ar@{-} "rr"; "rrr" };
\end{xy}
}
\right]
\]
No further reduction is possible.
The monic form of the last polynomial is $\beta$.
\end{proof}

The relation $\beta$ corresponds to this rewrite rule:
\begin{equation}
\label{mybeta2}
t \circ_2 ( t \circ_3 t )
=
\adjustbox{valign=m}{
\begin{xy}
(  0,  0 )*{\bullet} = "root";
( -3, -4 )*{} = "l";
(  0, -4 )*{} = "m";
(  3, -4 )*{} = "r";
{ \ar@{-} "root"; "l" };
{ \ar@{-} "root"; "m" };
{ \ar@{-} "root"; "r" };
( -3, -8 )*{} = "ml";
(  0, -8 )*{} = "mm";
(  3, -8 )*{} = "mr";
{ \ar@{-} "m"; "ml" };
{ \ar@{-} "m"; "mm" };
{ \ar@{-} "m"; "mr" };
(  0, -12 )*{} = "mrl";
(  3, -12 )*{} = "mrm";
(  6, -12 )*{} = "mrr";
{ \ar@{-} "mr"; "mrl" };
{ \ar@{-} "mr"; "mrm" };
{ \ar@{-} "mr"; "mrr" };
\end{xy}
}
\!\!\!
\xrightarrow{\quad}
\;
\bigminus
\adjustbox{valign=m}{
\begin{xy}
(  0,  0 )*{} = "root";
( -3, -4 )*{} = "l";
(  0, -4 )*{} = "m";
(  3, -4 )*{} = "r";
{ \ar@{-} "root"; "l" };
{ \ar@{-} "root"; "m" };
{ \ar@{-} "root"; "r" };
(  0, -8 )*{} = "rl";
(  3, -8 )*{} = "rm";
(  6, -8 )*{} = "rr";
{ \ar@{-} "r"; "rl" };
{ \ar@{-} "r"; "rm" };
{ \ar@{-} "r"; "rr" };
(  0, -12 )*{} = "rml";
(  3, -12 )*{} = "rmm";
(  6, -12 )*{} = "rmr";
{ \ar@{-} "rm"; "rml" };
{ \ar@{-} "rm"; "rmm" };
{ \ar@{-} "rm"; "rmr" };
\end{xy}
}
\bigminus
\adjustbox{valign=m}{
\begin{xy}
(  0,  0 )*{} = "root";
( -3, -4 )*{} = "l";
(  0, -4 )*{} = "m";
(  3, -4 )*{} = "r";
{ \ar@{-} "root"; "l" };
{ \ar@{-} "root"; "m" };
{ \ar@{-} "root"; "r" };
(  0, -8 )*{} = "rl";
(  3, -8 )*{} = "rm";
(  6, -8 )*{} = "rr";
{ \ar@{-} "r"; "rl" };
{ \ar@{-} "r"; "rm" };
{ \ar@{-} "r"; "rr" };
(  3, -12 )*{} = "rrl";
(  6, -12 )*{} = "rrm";
(  9, -12 )*{} = "rrr";
{ \ar@{-} "rr"; "rrl" };
{ \ar@{-} "rr"; "rrm" };
{ \ar@{-} "rr"; "rrr" };
\end{xy}
}
\!\!\!
= {}
- t \circ_3 ( t \circ_2 t )
- t \circ_3 ( t \circ_3 t )
\end{equation}

We consider separately the four SCMs of
$\ell m(\alpha) = t \circ_1 t$ and $\ell m(\beta) = t \circ_2 ( t \circ_3 t )$.

\begin{lemma}
\label{degree9case1}
Identifying the second $t$ of $\ell m(\alpha) = t \circ_1 t$
with the first $t$ of $\ell m(\beta) = t \circ_2 (t  \circ_3  t)$
produces the reduced S-polynomial $\gamma$,
and $\{ \alpha, \beta, \gamma \}$ is self-reduced:
\[
\gamma
=
\scalebox{1.5}{\emph{2}}
\adjustbox{valign=m}{
\begin{xy}
(  0,   0 )*{\bullet} = "root";
( -3,  -4 )*{} = "l";
(  0,  -4 )*{} = "m";
(  3,  -4 )*{} = "r";
{ \ar@{-} "root"; "l" };
{ \ar@{-} "root"; "m" };
{ \ar@{-} "root"; "r" };
(  0,  -8 )*{} = "rl";
(  3,  -8 )*{} = "rm";
( 12,  -8 )*{} = "rr";
{ \ar@{-} "r"; "rl" };
{ \ar@{-} "r"; "rm" };
{ \ar@{-} "r"; "rr" };
(  0, -12 )*{} = "rml";
(  3, -12 )*{} = "rmm";
(  6, -12 )*{} = "rmr";
{ \ar@{-} "rm"; "rml" };
{ \ar@{-} "rm"; "rmm" };
{ \ar@{-} "rm"; "rmr" };
(  9, -12 )*{} = "rrl";
( 12, -12 )*{} = "rrm";
( 15, -12 )*{} = "rrr";
{ \ar@{-} "rr"; "rrl" };
{ \ar@{-} "rr"; "rrm" };
{ \ar@{-} "rr"; "rrr" };
\end{xy}
}
\bigplus
\adjustbox{valign=m}{
\begin{xy}
(  3,   0 )*{} = "root";
(  0,  -4 )*{} = "l";
(  3,  -4 )*{} = "m";
(  6,  -4 )*{} = "r";
{ \ar@{-} "root"; "l" };
{ \ar@{-} "root"; "m" };
{ \ar@{-} "root"; "r" };
( 3, -8 )*{} = "rl";
(  6, -8 )*{} = "rm";
(  9, -8 )*{} = "rr";
{ \ar@{-} "r"; "rl" };
{ \ar@{-} "r"; "rm" };
{ \ar@{-} "r"; "rr" };
( 6, -12 )*{} = "rrl";
(  9, -12 )*{} = "rrm";
(  12, -12 )*{} = "rrr";
{ \ar@{-} "rr"; "rrr" };
{ \ar@{-} "rr"; "rrm" };
{ \ar@{-} "rr"; "rrl" };
(  9, -16 )*{} = "rrml";
(  12, -16 )*{} = "rrmm";
(  15, -16 )*{} = "rrmr";
{ \ar@{-} "rrr"; "rrml" };
{ \ar@{-} "rrr"; "rrmm" };
{ \ar@{-} "rrr"; "rrmr" };
\end{xy}
}
=
2
( t \circ_3 ( t \circ_2 t ) ) \circ_7 t
+
t \circ_3 ( t \circ_3 ( t \circ_3 t ) ).
\]
\end{lemma}

\begin{proof}
We have the following equations:
\[
\ell m(\alpha) \circ_2 ( t \circ_3 t )
=
( t \circ_1 t ) \circ_2 ( t \circ_3 t )
=
\adjustbox{valign=m}{
\begin{xy}
(  3,   0 )*{} = "root";
(  0,  -4 )*{} = "l";
(  3,  -4 )*{} = "m";
(  6,  -4 )*{} = "r";
{ \ar@{-} "root"; "l" };
{ \ar@{-} "root"; "m" };
{ \ar@{-} "root"; "r" };
( -3, -8 )*{} = "ll";
(  0, -8 )*{} = "lm";
(  3, -8 )*{} = "lr";
{ \ar@{-} "l"; "ll" };
{ \ar@{-} "l"; "lm" };
{ \ar@{-} "l"; "lr" };
( -3, -12 )*{} = "lml";
(  0, -12 )*{} = "lmm";
(  3, -12 )*{} = "lmr";
{ \ar@{-} "lm"; "lml" };
{ \ar@{-} "lm"; "lmm" };
{ \ar@{-} "lm"; "lmr" };
(  0, -16 )*{} = "lmrl";
(  3, -16 )*{} = "lmrm";
(  6, -16 )*{} = "lmrr";
{ \ar@{-} "lmr"; "lmrl" };
{ \ar@{-} "lmr"; "lmrm" };
{ \ar@{-} "lmr"; "lmrr" };
\end{xy}
}
\!\!\!
=
t \circ_1 ( t \circ_2 ( t \circ_3 t ) )
=
t \circ_1 \ell m(\beta).
\]
We apply the same partial compositions to $\alpha$ and $\beta$:
\begin{align*}
\alpha \circ_2 ( t \circ_3 t )
&=
  ( t \circ_1 t ) \circ_2 ( t \circ_3 t )
+ ( t \circ_2 t ) \circ_2 ( t \circ_3 t )
+ ( t \circ_3 t ) \circ_2 ( t \circ_3 t ),
\\
t \circ_1 \beta
&=
  t \circ_1 ( t \circ_2 ( t \circ_3 t ) )
+ t \circ_1 ( t \circ_3 ( t \circ_2 t ) )
+ t \circ_1 ( t \circ_3 ( t \circ_3 t ) ).
\end{align*}
Taking the difference, we obtain this (non-reduced) S-polynomial:
\begin{align*}
&
( t \circ_1 t ) \circ_2 ( t \circ_3 t )
+
( t \circ_2 t ) \circ_2 ( t \circ_3 t )
+
( t \circ_3 t ) \circ_2 ( t \circ_3 t )
\\
& {}
-
t \circ_1 ( t \circ_2 ( t \circ_3 t ) )
-
t \circ_1 ( t \circ_3 ( t \circ_2 t ) )
-
t \circ_1 ( t \circ_3 ( t \circ_3 t ) ).
\end{align*}
Terms 1 and 4 cancel, leaving
\begin{align*}
&
( t \circ_2 t ) \circ_2 ( t \circ_3 t )
+
( t \circ_3 t ) \circ_2 ( t \circ_3 t )
-
t \circ_1 ( t \circ_3 ( t \circ_2 t ) )
-
t \circ_1 ( t \circ_3 ( t \circ_3 t ) )
\\
&=
\adjustbox{valign=m}{
\begin{xy}
(  3,   0 )*{} = "root";
(  0,  -4 )*{} = "l";
(  3,  -4 )*{} = "m";
(  6,  -4 )*{} = "r";
{ \ar@{-} "root"; "l" };
{ \ar@{-} "root"; "m" };
{ \ar@{-} "root"; "r" };
(  0, -8 )*{} = "ml";
(  3, -8 )*{} = "mm";
(  6, -8 )*{} = "mr";
{ \ar@{-} "m"; "ml" };
{ \ar@{-} "m"; "mm" };
{ \ar@{-} "m"; "mr" };
( -3, -12 )*{} = "mll";
(  0, -12 )*{} = "mlm";
(  3, -12 )*{} = "mlr";
{ \ar@{-} "ml"; "mll" };
{ \ar@{-} "ml"; "mlm" };
{ \ar@{-} "ml"; "mlr" };
(  0, -16 )*{} = "mlrl";
(  3, -16 )*{} = "mlrm";
(  6, -16 )*{} = "mlrr";
{ \ar@{-} "mlr"; "mlrl" };
{ \ar@{-} "mlr"; "mlrm" };
{ \ar@{-} "mlr"; "mlrr" };
\end{xy}
}
\bigplus
\adjustbox{valign=m}{
\begin{xy}
(  0,   0 )*{} = "root";
( -6,  -4 )*{} = "l";
(  0,  -4 )*{} = "m";
(  9,  -4 )*{} = "r";
{ \ar@{-} "root"; "l" };
{ \ar@{-} "root"; "m" };
{ \ar@{-} "root"; "r" };
( -3, -8 )*{} = "ml";
(  0, -8 )*{} = "mm";
(  3, -8 )*{} = "mr";
{ \ar@{-} "m"; "ml" };
{ \ar@{-} "m"; "mm" };
{ \ar@{-} "m"; "mr" };
(  6, -8 )*{} = "rl";
(  9, -8 )*{} = "rm";
( 12, -8 )*{} = "rr";
{ \ar@{-} "r"; "rl" };
{ \ar@{-} "r"; "rm" };
{ \ar@{-} "r"; "rr" };
(  0, -12 )*{} = "mrl";
(  3, -12 )*{} = "mrm";
(  6, -12 )*{} = "mrr";
{ \ar@{-} "mr"; "mrl" };
{ \ar@{-} "mr"; "mrm" };
{ \ar@{-} "mr"; "mrr" };
\end{xy}
}
\bigminus
\adjustbox{valign=m}{
\begin{xy}
(  3,   0 )*{\bullet} = "root";
(  0,  -4 )*{} = "l";
(  3,  -4 )*{} = "m";
(  6,  -4 )*{} = "r";
{ \ar@{-} "root"; "l" };
{ \ar@{-} "root"; "m" };
{ \ar@{-} "root"; "r" };
( -3, -8 )*{} = "ll";
(  0, -8 )*{} = "lm";
(  3, -8 )*{} = "lr";
{ \ar@{-} "l"; "ll" };
{ \ar@{-} "l"; "lm" };
{ \ar@{-} "l"; "lr" };
(  0, -12 )*{} = "lrl";
(  3, -12 )*{} = "lrm";
(  6, -12 )*{} = "lrr";
{ \ar@{-} "lr"; "lrl" };
{ \ar@{-} "lr"; "lrm" };
{ \ar@{-} "lr"; "lrr" };
(  0, -16 )*{} = "lrml";
(  3, -16 )*{} = "lrmm";
(  6, -16 )*{} = "lrmr";
{ \ar@{-} "lrm"; "lrml" };
{ \ar@{-} "lrm"; "lrmm" };
{ \ar@{-} "lrm"; "lrmr" };
\end{xy}
}
\bigminus
\adjustbox{valign=m}{
\begin{xy}
(  3,   0 )*{} = "root";
(  0,  -4 )*{} = "l";
(  3,  -4 )*{} = "m";
(  6,  -4 )*{} = "r";
{ \ar@{-} "root"; "l" };
{ \ar@{-} "root"; "m" };
{ \ar@{-} "root"; "r" };
( -3, -8 )*{} = "ll";
(  0, -8 )*{} = "lm";
(  3, -8 )*{} = "lr";
{ \ar@{-} "l"; "ll" };
{ \ar@{-} "l"; "lm" };
{ \ar@{-} "l"; "lr" };
(  0, -12 )*{} = "lrl";
(  3, -12 )*{} = "lrm";
(  6, -12 )*{} = "lrr";
{ \ar@{-} "lr"; "lrl" };
{ \ar@{-} "lr"; "lrm" };
{ \ar@{-} "lr"; "lrr" };
(  3, -16 )*{} = "lrrl";
(  6, -16 )*{} = "lrrm";
(  9, -16 )*{} = "lrrr";
{ \ar@{-} "lrr"; "lrrl" };
{ \ar@{-} "lrr"; "lrrm" };
{ \ar@{-} "lrr"; "lrrr" };
\end{xy}
}
\\
&=
t \circ_2 ( ( t \circ_1 t ) \circ_3 t )
+
( t \circ_3 t ) \circ_2 ( t \circ_3 t )
-
( t \circ_1 t ) \circ_3 ( t \circ_2 t )
-
( t \circ_1 t ) \circ_3 ( t \circ_3 t ).
\end{align*}
Terms 1, 3, 4 contain the subtree $\ell m(\alpha) = t \circ_1 t$,
so we reduce them using $\alpha$:
\begin{align*}
& {}
- t \circ_2 ( ( t \circ_2 t ) \circ_3 t )
- t \circ_2 ( ( t \circ_3 t ) \circ_3 t )
+ ( t \circ_3 t ) \circ_2 ( t \circ_3 t )
+ ( t \circ_2 t ) \circ_3 ( t \circ_2 t )
\\
& {}
+ ( t \circ_3 t ) \circ_3 ( t \circ_2 t )
+ ( t \circ_2 t ) \circ_3 ( t \circ_3 t )
+ ( t \circ_3 t ) \circ_3 ( t \circ_3 t ).
\end{align*}
We write this polynomial in terms of trees:
\begin{align*}
&
\bigminus
\adjustbox{valign=m}{
\begin{xy}
(  3,   0 )*{} = "root";
(  0,  -4 )*{} = "l";
(  3,  -4 )*{} = "m";
(  6,  -4 )*{} = "r";
{ \ar@{-} "root"; "l" };
{ \ar@{-} "root"; "m" };
{ \ar@{-} "root"; "r" };
(  0, -8 )*{} = "ml";
(  3, -8 )*{} = "mm";
(  6, -8 )*{} = "mr";
{ \ar@{-} "m"; "ml" };
{ \ar@{-} "m"; "mm" };
{ \ar@{-} "m"; "mr" };
(  0, -12 )*{} = "mml";
(  3, -12 )*{} = "mmm";
(  6, -12 )*{} = "mmr";
{ \ar@{-} "mm"; "mml" };
{ \ar@{-} "mm"; "mmm" };
{ \ar@{-} "mm"; "mmr" };
(  0, -16 )*{} = "mmml";
(  3, -16 )*{} = "mmmm";
(  6, -16 )*{} = "mmmr";
{ \ar@{-} "mmm"; "mmml" };
{ \ar@{-} "mmm"; "mmmm" };
{ \ar@{-} "mmm"; "mmmr" };
\end{xy}
}
\bigminus
\adjustbox{valign=m}{
\begin{xy}
(  3,   0 )*{} = "root";
(  0,  -4 )*{} = "l";
(  3,  -4 )*{} = "m";
(  6,  -4 )*{} = "r";
{ \ar@{-} "root"; "l" };
{ \ar@{-} "root"; "m" };
{ \ar@{-} "root"; "r" };
(  0, -8 )*{} = "ml";
(  3, -8 )*{} = "mm";
(  6, -8 )*{} = "mr";
{ \ar@{-} "m"; "ml" };
{ \ar@{-} "m"; "mm" };
{ \ar@{-} "m"; "mr" };
(  3, -12 )*{} = "mrl";
(  6, -12 )*{} = "mrm";
(  9, -12 )*{} = "mrr";
{ \ar@{-} "mr"; "mrl" };
{ \ar@{-} "mr"; "mrm" };
{ \ar@{-} "mr"; "mrr" };
(  0, -16 )*{} = "mrll";
(  3, -16 )*{} = "mrlm";
(  6, -16 )*{} = "mrlr";
{ \ar@{-} "mrl"; "mrll" };
{ \ar@{-} "mrl"; "mrlm" };
{ \ar@{-} "mrl"; "mrlr" };
\end{xy}
}
\bigplus
\adjustbox{valign=m}{
\begin{xy}
(  0,   0 )*{} = "root";
( -9,  -4 )*{} = "l";
(  0,  -4 )*{} = "m";
(  9,  -4 )*{} = "r";
{ \ar@{-} "root"; "l" };
{ \ar@{-} "root"; "m" };
{ \ar@{-} "root"; "r" };
( -3, -8 )*{} = "ml";
(  0, -8 )*{} = "mm";
(  3, -8 )*{} = "mr";
{ \ar@{-} "m"; "ml" };
{ \ar@{-} "m"; "mm" };
{ \ar@{-} "m"; "mr" };
(  6, -8 )*{} = "rl";
(  9, -8 )*{} = "rm";
( 12, -8 )*{} = "rr";
{ \ar@{-} "r"; "rl" };
{ \ar@{-} "r"; "rm" };
{ \ar@{-} "r"; "rr" };
(  0, -12 )*{} = "mrl";
(  3, -12 )*{} = "mrm";
(  6, -12 )*{} = "mrr";
{ \ar@{-} "mr"; "mrl" };
{ \ar@{-} "mr"; "mrm" };
{ \ar@{-} "mr"; "mrr" };
\end{xy}
}
\bigplus
\adjustbox{valign=m}{
\begin{xy}
(  3,   0 )*{} = "root";
(  0,  -4 )*{} = "l";
(  3,  -4 )*{} = "m";
(  6,  -4 )*{} = "r";
{ \ar@{-} "root"; "l" };
{ \ar@{-} "root"; "m" };
{ \ar@{-} "root"; "r" };
(  0, -8 )*{} = "ml";
(  3, -8 )*{} = "mm";
(  6, -8 )*{} = "mr";
{ \ar@{-} "m"; "ml" };
{ \ar@{-} "m"; "mm" };
{ \ar@{-} "m"; "mr" };
(  0, -12 )*{} = "mml";
(  3, -12 )*{} = "mmm";
(  6, -12 )*{} = "mmr";
{ \ar@{-} "mm"; "mml" };
{ \ar@{-} "mm"; "mmm" };
{ \ar@{-} "mm"; "mmr" };
(  0, -16 )*{} = "mmml";
(  3, -16 )*{} = "mmmm";
(  6, -16 )*{} = "mmmr";
{ \ar@{-} "mmm"; "mmml" };
{ \ar@{-} "mmm"; "mmmm" };
{ \ar@{-} "mmm"; "mmmr" };
\end{xy}
}
\bigplus
\adjustbox{valign=m}{
\begin{xy}
(  3,   0 )*{} = "root";
(  0,  -4 )*{} = "l";
(  3,  -4 )*{} = "m";
(  6,  -4 )*{} = "r";
{ \ar@{-} "root"; "l" };
{ \ar@{-} "root"; "m" };
{ \ar@{-} "root"; "r" };
(  3, -8 )*{} = "rl";
(  6, -8 )*{} = "rm";
(  9, -8 )*{} = "rr";
{ \ar@{-} "r"; "rl" };
{ \ar@{-} "r"; "rm" };
{ \ar@{-} "r"; "rr" };
(  0, -12 )*{} = "rll";
(  3, -12 )*{} = "rlm";
(  6, -12 )*{} = "rlr";
{ \ar@{-} "rl"; "rll" };
{ \ar@{-} "rl"; "rlm" };
{ \ar@{-} "rl"; "rlr" };
(  0, -16 )*{} = "rlml";
(  3, -16 )*{} = "rlmm";
(  6, -16 )*{} = "rlmr";
{ \ar@{-} "rlm"; "rlml" };
{ \ar@{-} "rlm"; "rlmm" };
{ \ar@{-} "rlm"; "rlmr" };
\end{xy}
}
\bigplus
\adjustbox{valign=m}{
\begin{xy}
(  3,   0 )*{} = "root";
(  0,  -4 )*{} = "l";
(  3,  -4 )*{} = "m";
(  6,  -4 )*{} = "r";
{ \ar@{-} "root"; "l" };
{ \ar@{-} "root"; "m" };
{ \ar@{-} "root"; "r" };
(  0, -8 )*{} = "ml";
(  3, -8 )*{} = "mm";
(  6, -8 )*{} = "mr";
{ \ar@{-} "m"; "ml" };
{ \ar@{-} "m"; "mm" };
{ \ar@{-} "m"; "mr" };
(  0, -12 )*{} = "mml";
(  3, -12 )*{} = "mmm";
(  6, -12 )*{} = "mmr";
{ \ar@{-} "mm"; "mml" };
{ \ar@{-} "mm"; "mmm" };
{ \ar@{-} "mm"; "mmr" };
(  3, -16 )*{} = "mmrl";
(  6, -16 )*{} = "mmrm";
(  9, -16 )*{} = "mmrr";
{ \ar@{-} "mmr"; "mmrl" };
{ \ar@{-} "mmr"; "mmrm" };
{ \ar@{-} "mmr"; "mmrr" };
\end{xy}
}
\bigplus
\adjustbox{valign=m}{
\begin{xy}
(  3,   0 )*{} = "root";
(  0,  -4 )*{} = "l";
(  3,  -4 )*{} = "m";
(  6,  -4 )*{} = "r";
{ \ar@{-} "root"; "l" };
{ \ar@{-} "root"; "m" };
{ \ar@{-} "root"; "r" };
(  3, -8 )*{} = "rl";
(  6, -8 )*{} = "rm";
(  9, -8 )*{} = "rr";
{ \ar@{-} "r"; "rl" };
{ \ar@{-} "r"; "rm" };
{ \ar@{-} "r"; "rr" };
(  0, -12 )*{} = "rll";
(  3, -12 )*{} = "rlm";
(  6, -12 )*{} = "rlr";
{ \ar@{-} "rl"; "rll" };
{ \ar@{-} "rl"; "rlm" };
{ \ar@{-} "rl"; "rlr" };
(  3, -16 )*{} = "rlrl";
(  6, -16 )*{} = "rlrm";
(  9, -16 )*{} = "rlrr";
{ \ar@{-} "rlr"; "rlrl" };
{ \ar@{-} "rlr"; "rlrm" };
{ \ar@{-} "rlr"; "rlrr" };
\end{xy}
}
\end{align*}
Terms 1 and 4 cancel, leaving
\[
\bigminus
\adjustbox{valign=m}{
\begin{xy}
(  3,   0 )*{} = "root";
(  0,  -4 )*{} = "l";
(  3,  -4 )*{} = "m";
(  6,  -4 )*{} = "r";
{ \ar@{-} "root"; "l" };
{ \ar@{-} "root"; "m" };
{ \ar@{-} "root"; "r" };
(  0, -8 )*{} = "ml";
(  3, -8 )*{} = "mm";
(  6, -8 )*{} = "mr";
{ \ar@{-} "m"; "ml" };
{ \ar@{-} "m"; "mm" };
{ \ar@{-} "m"; "mr" };
(  3, -12 )*{} = "mrl";
(  6, -12 )*{} = "mrm";
(  9, -12 )*{} = "mrr";
{ \ar@{-} "mr"; "mrl" };
{ \ar@{-} "mr"; "mrm" };
{ \ar@{-} "mr"; "mrr" };
(  0, -16 )*{} = "mrll";
(  3, -16 )*{} = "mrlm";
(  6, -16 )*{} = "mrlr";
{ \ar@{-} "mrl"; "mrll" };
{ \ar@{-} "mrl"; "mrlm" };
{ \ar@{-} "mrl"; "mrlr" };
\end{xy}
}
\bigplus
\adjustbox{valign=m}{
\begin{xy}
(  0,   0 )*{} = "root";
( -9,  -4 )*{} = "l";
(  0,  -4 )*{} = "m";
(  9,  -4 )*{} = "r";
{ \ar@{-} "root"; "l" };
{ \ar@{-} "root"; "m" };
{ \ar@{-} "root"; "r" };
( -3, -8 )*{} = "ml";
(  0, -8 )*{} = "mm";
(  3, -8 )*{} = "mr";
{ \ar@{-} "m"; "ml" };
{ \ar@{-} "m"; "mm" };
{ \ar@{-} "m"; "mr" };
(  6, -8 )*{} = "rl";
(  9, -8 )*{} = "rm";
( 12, -8 )*{} = "rr";
{ \ar@{-} "r"; "rl" };
{ \ar@{-} "r"; "rm" };
{ \ar@{-} "r"; "rr" };
(  0, -12 )*{} = "mrl";
(  3, -12 )*{} = "mrm";
(  6, -12 )*{} = "mrr";
{ \ar@{-} "mr"; "mrl" };
{ \ar@{-} "mr"; "mrm" };
{ \ar@{-} "mr"; "mrr" };
\end{xy}
}
\bigplus
\adjustbox{valign=m}{
\begin{xy}
(  3,   0 )*{} = "root";
(  0,  -4 )*{} = "l";
(  3,  -4 )*{} = "m";
(  6,  -4 )*{} = "r";
{ \ar@{-} "root"; "l" };
{ \ar@{-} "root"; "m" };
{ \ar@{-} "root"; "r" };
(  3, -8 )*{} = "rl";
(  6, -8 )*{} = "rm";
(  9, -8 )*{} = "rr";
{ \ar@{-} "r"; "rl" };
{ \ar@{-} "r"; "rm" };
{ \ar@{-} "r"; "rr" };
(  0, -12 )*{} = "rll";
(  3, -12 )*{} = "rlm";
(  6, -12 )*{} = "rlr";
{ \ar@{-} "rl"; "rll" };
{ \ar@{-} "rl"; "rlm" };
{ \ar@{-} "rl"; "rlr" };
(  0, -16 )*{} = "rlml";
(  3, -16 )*{} = "rlmm";
(  6, -16 )*{} = "rlmr";
{ \ar@{-} "rlm"; "rlml" };
{ \ar@{-} "rlm"; "rlmm" };
{ \ar@{-} "rlm"; "rlmr" };
\end{xy}
}
\bigplus
\adjustbox{valign=m}{
\begin{xy}
(  3,   0 )*{\bullet} = "root";
(  0,  -4 )*{} = "l";
(  3,  -4 )*{} = "m";
(  6,  -4 )*{} = "r";
{ \ar@{-} "root"; "l" };
{ \ar@{-} "root"; "m" };
{ \ar@{-} "root"; "r" };
(  0, -8 )*{} = "ml";
(  3, -8 )*{} = "mm";
(  6, -8 )*{} = "mr";
{ \ar@{-} "m"; "ml" };
{ \ar@{-} "m"; "mm" };
{ \ar@{-} "m"; "mr" };
(  0, -12 )*{} = "mml";
(  3, -12 )*{} = "mmm";
(  6, -12 )*{} = "mmr";
{ \ar@{-} "mm"; "mml" };
{ \ar@{-} "mm"; "mmm" };
{ \ar@{-} "mm"; "mmr" };
(  3, -16 )*{} = "mmrl";
(  6, -16 )*{} = "mmrm";
(  9, -16 )*{} = "mmrr";
{ \ar@{-} "mmr"; "mmrl" };
{ \ar@{-} "mmr"; "mmrm" };
{ \ar@{-} "mmr"; "mmrr" };
\end{xy}
}
\bigplus
\adjustbox{valign=m}{
\begin{xy}
(  3,   0 )*{} = "root";
(  0,  -4 )*{} = "l";
(  3,  -4 )*{} = "m";
(  6,  -4 )*{} = "r";
{ \ar@{-} "root"; "l" };
{ \ar@{-} "root"; "m" };
{ \ar@{-} "root"; "r" };
(  3, -8 )*{} = "rl";
(  6, -8 )*{} = "rm";
(  9, -8 )*{} = "rr";
{ \ar@{-} "r"; "rl" };
{ \ar@{-} "r"; "rm" };
{ \ar@{-} "r"; "rr" };
(  0, -12 )*{} = "rll";
(  3, -12 )*{} = "rlm";
(  6, -12 )*{} = "rlr";
{ \ar@{-} "rl"; "rll" };
{ \ar@{-} "rl"; "rlm" };
{ \ar@{-} "rl"; "rlr" };
(  3, -16 )*{} = "rlrl";
(  6, -16 )*{} = "rlrm";
(  9, -16 )*{} = "rlrr";
{ \ar@{-} "rlr"; "rlrl" };
{ \ar@{-} "rlr"; "rlrm" };
{ \ar@{-} "rlr"; "rlrr" };
\end{xy}
}
\]
The leading monomial is divisible by $\ell m(\beta)$ but not $\ell m(\alpha)$;
we reduce it using $\beta$:
\[
\bigminus
\adjustbox{valign=m}{
\begin{xy}
(  3,   0 )*{\bullet} = "root";
(  0,  -4 )*{} = "l";
(  3,  -4 )*{} = "m";
(  6,  -4 )*{} = "r";
{ \ar@{-} "root"; "l" };
{ \ar@{-} "root"; "m" };
{ \ar@{-} "root"; "r" };
(  0, -8 )*{} = "ml";
(  3, -8 )*{} = "mm";
(  6, -8 )*{} = "mr";
{ \ar@{-} "m"; "ml" };
{ \ar@{-} "m"; "mm" };
{ \ar@{-} "m"; "mr" };
(  3, -12 )*{} = "mrl";
(  6, -12 )*{} = "mrm";
(  9, -12 )*{} = "mrr";
{ \ar@{-} "mr"; "mrl" };
{ \ar@{-} "mr"; "mrm" };
{ \ar@{-} "mr"; "mrr" };
(  0, -16 )*{} = "mrll";
(  3, -16 )*{} = "mrlm";
(  6, -16 )*{} = "mrlr";
{ \ar@{-} "mrl"; "mrll" };
{ \ar@{-} "mrl"; "mrlm" };
{ \ar@{-} "mrl"; "mrlr" };
\end{xy}
}
\bigplus
\adjustbox{valign=m}{
\begin{xy}
(  0,   0 )*{} = "root";
( -9,  -4 )*{} = "l";
(  0,  -4 )*{} = "m";
(  9,  -4 )*{} = "r";
{ \ar@{-} "root"; "l" };
{ \ar@{-} "root"; "m" };
{ \ar@{-} "root"; "r" };
( -3, -8 )*{} = "ml";
(  0, -8 )*{} = "mm";
(  3, -8 )*{} = "mr";
{ \ar@{-} "m"; "ml" };
{ \ar@{-} "m"; "mm" };
{ \ar@{-} "m"; "mr" };
(  6, -8 )*{} = "rl";
(  9, -8 )*{} = "rm";
( 12, -8 )*{} = "rr";
{ \ar@{-} "r"; "rl" };
{ \ar@{-} "r"; "rm" };
{ \ar@{-} "r"; "rr" };
(  0, -12 )*{} = "mrl";
(  3, -12 )*{} = "mrm";
(  6, -12 )*{} = "mrr";
{ \ar@{-} "mr"; "mrl" };
{ \ar@{-} "mr"; "mrm" };
{ \ar@{-} "mr"; "mrr" };
\end{xy}
}
\bigplus
\adjustbox{valign=m}{
\begin{xy}
(  3,   0 )*{} = "root";
(  0,  -4 )*{} = "l";
(  3,  -4 )*{} = "m";
(  6,  -4 )*{} = "r";
{ \ar@{-} "root"; "l" };
{ \ar@{-} "root"; "m" };
{ \ar@{-} "root"; "r" };
(  3, -8 )*{} = "rl";
(  6, -8 )*{} = "rm";
(  9, -8 )*{} = "rr";
{ \ar@{-} "r"; "rl" };
{ \ar@{-} "r"; "rm" };
{ \ar@{-} "r"; "rr" };
(  0, -12 )*{} = "rll";
(  3, -12 )*{} = "rlm";
(  6, -12 )*{} = "rlr";
{ \ar@{-} "rl"; "rll" };
{ \ar@{-} "rl"; "rlm" };
{ \ar@{-} "rl"; "rlr" };
(  0, -16 )*{} = "rlml";
(  3, -16 )*{} = "rlmm";
(  6, -16 )*{} = "rlmr";
{ \ar@{-} "rlm"; "rlml" };
{ \ar@{-} "rlm"; "rlmm" };
{ \ar@{-} "rlm"; "rlmr" };
\end{xy}
}
\bigminus
\adjustbox{valign=m}{
\begin{xy}
(  3,   0 )*{} = "root";
(  0,  -4 )*{} = "l";
(  3,  -4 )*{} = "m";
(  6,  -4 )*{} = "r";
{ \ar@{-} "root"; "l" };
{ \ar@{-} "root"; "m" };
{ \ar@{-} "root"; "r" };
(  0, -8 )*{} = "ml";
(  3, -8 )*{} = "mm";
(  6, -8 )*{} = "mr";
{ \ar@{-} "m"; "ml" };
{ \ar@{-} "m"; "mm" };
{ \ar@{-} "m"; "mr" };
(  3, -12 )*{} = "mrl";
(  6, -12 )*{} = "mrm";
(  9, -12 )*{} = "mrr";
{ \ar@{-} "mr"; "mrl" };
{ \ar@{-} "mr"; "mrm" };
{ \ar@{-} "mr"; "mrr" };
(  3, -16 )*{} = "mrml";
(  6, -16 )*{} = "mrmm";
(  9, -16 )*{} = "mrmr";
{ \ar@{-} "mrm"; "mrml" };
{ \ar@{-} "mrm"; "mrmm" };
{ \ar@{-} "mrm"; "mrmr" };
\end{xy}
}
\bigminus
\adjustbox{valign=m}{
\begin{xy}
(  3,   0 )*{} = "root";
(  0,  -4 )*{} = "l";
(  3,  -4 )*{} = "m";
(  6,  -4 )*{} = "r";
{ \ar@{-} "root"; "l" };
{ \ar@{-} "root"; "m" };
{ \ar@{-} "root"; "r" };
(  0, -8 )*{} = "ml";
(  3, -8 )*{} = "mm";
(  6, -8 )*{} = "mr";
{ \ar@{-} "m"; "ml" };
{ \ar@{-} "m"; "mm" };
{ \ar@{-} "m"; "mr" };
(  3, -12 )*{} = "mrl";
(  6, -12 )*{} = "mrm";
(  9, -12 )*{} = "mrr";
{ \ar@{-} "mr"; "mrl" };
{ \ar@{-} "mr"; "mrm" };
{ \ar@{-} "mr"; "mrr" };
(  6, -16 )*{} = "mrrl";
(  9, -16 )*{} = "mrrm";
( 12, -16 )*{} = "mrrr";
{ \ar@{-} "mrr"; "mrrl" };
{ \ar@{-} "mrr"; "mrrm" };
{ \ar@{-} "mrr"; "mrrr" };
\end{xy}
}
\bigplus
\adjustbox{valign=m}{
\begin{xy}
(  3,   0 )*{} = "root";
(  0,  -4 )*{} = "l";
(  3,  -4 )*{} = "m";
(  6,  -4 )*{} = "r";
{ \ar@{-} "root"; "l" };
{ \ar@{-} "root"; "m" };
{ \ar@{-} "root"; "r" };
(  3, -8 )*{} = "rl";
(  6, -8 )*{} = "rm";
(  9, -8 )*{} = "rr";
{ \ar@{-} "r"; "rl" };
{ \ar@{-} "r"; "rm" };
{ \ar@{-} "r"; "rr" };
(  0, -12 )*{} = "rll";
(  3, -12 )*{} = "rlm";
(  6, -12 )*{} = "rlr";
{ \ar@{-} "rl"; "rll" };
{ \ar@{-} "rl"; "rlm" };
{ \ar@{-} "rl"; "rlr" };
(  3, -16 )*{} = "rlrl";
(  6, -16 )*{} = "rlrm";
(  9, -16 )*{} = "rlrr";
{ \ar@{-} "rlr"; "rlrl" };
{ \ar@{-} "rlr"; "rlrm" };
{ \ar@{-} "rlr"; "rlrr" };
\end{xy}
}
\]
The leading monomial is divisible by $\alpha$ (bottom) and $\beta$ (top).
Using $\alpha$ gives
\[
\adjustbox{valign=m}{
\begin{xy}
(  3,   0 )*{} = "root";
(  0,  -4 )*{} = "l";
(  3,  -4 )*{} = "m";
(  6,  -4 )*{} = "r";
{ \ar@{-} "root"; "l" };
{ \ar@{-} "root"; "m" };
{ \ar@{-} "root"; "r" };
(  0, -8 )*{} = "ml";
(  3, -8 )*{} = "mm";
(  6, -8 )*{} = "mr";
{ \ar@{-} "m"; "ml" };
{ \ar@{-} "m"; "mm" };
{ \ar@{-} "m"; "mr" };
(  3, -12 )*{} = "mrl";
(  6, -12 )*{} = "mrm";
(  9, -12 )*{} = "mrr";
{ \ar@{-} "mr"; "mrl" };
{ \ar@{-} "mr"; "mrm" };
{ \ar@{-} "mr"; "mrr" };
(  3, -16 )*{} = "mrml";
(  6, -16 )*{} = "mrmm";
(  9, -16 )*{} = "mrmr";
{ \ar@{-} "mrm"; "mrml" };
{ \ar@{-} "mrm"; "mrmm" };
{ \ar@{-} "mrm"; "mrmr" };
\end{xy}
}
\bigplus
\adjustbox{valign=m}{
\begin{xy}
(  3,   0 )*{} = "root";
(  0,  -4 )*{} = "l";
(  3,  -4 )*{} = "m";
(  6,  -4 )*{} = "r";
{ \ar@{-} "root"; "l" };
{ \ar@{-} "root"; "m" };
{ \ar@{-} "root"; "r" };
(  0, -8 )*{} = "ml";
(  3, -8 )*{} = "mm";
(  6, -8 )*{} = "mr";
{ \ar@{-} "m"; "ml" };
{ \ar@{-} "m"; "mm" };
{ \ar@{-} "m"; "mr" };
(  3, -12 )*{} = "mrl";
(  6, -12 )*{} = "mrm";
(  9, -12 )*{} = "mrr";
{ \ar@{-} "mr"; "mrl" };
{ \ar@{-} "mr"; "mrm" };
{ \ar@{-} "mr"; "mrr" };
(  6, -16 )*{} = "mrrl";
(  9, -16 )*{} = "mrrm";
( 12, -16 )*{} = "mrrr";
{ \ar@{-} "mrr"; "mrrl" };
{ \ar@{-} "mrr"; "mrrm" };
{ \ar@{-} "mrr"; "mrrr" };
\end{xy}
}
\!\!\!\!\!\!
\bigplus
\adjustbox{valign=m}{
\begin{xy}
(  0,   0 )*{} = "root";
( -9,  -4 )*{} = "l";
(  0,  -4 )*{} = "m";
(  9,  -4 )*{} = "r";
{ \ar@{-} "root"; "l" };
{ \ar@{-} "root"; "m" };
{ \ar@{-} "root"; "r" };
( -3, -8 )*{} = "ml";
(  0, -8 )*{} = "mm";
(  3, -8 )*{} = "mr";
{ \ar@{-} "m"; "ml" };
{ \ar@{-} "m"; "mm" };
{ \ar@{-} "m"; "mr" };
(  6, -8 )*{} = "rl";
(  9, -8 )*{} = "rm";
( 12, -8 )*{} = "rr";
{ \ar@{-} "r"; "rl" };
{ \ar@{-} "r"; "rm" };
{ \ar@{-} "r"; "rr" };
(  0, -12 )*{} = "mrl";
(  3, -12 )*{} = "mrm";
(  6, -12 )*{} = "mrr";
{ \ar@{-} "mr"; "mrl" };
{ \ar@{-} "mr"; "mrm" };
{ \ar@{-} "mr"; "mrr" };
\end{xy}
}
\bigplus
\adjustbox{valign=m}{
\begin{xy}
(  3,   0 )*{} = "root";
(  0,  -4 )*{} = "l";
(  3,  -4 )*{} = "m";
(  6,  -4 )*{} = "r";
{ \ar@{-} "root"; "l" };
{ \ar@{-} "root"; "m" };
{ \ar@{-} "root"; "r" };
(  3, -8 )*{} = "rl";
(  6, -8 )*{} = "rm";
(  9, -8 )*{} = "rr";
{ \ar@{-} "r"; "rl" };
{ \ar@{-} "r"; "rm" };
{ \ar@{-} "r"; "rr" };
(  0, -12 )*{} = "rll";
(  3, -12 )*{} = "rlm";
(  6, -12 )*{} = "rlr";
{ \ar@{-} "rl"; "rll" };
{ \ar@{-} "rl"; "rlm" };
{ \ar@{-} "rl"; "rlr" };
(  0, -16 )*{} = "rlml";
(  3, -16 )*{} = "rlmm";
(  6, -16 )*{} = "rlmr";
{ \ar@{-} "rlm"; "rlml" };
{ \ar@{-} "rlm"; "rlmm" };
{ \ar@{-} "rlm"; "rlmr" };
\end{xy}
}
\bigminus
\adjustbox{valign=m}{
\begin{xy}
(  3,   0 )*{} = "root";
(  0,  -4 )*{} = "l";
(  3,  -4 )*{} = "m";
(  6,  -4 )*{} = "r";
{ \ar@{-} "root"; "l" };
{ \ar@{-} "root"; "m" };
{ \ar@{-} "root"; "r" };
(  0, -8 )*{} = "ml";
(  3, -8 )*{} = "mm";
(  6, -8 )*{} = "mr";
{ \ar@{-} "m"; "ml" };
{ \ar@{-} "m"; "mm" };
{ \ar@{-} "m"; "mr" };
(  3, -12 )*{} = "mrl";
(  6, -12 )*{} = "mrm";
(  9, -12 )*{} = "mrr";
{ \ar@{-} "mr"; "mrl" };
{ \ar@{-} "mr"; "mrm" };
{ \ar@{-} "mr"; "mrr" };
(  3, -16 )*{} = "mrml";
(  6, -16 )*{} = "mrmm";
(  9, -16 )*{} = "mrmr";
{ \ar@{-} "mrm"; "mrml" };
{ \ar@{-} "mrm"; "mrmm" };
{ \ar@{-} "mrm"; "mrmr" };
\end{xy}
}
\bigminus
\adjustbox{valign=m}{
\begin{xy}
(  3,   0 )*{} = "root";
(  0,  -4 )*{} = "l";
(  3,  -4 )*{} = "m";
(  6,  -4 )*{} = "r";
{ \ar@{-} "root"; "l" };
{ \ar@{-} "root"; "m" };
{ \ar@{-} "root"; "r" };
(  0, -8 )*{} = "ml";
(  3, -8 )*{} = "mm";
(  6, -8 )*{} = "mr";
{ \ar@{-} "m"; "ml" };
{ \ar@{-} "m"; "mm" };
{ \ar@{-} "m"; "mr" };
(  3, -12 )*{} = "mrl";
(  6, -12 )*{} = "mrm";
(  9, -12 )*{} = "mrr";
{ \ar@{-} "mr"; "mrl" };
{ \ar@{-} "mr"; "mrm" };
{ \ar@{-} "mr"; "mrr" };
(  6, -16 )*{} = "mrrl";
(  9, -16 )*{} = "mrrm";
( 12, -16 )*{} = "mrrr";
{ \ar@{-} "mrr"; "mrrl" };
{ \ar@{-} "mrr"; "mrrm" };
{ \ar@{-} "mrr"; "mrrr" };
\end{xy}
}
\!\!\!\!\!\!
\bigplus
\adjustbox{valign=m}{
\begin{xy}
(  3,   0 )*{} = "root";
(  0,  -4 )*{} = "l";
(  3,  -4 )*{} = "m";
(  6,  -4 )*{} = "r";
{ \ar@{-} "root"; "l" };
{ \ar@{-} "root"; "m" };
{ \ar@{-} "root"; "r" };
(  3, -8 )*{} = "rl";
(  6, -8 )*{} = "rm";
(  9, -8 )*{} = "rr";
{ \ar@{-} "r"; "rl" };
{ \ar@{-} "r"; "rm" };
{ \ar@{-} "r"; "rr" };
(  0, -12 )*{} = "rll";
(  3, -12 )*{} = "rlm";
(  6, -12 )*{} = "rlr";
{ \ar@{-} "rl"; "rll" };
{ \ar@{-} "rl"; "rlm" };
{ \ar@{-} "rl"; "rlr" };
(  3, -16 )*{} = "rlrl";
(  6, -16 )*{} = "rlrm";
(  9, -16 )*{} = "rlrr";
{ \ar@{-} "rlr"; "rlrl" };
{ \ar@{-} "rlr"; "rlrm" };
{ \ar@{-} "rlr"; "rlrr" };
\end{xy}
}
\]
Terms 1, 5 and terms 2, 6 cancel, leaving
\[
\adjustbox{valign=m}{
\begin{xy}
(  0,   0 )*{\bullet} = "root";
( -9,  -4 )*{} = "l";
(  0,  -4 )*{} = "m";
(  9,  -4 )*{} = "r";
{ \ar@{-} "root"; "l" };
{ \ar@{-} "root"; "m" };
{ \ar@{-} "root"; "r" };
( -3, -8 )*{} = "ml";
(  0, -8 )*{} = "mm";
(  3, -8 )*{} = "mr";
{ \ar@{-} "m"; "ml" };
{ \ar@{-} "m"; "mm" };
{ \ar@{-} "m"; "mr" };
(  6, -8 )*{} = "rl";
(  9, -8 )*{} = "rm";
( 12, -8 )*{} = "rr";
{ \ar@{-} "r"; "rl" };
{ \ar@{-} "r"; "rm" };
{ \ar@{-} "r"; "rr" };
(  0, -12 )*{} = "mrl";
(  3, -12 )*{} = "mrm";
(  6, -12 )*{} = "mrr";
{ \ar@{-} "mr"; "mrl" };
{ \ar@{-} "mr"; "mrm" };
{ \ar@{-} "mr"; "mrr" };
\end{xy}
}
\bigplus
\adjustbox{valign=m}{
\begin{xy}
(  3,   0 )*{} = "root";
(  0,  -4 )*{} = "l";
(  3,  -4 )*{} = "m";
(  6,  -4 )*{} = "r";
{ \ar@{-} "root"; "l" };
{ \ar@{-} "root"; "m" };
{ \ar@{-} "root"; "r" };
(  3, -8 )*{} = "rl";
(  6, -8 )*{} = "rm";
(  9, -8 )*{} = "rr";
{ \ar@{-} "r"; "rl" };
{ \ar@{-} "r"; "rm" };
{ \ar@{-} "r"; "rr" };
(  0, -12 )*{} = "rll";
(  3, -12 )*{} = "rlm";
(  6, -12 )*{} = "rlr";
{ \ar@{-} "rl"; "rll" };
{ \ar@{-} "rl"; "rlm" };
{ \ar@{-} "rl"; "rlr" };
(  0, -16 )*{} = "rlml";
(  3, -16 )*{} = "rlmm";
(  6, -16 )*{} = "rlmr";
{ \ar@{-} "rlm"; "rlml" };
{ \ar@{-} "rlm"; "rlmm" };
{ \ar@{-} "rlm"; "rlmr" };
\end{xy}
}
\bigplus
\adjustbox{valign=m}{
\begin{xy}
(  3,   0 )*{} = "root";
(  0,  -4 )*{} = "l";
(  3,  -4 )*{} = "m";
(  6,  -4 )*{} = "r";
{ \ar@{-} "root"; "l" };
{ \ar@{-} "root"; "m" };
{ \ar@{-} "root"; "r" };
(  3, -8 )*{} = "rl";
(  6, -8 )*{} = "rm";
(  9, -8 )*{} = "rr";
{ \ar@{-} "r"; "rl" };
{ \ar@{-} "r"; "rm" };
{ \ar@{-} "r"; "rr" };
(  0, -12 )*{} = "rll";
(  3, -12 )*{} = "rlm";
(  6, -12 )*{} = "rlr";
{ \ar@{-} "rl"; "rll" };
{ \ar@{-} "rl"; "rlm" };
{ \ar@{-} "rl"; "rlr" };
(  3, -16 )*{} = "rlrl";
(  6, -16 )*{} = "rlrm";
(  9, -16 )*{} = "rlrr";
{ \ar@{-} "rlr"; "rlrl" };
{ \ar@{-} "rlr"; "rlrm" };
{ \ar@{-} "rlr"; "rlrr" };
\end{xy}
}
\]
We reduce the leading monomial using $\beta$:
\[
\bigminus
\adjustbox{valign=m}{
\begin{xy}
(  0,   0 )*{} = "root";
( -6,  -4 )*{} = "l";
(  0,  -4 )*{} = "m";
(  6,  -4 )*{} = "r";
{ \ar@{-} "root"; "l" };
{ \ar@{-} "root"; "m" };
{ \ar@{-} "root"; "r" };
( -3,  -8 )*{} = "rl";
(  6,  -8 )*{} = "rm";
( 15,  -8 )*{} = "rr";
{ \ar@{-} "r"; "rl" };
{ \ar@{-} "r"; "rm" };
{ \ar@{-} "r"; "rr" };
(  3, -12 )*{} = "rml";
(  6, -12 )*{} = "rmm";
(  9, -12 )*{} = "rmr";
{ \ar@{-} "rm"; "rml" };
{ \ar@{-} "rm"; "rmm" };
{ \ar@{-} "rm"; "rmr" };
( 12, -12 )*{} = "rrl";
( 15, -12 )*{} = "rrm";
( 18, -12 )*{} = "rrr";
{ \ar@{-} "rr"; "rrl" };
{ \ar@{-} "rr"; "rrm" };
{ \ar@{-} "rr"; "rrr" };
\end{xy}
}
\bigminus
\adjustbox{valign=m}{
\begin{xy}
(  3,   0 )*{} = "root";
(  0,  -4 )*{} = "l";
(  3,  -4 )*{} = "m";
(  6,  -4 )*{} = "r";
{ \ar@{-} "root"; "l" };
{ \ar@{-} "root"; "m" };
{ \ar@{-} "root"; "r" };
( 3, -8 )*{} = "rl";
(  6, -8 )*{} = "rm";
(  9, -8 )*{} = "rr";
{ \ar@{-} "r"; "rl" };
{ \ar@{-} "r"; "rm" };
{ \ar@{-} "r"; "rr" };
( 6, -12 )*{} = "rrl";
(  9, -12 )*{} = "rrm";
(  12, -12 )*{} = "rrr";
{ \ar@{-} "rr"; "rrr" };
{ \ar@{-} "rr"; "rrm" };
{ \ar@{-} "rr"; "rrl" };
(  9, -16 )*{} = "rrml";
(  12, -16 )*{} = "rrmm";
(  15, -16 )*{} = "rrmr";
{ \ar@{-} "rrr"; "rrml" };
{ \ar@{-} "rrr"; "rrmm" };
{ \ar@{-} "rrr"; "rrmr" };
\end{xy}
}
\bigplus
\adjustbox{valign=m}{
\begin{xy}
(  3,   0 )*{\bullet} = "root";
(  0,  -4 )*{} = "l";
(  3,  -4 )*{} = "m";
(  6,  -4 )*{} = "r";
{ \ar@{-} "root"; "l" };
{ \ar@{-} "root"; "m" };
{ \ar@{-} "root"; "r" };
(  3, -8 )*{} = "rl";
(  6, -8 )*{} = "rm";
(  9, -8 )*{} = "rr";
{ \ar@{-} "r"; "rl" };
{ \ar@{-} "r"; "rm" };
{ \ar@{-} "r"; "rr" };
(  0, -12 )*{} = "rll";
(  3, -12 )*{} = "rlm";
(  6, -12 )*{} = "rlr";
{ \ar@{-} "rl"; "rll" };
{ \ar@{-} "rl"; "rlm" };
{ \ar@{-} "rl"; "rlr" };
(  0, -16 )*{} = "rlml";
(  3, -16 )*{} = "rlmm";
(  6, -16 )*{} = "rlmr";
{ \ar@{-} "rlm"; "rlml" };
{ \ar@{-} "rlm"; "rlmm" };
{ \ar@{-} "rlm"; "rlmr" };
\end{xy}
}
\bigplus
\adjustbox{valign=m}{
\begin{xy}
(  3,   0 )*{} = "root";
(  0,  -4 )*{} = "l";
(  3,  -4 )*{} = "m";
(  6,  -4 )*{} = "r";
{ \ar@{-} "root"; "l" };
{ \ar@{-} "root"; "m" };
{ \ar@{-} "root"; "r" };
(  3, -8 )*{} = "rl";
(  6, -8 )*{} = "rm";
(  9, -8 )*{} = "rr";
{ \ar@{-} "r"; "rl" };
{ \ar@{-} "r"; "rm" };
{ \ar@{-} "r"; "rr" };
(  0, -12 )*{} = "rll";
(  3, -12 )*{} = "rlm";
(  6, -12 )*{} = "rlr";
{ \ar@{-} "rl"; "rll" };
{ \ar@{-} "rl"; "rlm" };
{ \ar@{-} "rl"; "rlr" };
(  3, -16 )*{} = "rlrl";
(  6, -16 )*{} = "rlrm";
(  9, -16 )*{} = "rlrr";
{ \ar@{-} "rlr"; "rlrl" };
{ \ar@{-} "rlr"; "rlrm" };
{ \ar@{-} "rlr"; "rlrr" };
\end{xy}
}
\]
Terms 1, 2 cannot be reduced; terms 3, 4 can be reduced by $\alpha$:
\[
\bigminus
\scalebox{1.5}{2}
\adjustbox{valign=m}{
\begin{xy}
(  0,   0 )*{} = "root";
( -6,  -4 )*{} = "l";
(  0,  -4 )*{} = "m";
(  6,  -4 )*{} = "r";
{ \ar@{-} "root"; "l" };
{ \ar@{-} "root"; "m" };
{ \ar@{-} "root"; "r" };
( -3,  -8 )*{} = "rl";
(  6,  -8 )*{} = "rm";
( 15,  -8 )*{} = "rr";
{ \ar@{-} "r"; "rl" };
{ \ar@{-} "r"; "rm" };
{ \ar@{-} "r"; "rr" };
(  3, -12 )*{} = "rml";
(  6, -12 )*{} = "rmm";
(  9, -12 )*{} = "rmr";
{ \ar@{-} "rm"; "rml" };
{ \ar@{-} "rm"; "rmm" };
{ \ar@{-} "rm"; "rmr" };
( 12, -12 )*{} = "rrl";
( 15, -12 )*{} = "rrm";
( 18, -12 )*{} = "rrr";
{ \ar@{-} "rr"; "rrl" };
{ \ar@{-} "rr"; "rrm" };
{ \ar@{-} "rr"; "rrr" };
\end{xy}
}
\bigminus
\adjustbox{valign=m}{
\begin{xy}
(  3,   0 )*{} = "root";
(  0,  -4 )*{} = "l";
(  3,  -4 )*{} = "m";
(  6,  -4 )*{} = "r";
{ \ar@{-} "root"; "l" };
{ \ar@{-} "root"; "m" };
{ \ar@{-} "root"; "r" };
( 3, -8 )*{} = "rl";
(  6, -8 )*{} = "rm";
(  9, -8 )*{} = "rr";
{ \ar@{-} "r"; "rl" };
{ \ar@{-} "r"; "rm" };
{ \ar@{-} "r"; "rr" };
( 6, -12 )*{} = "rrl";
(  9, -12 )*{} = "rrm";
(  12, -12 )*{} = "rrr";
{ \ar@{-} "rr"; "rrr" };
{ \ar@{-} "rr"; "rrm" };
{ \ar@{-} "rr"; "rrl" };
(  9, -16 )*{} = "rrml";
(  12, -16 )*{} = "rrmm";
(  15, -16 )*{} = "rrmr";
{ \ar@{-} "rrr"; "rrml" };
{ \ar@{-} "rrr"; "rrmm" };
{ \ar@{-} "rrr"; "rrmr" };
\end{xy}
}
\bigminus
\adjustbox{valign=m}{
\begin{xy}
(  3,   0 )*{} = "root";
(  0,  -4 )*{} = "l";
(  3,  -4 )*{} = "m";
(  6,  -4 )*{} = "r";
{ \ar@{-} "root"; "l" };
{ \ar@{-} "root"; "m" };
{ \ar@{-} "root"; "r" };
(  3, -8 )*{} = "rl";
(  6, -8 )*{} = "rm";
(  9, -8 )*{} = "rr";
{ \ar@{-} "r"; "rl" };
{ \ar@{-} "r"; "rm" };
{ \ar@{-} "r"; "rr" };
( 3, -12 )*{} = "rml";
(  6, -12 )*{} = "rmm";
(  9, -12 )*{} = "rmr";
{ \ar@{-} "rm"; "rml" };
{ \ar@{-} "rm"; "rmm" };
{ \ar@{-} "rm"; "rmr" };
(  0, -16 )*{} = "rmll";
(  3, -16 )*{} = "rmlm";
(  6, -16 )*{} = "rmlr";
{ \ar@{-} "rml"; "rmll" };
{ \ar@{-} "rml"; "rmlm" };
{ \ar@{-} "rml"; "rmlr" };
\end{xy}
}
\bigminus
\adjustbox{valign=m}{
\begin{xy}
(  3,   0 )*{} = "root";
(  0,  -4 )*{} = "l";
(  3,  -4 )*{} = "m";
(  6,  -4 )*{} = "r";
{ \ar@{-} "root"; "l" };
{ \ar@{-} "root"; "m" };
{ \ar@{-} "root"; "r" };
(  3, -8 )*{} = "ml";
(  6, -8 )*{} = "mm";
(  9, -8 )*{} = "mr";
{ \ar@{-} "r"; "ml" };
{ \ar@{-} "r"; "mm" };
{ \ar@{-} "r"; "mr" };
( 3, -12 )*{} = "rml";
(  6, -12 )*{} = "rmm";
(  9, -12 )*{} = "rmr";
{ \ar@{-} "mm"; "rml" };
{ \ar@{-} "mm"; "rmm" };
{ \ar@{-} "mm"; "rmr" };
(  3, -16 )*{} = "rmrl";
(  6, -16 )*{} = "rmrm";
(  9, -16 )*{} = "rmrr";
{ \ar@{-} "rmm"; "rmrl" };
{ \ar@{-} "rmm"; "rmrm" };
{ \ar@{-} "rmm"; "rmrr" };
\end{xy}
}
\bigminus
\adjustbox{valign=m}{
\begin{xy}
(  3,   0 )*{} = "root";
(  0,  -4 )*{} = "l";
(  3,  -4 )*{} = "m";
(  6,  -4 )*{} = "r";
{ \ar@{-} "root"; "l" };
{ \ar@{-} "root"; "m" };
{ \ar@{-} "root"; "r" };
(  3, -8 )*{} = "rl";
(  6, -8 )*{} = "rm";
(  9, -8 )*{} = "rr";
{ \ar@{-} "r"; "rl" };
{ \ar@{-} "r"; "rm" };
{ \ar@{-} "r"; "rr" };
( 6, -12 )*{} = "rrl";
(  9, -12 )*{} = "rrm";
(  12, -12 )*{} = "rrr";
{ \ar@{-} "rr"; "rrl" };
{ \ar@{-} "rr"; "rrm" };
{ \ar@{-} "rr"; "rrr" };
(  3, -16 )*{} = "rrll";
(  6, -16 )*{} = "rrlm";
(  9, -16 )*{} = "rrlr";
{ \ar@{-} "rrl"; "rrll" };
{ \ar@{-} "rrl"; "rrlm" };
{ \ar@{-} "rrl"; "rrlr" };
\end{xy}
}
\]
We reduce terms 3, 5 by $\alpha$:
\[
\bigminus
\scalebox{1.5}{2}
\adjustbox{valign=m}{
\begin{xy}
(  0,   0 )*{} = "root";
( -6,  -4 )*{} = "l";
(  0,  -4 )*{} = "m";
(  6,  -4 )*{} = "r";
{ \ar@{-} "root"; "l" };
{ \ar@{-} "root"; "m" };
{ \ar@{-} "root"; "r" };
( -3,  -8 )*{} = "rl";
(  6,  -8 )*{} = "rm";
( 15,  -8 )*{} = "rr";
{ \ar@{-} "r"; "rl" };
{ \ar@{-} "r"; "rm" };
{ \ar@{-} "r"; "rr" };
(  3, -12 )*{} = "rml";
(  6, -12 )*{} = "rmm";
(  9, -12 )*{} = "rmr";
{ \ar@{-} "rm"; "rml" };
{ \ar@{-} "rm"; "rmm" };
{ \ar@{-} "rm"; "rmr" };
( 12, -12 )*{} = "rrl";
( 15, -12 )*{} = "rrm";
( 18, -12 )*{} = "rrr";
{ \ar@{-} "rr"; "rrl" };
{ \ar@{-} "rr"; "rrm" };
{ \ar@{-} "rr"; "rrr" };
\end{xy}
}
\bigplus
\adjustbox{valign=m}{
\begin{xy}
(  3,   0 )*{\bullet} = "root";
(  0,  -4 )*{} = "l";
(  3,  -4 )*{} = "m";
(  6,  -4 )*{} = "r";
{ \ar@{-} "root"; "l" };
{ \ar@{-} "root"; "m" };
{ \ar@{-} "root"; "r" };
(  3, -8 )*{} = "ml";
(  6, -8 )*{} = "mm";
(  9, -8 )*{} = "mr";
{ \ar@{-} "r"; "ml" };
{ \ar@{-} "r"; "mm" };
{ \ar@{-} "r"; "mr" };
( 3, -12 )*{} = "mrl";
(  6, -12 )*{} = "mrm";
(  9, -12 )*{} = "mrr";
{ \ar@{-} "mm"; "mrl" };
{ \ar@{-} "mm"; "mrm" };
{ \ar@{-} "mm"; "mrr" };
(  6, -16 )*{} = "mrrl";
(  9, -16 )*{} = "mrrm";
(  12, -16 )*{} = "mrrr";
{ \ar@{-} "mrr"; "mrrl" };
{ \ar@{-} "mrr"; "mrrm" };
{ \ar@{-} "mrr"; "mrrr" };
\end{xy}
}
\bigplus
\adjustbox{valign=m}{
\begin{xy}
(  3,   0 )*{} = "root";
(  0,  -4 )*{} = "l";
(  3,  -4 )*{} = "m";
(  6,  -4 )*{} = "r";
{ \ar@{-} "root"; "l" };
{ \ar@{-} "root"; "m" };
{ \ar@{-} "root"; "r" };
(  3, -8 )*{} = "rl";
(  6, -8 )*{} = "rm";
(  9, -8 )*{} = "rr";
{ \ar@{-} "r"; "rl" };
{ \ar@{-} "r"; "rm" };
{ \ar@{-} "r"; "rr" };
( 6, -12 )*{} = "rrl";
(  9, -12 )*{} = "rrm";
(  12, -12 )*{} = "rrr";
{ \ar@{-} "rr"; "rrl" };
{ \ar@{-} "rr"; "rrm" };
{ \ar@{-} "rr"; "rrr" };
(  6, -16 )*{} = "rrml";
(  9, -16 )*{} = "rrmm";
(  12, -16 )*{} = "rrmr";
{ \ar@{-} "rrm"; "rrml" };
{ \ar@{-} "rrm"; "rrmm" };
{ \ar@{-} "rrm"; "rrmr" };
\end{xy}
}
\]
If we reduce term 2 using $\beta$ then two terms cancel and we obtain $-\gamma$.
\end{proof}


\begin{lemma}
\label{degree9case2}
Identifying the first $t$ of $\ell m(\alpha) = t \circ_1 t$
and the first $t$ of $\ell m(\beta) = t \circ_2 (t  \circ_3  t)$
we obtain the S-polynomial $\delta$, and
$\{ \alpha, \beta, \delta \}$ is self-reduced:
\begin{align*}
\delta
&=
\adjustbox{valign=m}{
\begin{xy}
(  3,   0 )*{\bullet} = "root";
(  -3,  -4 )*{} = "l";
(  3,  -4 )*{} = "m";
(  12,  -4 )*{} = "r";
{ \ar@{-} "root"; "l" };
{ \ar@{-} "root"; "m" };
{ \ar@{-} "root"; "r" };
( 0, -8 )*{} = "ll";
(  3, -8 )*{} = "lm";
(  6, -8 )*{} = "lr";
{ \ar@{-} "m"; "ll" };
{ \ar@{-} "m"; "lm" };
{ \ar@{-} "m"; "lr" };
( 9, -8 )*{} = "ml";
(  12, -8 )*{} = "mm";
(  15, -8 )*{} = "mr";
{ \ar@{-} "r"; "mr" };
{ \ar@{-} "r"; "mm" };
{ \ar@{-} "r"; "ml" };
(  9, -12 )*{} = "mll";
(  12, -12 )*{} = "mlm";
(  15, -12 )*{} = "mlr";
{ \ar@{-} "mm"; "mll" };
{ \ar@{-} "mm"; "mlm" };
{ \ar@{-} "mm"; "mlr" };
\end{xy}
}
\bigplus
\adjustbox{valign=m}{
\begin{xy}
(  3,   0 )*{} = "root";
(  -3,  -4 )*{} = "l";
(  3,  -4 )*{} = "m";
(  12,  -4 )*{} = "r";
{ \ar@{-} "root"; "l" };
{ \ar@{-} "root"; "m" };
{ \ar@{-} "root"; "r" };
( 0, -8 )*{} = "ml";
(  3, -8 )*{} = "mm";
(  6, -8 )*{} = "mr";
{ \ar@{-} "m"; "ml" };
{ \ar@{-} "m"; "mm" };
{ \ar@{-} "m"; "mr" };
( 9, -8 )*{} = "ml";
(  12, -8 )*{} = "mm";
(  15, -8 )*{} = "mr";
{ \ar@{-} "r"; "mr" };
{ \ar@{-} "r"; "mm" };
{ \ar@{-} "r"; "ml" };
(  12, -12 )*{} = "mll";
(  15, -12 )*{} = "mlm";
(  18, -12 )*{} = "mlr";
{ \ar@{-} "mr"; "mll" };
{ \ar@{-} "mr"; "mlm" };
{ \ar@{-} "mr"; "mlr" };
\end{xy}
}
\bigplus
\adjustbox{valign=m}{
\begin{xy}
(  3,   0 )*{} = "root";
(  0,  -4 )*{} = "l";
(  3,  -4 )*{} = "m";
(  6,  -4 )*{} = "r";
{ \ar@{-} "root"; "l" };
{ \ar@{-} "root"; "m" };
{ \ar@{-} "root"; "r" };
( 3, -8 )*{} = "rl";
(  6, -8 )*{} = "rm";
(  9, -8 )*{} = "rr";
{ \ar@{-} "r"; "rl" };
{ \ar@{-} "r"; "rm" };
{ \ar@{-} "r"; "rr" };
( 6, -12 )*{} = "rrl";
(  9, -12 )*{} = "rrm";
(  12, -12 )*{} = "rrr";
{ \ar@{-} "rr"; "rrr" };
{ \ar@{-} "rr"; "rrm" };
{ \ar@{-} "rr"; "rrl" };
(  9, -16 )*{} = "rrml";
(  12, -16 )*{} = "rrmm";
(  15, -16 )*{} = "rrmr";
{ \ar@{-} "rrr"; "rrml" };
{ \ar@{-} "rrr"; "rrmm" };
{ \ar@{-} "rrr"; "rrmr" };
\end{xy}
}
\\
&=
(t \circ_3  (t \circ_2 t )) \circ_2 t )
+
(t \circ_3  (t \circ_3 t )) \circ_2 t )
+
(t \circ_3  (t \circ_3 (t \circ_3 t ) ).
\end{align*}
\end{lemma}

\begin{proof}
We have the equations
\[
\ell m(\alpha) \circ_4 ( t \circ_3 t )
=
(t \circ_1  t) \circ_4 ( t \circ_3 t )
=
\adjustbox{valign=m}{
\begin{xy}
(   3,   0 )*{} = "root";
(  -6,  -4 )*{} = "l";
(   3,  -4 )*{} = "m";
(   9,  -4 )*{} = "r";
{ \ar@{-} "root"; "l" };
{ \ar@{-} "root"; "m" };
{ \ar@{-} "root"; "r" };
( -9, -8 )*{} = "ll";
( -6, -8 )*{} = "lm";
( -3, -8 )*{} = "lr";
{ \ar@{-} "l"; "ll" };
{ \ar@{-} "l"; "lm" };
{ \ar@{-} "l"; "lr" };
( 0, -8 )*{} = "ml";
(  3, -8 )*{} = "mm";
(  6, -8 )*{} = "mr";
{ \ar@{-} "m"; "mr" };
{ \ar@{-} "m"; "mm" };
{ \ar@{-} "m"; "ml" };
(  3, -12 )*{} = "mll";
(  6, -12 )*{} = "mlm";
(  9, -12 )*{} = "mlr";
{ \ar@{-} "mr"; "mll" };
{ \ar@{-} "mr"; "mlm" };
{ \ar@{-} "mr"; "mlr" };
\end{xy}
}
=
( t \circ_2 (t  \circ_3  t)) \circ_1 t
=
\ell m(\beta) \circ_1 t.
\]
We apply the same partial compositions to $\alpha$ and $\beta$:
\begin{align*}
\alpha \circ_4 ( t \circ_3 t )
&=
  ( t \circ_1 t ) \circ_4 ( t \circ_3 t )
+ ( t \circ_2 t ) \circ_4 ( t \circ_3 t )
+ ( t \circ_3 t ) \circ_4 ( t \circ_3 t ),
\\
\beta  \circ_1 t
&=
   ( t \circ_2 ( t \circ_3 t ) ) \circ_1 t
+  ( t \circ_3 ( t \circ_2 t ) ) \circ_1 t
+  ( t \circ_3 ( t \circ_3 t ) ) \circ_1 t.
\end{align*}
The resulting S-polynomial is
\begin{align*}
&
( t \circ_1 t ) \circ_4 ( t \circ_3 t )
+ ( t \circ_2 t ) \circ_4 ( t \circ_3 t )
+ ( t \circ_3 t ) \circ_4 ( t \circ_3 t )
\\
 \;
&
 - ( t \circ_2 ( t \circ_3 t ) ) \circ_1 t
 - ( t \circ_3 ( t \circ_2 t ) ) \circ_1 t
 - ( t \circ_3 ( t \circ_3 t ) ) \circ_1 t.
\end{align*}
Terms 1, 4 cancel, leaving
\begin{align*}
&
( t \circ_2 t ) \circ_4 ( t \circ_3 t )
+ ( t \circ_3 t ) \circ_4 ( t \circ_3 t )
- ( t \circ_3 ( t \circ_2 t ) ) \circ_1 t
- ( t \circ_3 ( t \circ_3 t ) ) \circ_1 t
\\
&=
\adjustbox{valign=m}{
\begin{xy}
(  3,   0 )*{} = "root";
(  0,  -4 )*{} = "l";
(  3,  -4 )*{} = "m";
(  6,  -4 )*{} = "r";
{ \ar@{-} "root"; "l" };
{ \ar@{-} "root"; "m" };
{ \ar@{-} "root"; "r" };
(  0, -8 )*{} = "ml";
(  3, -8 )*{} = "mm";
(  6, -8 )*{} = "mr";
{ \ar@{-} "m"; "ml" };
{ \ar@{-} "m"; "mm" };
{ \ar@{-} "m"; "mr" };
( 3, -12 )*{} = "mrl";
(  6, -12 )*{} = "mrm";
(  9, -12 )*{} = "mrr";
{ \ar@{-} "mr"; "mrl" };
{ \ar@{-} "mr"; "mrm" };
{ \ar@{-} "mr"; "mrr" };
(  6, -16 )*{} = "mrrl";
(  9, -16 )*{} = "mrrm";
(  12, -16 )*{} = "mrrr";
{ \ar@{-} "mrr"; "mrrl" };
{ \ar@{-} "mrr"; "mrrm" };
{ \ar@{-} "mrr"; "mrrr" };
\end{xy}
}
\bigplus
\adjustbox{valign=m}{
\begin{xy}
(  3,   0 )*{} = "root";
(  0,  -4 )*{} = "l";
(  3,  -4 )*{} = "m";
(  6,  -4 )*{} = "r";
{ \ar@{-} "root"; "l" };
{ \ar@{-} "root"; "m" };
{ \ar@{-} "root"; "r" };
(  3, -8 )*{} = "ml";
(  6, -8 )*{} = "mm";
(  9, -8 )*{} = "mr";
{ \ar@{-} "r"; "ml" };
{ \ar@{-} "r"; "mm" };
{ \ar@{-} "r"; "mr" };
( 3, -12 )*{} = "mrl";
(  6, -12 )*{} = "mrm";
(  9, -12 )*{} = "mrr";
{ \ar@{-} "mm"; "mrl" };
{ \ar@{-} "mm"; "mrm" };
{ \ar@{-} "mm"; "mrr" };
(  6, -16 )*{} = "mrrl";
(  9, -16 )*{} = "mrrm";
(  12, -16 )*{} = "mrrr";
{ \ar@{-} "mrr"; "mrrl" };
{ \ar@{-} "mrr"; "mrrm" };
{ \ar@{-} "mrr"; "mrrr" };
\end{xy}
}
\bigminus
\adjustbox{valign=m}{
\begin{xy}
(  3,   0 )*{\bullet} = "root";
( -3,  -4 )*{} = "l";
(  3,  -4 )*{} = "m";
(  9,  -4 )*{} = "r";
{ \ar@{-} "root"; "l" };
{ \ar@{-} "root"; "m" };
{ \ar@{-} "root"; "r" };
( -6, -8 )*{} = "ll";
( -3, -8 )*{} = "lm";
(  0, -8 )*{} = "lr";
{ \ar@{-} "l"; "ll" };
{ \ar@{-} "l"; "lm" };
{ \ar@{-} "l"; "lr" };
( 6, -8 )*{} = "ml";
(  9, -8 )*{} = "mm";
(  12, -8 )*{} = "mr";
{ \ar@{-} "r"; "mr" };
{ \ar@{-} "r"; "mm" };
{ \ar@{-} "r"; "ml" };
(  6, -12 )*{} = "mll";
(  9, -12 )*{} = "mlm";
(  12, -12 )*{} = "mlr";
{ \ar@{-} "mm"; "mll" };
{ \ar@{-} "mm"; "mlm" };
{ \ar@{-} "mm"; "mlr" };
\end{xy}
}
\bigminus
\adjustbox{valign=m}{
\begin{xy}
(  3,   0 )*{} = "root";
( -3,  -4 )*{} = "l";
(  3,  -4 )*{} = "m";
(  9,  -4 )*{} = "r";
{ \ar@{-} "root"; "l" };
{ \ar@{-} "root"; "m" };
{ \ar@{-} "root"; "r" };
( -6, -8 )*{} = "ll";
( -3, -8 )*{} = "lm";
(  0, -8 )*{} = "lr";
{ \ar@{-} "l"; "ll" };
{ \ar@{-} "l"; "lm" };
{ \ar@{-} "l"; "lr" };
( 6, -8 )*{} = "ml";
(  9, -8 )*{} = "mm";
(  12, -8 )*{} = "mr";
{ \ar@{-} "r"; "mr" };
{ \ar@{-} "r"; "mm" };
{ \ar@{-} "r"; "ml" };
(  9, -12 )*{} = "mll";
(  12, -12 )*{} = "mlm";
(  15, -12 )*{} = "mlr";
{ \ar@{-} "mr"; "mll" };
{ \ar@{-} "mr"; "mlm" };
{ \ar@{-} "mr"; "mlr" };
\end{xy}
}
\end{align*}
We reduce terms 3, 4 using $\alpha$:
\begin{align*}
&
\adjustbox{valign=m}{
\begin{xy}
(  3,   0 )*{\bullet} = "root";
(  0,  -4 )*{} = "l";
(  3,  -4 )*{} = "m";
(  6,  -4 )*{} = "r";
{ \ar@{-} "root"; "l" };
{ \ar@{-} "root"; "m" };
{ \ar@{-} "root"; "r" };
(  0, -8 )*{} = "ml";
(  3, -8 )*{} = "mm";
(  6, -8 )*{} = "mr";
{ \ar@{-} "m"; "ml" };
{ \ar@{-} "m"; "mm" };
{ \ar@{-} "m"; "mr" };
( 3, -12 )*{} = "mrl";
(  6, -12 )*{} = "mrm";
(  9, -12 )*{} = "mrr";
{ \ar@{-} "mr"; "mrl" };
{ \ar@{-} "mr"; "mrm" };
{ \ar@{-} "mr"; "mrr" };
(  6, -16 )*{} = "mrrl";
(  9, -16 )*{} = "mrrm";
(  12, -16 )*{} = "mrrr";
{ \ar@{-} "mrr"; "mrrl" };
{ \ar@{-} "mrr"; "mrrm" };
{ \ar@{-} "mrr"; "mrrr" };
\end{xy}
}
\!\!\!\!\!\! \bigplus
\adjustbox{valign=m}{
\begin{xy}
(  3,   0 )*{} = "root";
(  0,  -4 )*{} = "l";
(  3,  -4 )*{} = "m";
(  6,  -4 )*{} = "r";
{ \ar@{-} "root"; "l" };
{ \ar@{-} "root"; "m" };
{ \ar@{-} "root"; "r" };
(  3, -8 )*{} = "ml";
(  6, -8 )*{} = "mm";
(  9, -8 )*{} = "mr";
{ \ar@{-} "r"; "ml" };
{ \ar@{-} "r"; "mm" };
{ \ar@{-} "r"; "mr" };
( 3, -12 )*{} = "mrl";
(  6, -12 )*{} = "mrm";
(  9, -12 )*{} = "mrr";
{ \ar@{-} "mm"; "mrl" };
{ \ar@{-} "mm"; "mrm" };
{ \ar@{-} "mm"; "mrr" };
(  6, -16 )*{} = "mrrl";
(  9, -16 )*{} = "mrrm";
(  12, -16 )*{} = "mrrr";
{ \ar@{-} "mrr"; "mrrl" };
{ \ar@{-} "mrr"; "mrrm" };
{ \ar@{-} "mrr"; "mrrr" };
\end{xy}
}
\!\!\!\! \bigplus
\adjustbox{valign=m}{
\begin{xy}
(  3,   0 )*{} = "root";
(  -3,  -4 )*{} = "l";
(  3,  -4 )*{} = "m";
(  12,  -4 )*{} = "r";
{ \ar@{-} "root"; "l" };
{ \ar@{-} "root"; "m" };
{ \ar@{-} "root"; "r" };
( 0, -8 )*{} = "ll";
(  3, -8 )*{} = "lm";
(  6, -8 )*{} = "lr";
{ \ar@{-} "m"; "ll" };
{ \ar@{-} "m"; "lm" };
{ \ar@{-} "m"; "lr" };
( 9, -8 )*{} = "ml";
(  12, -8 )*{} = "mm";
(  15, -8 )*{} = "mr";
{ \ar@{-} "r"; "mr" };
{ \ar@{-} "r"; "mm" };
{ \ar@{-} "r"; "ml" };
(  9, -12 )*{} = "mll";
(  12, -12 )*{} = "mlm";
(  15, -12 )*{} = "mlr";
{ \ar@{-} "mm"; "mll" };
{ \ar@{-} "mm"; "mlm" };
{ \ar@{-} "mm"; "mlr" };
\end{xy}
}
\bigplus \!
\adjustbox{valign=m}{
\begin{xy}
(  3,   0 )*{} = "root";
(  -3,  -4 )*{} = "l";
(  3,  -4 )*{} = "m";
(  9,  -4 )*{} = "r";
{ \ar@{-} "root"; "l" };
{ \ar@{-} "root"; "m" };
{ \ar@{-} "root"; "r" };
( 6, -8 )*{} = "rl";
(  9, -8 )*{} = "rm";
(  12, -8 )*{} = "rr";
{ \ar@{-} "r"; "rl" };
{ \ar@{-} "r"; "rm" };
{ \ar@{-} "r"; "rr" };
( 9, -12 )*{} = "rrr";
(  12, -12 )*{} = "rrm";
(  15, -12 )*{} = "rrl";
{ \ar@{-} "rr"; "rrr" };
{ \ar@{-} "rr"; "rrm" };
{ \ar@{-} "rr"; "rrl" };
(  9, -16 )*{} = "rrml";
(  12, -16 )*{} = "rrmm";
(  15, -16 )*{} = "rrmr";
{ \ar@{-} "rrm"; "rrml" };
{ \ar@{-} "rrm"; "rrmm" };
{ \ar@{-} "rrm"; "rrmr" };
\end{xy}
}
\!\! \bigplus
\adjustbox{valign=m}{
\begin{xy}
(  3,   0 )*{} = "root";
(  -3,  -4 )*{} = "l";
(  3,  -4 )*{} = "m";
(  12,  -4 )*{} = "r";
{ \ar@{-} "root"; "l" };
{ \ar@{-} "root"; "m" };
{ \ar@{-} "root"; "r" };
( 0, -8 )*{} = "ml";
(  3, -8 )*{} = "mm";
(  6, -8 )*{} = "mr";
{ \ar@{-} "m"; "ml" };
{ \ar@{-} "m"; "mm" };
{ \ar@{-} "m"; "mr" };
( 9, -8 )*{} = "ml";
(  12, -8 )*{} = "mm";
(  15, -8 )*{} = "mr";
{ \ar@{-} "r"; "mr" };
{ \ar@{-} "r"; "mm" };
{ \ar@{-} "r"; "ml" };
(  12, -12 )*{} = "mll";
(  15, -12 )*{} = "mlm";
(  18, -12 )*{} = "mlr";
{ \ar@{-} "mr"; "mll" };
{ \ar@{-} "mr"; "mlm" };
{ \ar@{-} "mr"; "mlr" };
\end{xy}
}
\!\!\!\!\! \bigplus
\adjustbox{valign=m}{
\begin{xy}
(  3,   0 )*{} = "root";
(  0,  -4 )*{} = "l";
(  3,  -4 )*{} = "m";
(  6,  -4 )*{} = "r";
{ \ar@{-} "root"; "l" };
{ \ar@{-} "root"; "m" };
{ \ar@{-} "root"; "r" };
( 3, -8 )*{} = "rl";
(  6, -8 )*{} = "rm";
(  9, -8 )*{} = "rr";
{ \ar@{-} "r"; "rl" };
{ \ar@{-} "r"; "rm" };
{ \ar@{-} "r"; "rr" };
( 6, -12 )*{} = "rrl";
(  9, -12 )*{} = "rrm";
(  12, -12 )*{} = "rrr";
{ \ar@{-} "rr"; "rrr" };
{ \ar@{-} "rr"; "rrm" };
{ \ar@{-} "rr"; "rrl" };
(  9, -16 )*{} = "rrml";
(  12, -16 )*{} = "rrmm";
(  15, -16 )*{} = "rrmr";
{ \ar@{-} "rrr"; "rrml" };
{ \ar@{-} "rrr"; "rrmm" };
{ \ar@{-} "rrr"; "rrmr" };
\end{xy}
}
\end{align*}
Reducing term 1 using $\beta$ gives
\begin{align*}
& {}
\bigminus
\adjustbox{valign=m}{
\begin{xy}
(  3,   0 )*{} = "root";
(  0,  -4 )*{} = "l";
(  3,  -4 )*{} = "m";
(  6,  -4 )*{} = "r";
{ \ar@{-} "root"; "l" };
{ \ar@{-} "root"; "m" };
{ \ar@{-} "root"; "r" };
(  3, -8 )*{} = "rl";
(  6, -8 )*{} = "rm";
(  9, -8 )*{} = "rr";
{ \ar@{-} "r"; "rl" };
{ \ar@{-} "r"; "rm" };
{ \ar@{-} "r"; "rr" };
( 3, -12 )*{} = "mrl";
(  6, -12 )*{} = "mrm";
(  9, -12 )*{} = "mrr";
{ \ar@{-} "rm"; "mrl" };
{ \ar@{-} "rm"; "mrm" };
{ \ar@{-} "rm"; "mrr" };
(  6, -16 )*{} = "mrrl";
(  9, -16 )*{} = "mrrm";
(  12, -16 )*{} = "mrrr";
{ \ar@{-} "mrr"; "mrrl" };
{ \ar@{-} "mrr"; "mrrm" };
{ \ar@{-} "mrr"; "mrrr" };
\end{xy}
}
\bigminus
\adjustbox{valign=m}{
\begin{xy}
(  3,   0 )*{} = "root";
(  -3,  -4 )*{} = "l";
(  3,  -4 )*{} = "m";
(  9,  -4 )*{} = "r";
{ \ar@{-} "root"; "l" };
{ \ar@{-} "root"; "m" };
{ \ar@{-} "root"; "r" };
( 6, -8 )*{} = "rl";
(  9, -8 )*{} = "rm";
(  12, -8 )*{} = "rr";
{ \ar@{-} "r"; "rl" };
{ \ar@{-} "r"; "rm" };
{ \ar@{-} "r"; "rr" };
( 9, -12 )*{} = "rrr";
(  12, -12 )*{} = "rrm";
(  15, -12 )*{} = "rrl";
{ \ar@{-} "rr"; "rrr" };
{ \ar@{-} "rr"; "rrm" };
{ \ar@{-} "rr"; "rrl" };
(  9, -16 )*{} = "rrml";
(  12, -16 )*{} = "rrmm";
(  15, -16 )*{} = "rrmr";
{ \ar@{-} "rrm"; "rrml" };
{ \ar@{-} "rrm"; "rrmm" };
{ \ar@{-} "rrm"; "rrmr" };
\end{xy}
}
\bigplus
\adjustbox{valign=m}{
\begin{xy}
(  3,   0 )*{} = "root";
(  0,  -4 )*{} = "l";
(  3,  -4 )*{} = "m";
(  6,  -4 )*{} = "r";
{ \ar@{-} "root"; "l" };
{ \ar@{-} "root"; "m" };
{ \ar@{-} "root"; "r" };
(  3, -8 )*{} = "ml";
(  6, -8 )*{} = "mm";
(  9, -8 )*{} = "mr";
{ \ar@{-} "r"; "ml" };
{ \ar@{-} "r"; "mm" };
{ \ar@{-} "r"; "mr" };
( 3, -12 )*{} = "mrl";
(  6, -12 )*{} = "mrm";
(  9, -12 )*{} = "mrr";
{ \ar@{-} "mm"; "mrl" };
{ \ar@{-} "mm"; "mrm" };
{ \ar@{-} "mm"; "mrr" };
(  6, -16 )*{} = "mrrl";
(  9, -16 )*{} = "mrrm";
(  12, -16 )*{} = "mrrr";
{ \ar@{-} "mrr"; "mrrl" };
{ \ar@{-} "mrr"; "mrrm" };
{ \ar@{-} "mrr"; "mrrr" };
\end{xy}
}
\bigplus
\adjustbox{valign=m}{
\begin{xy}
(  3,   0 )*{} = "root";
(  -3,  -4 )*{} = "l";
(  3,  -4 )*{} = "m";
(  12,  -4 )*{} = "r";
{ \ar@{-} "root"; "l" };
{ \ar@{-} "root"; "m" };
{ \ar@{-} "root"; "r" };
( 0, -8 )*{} = "ll";
(  3, -8 )*{} = "lm";
(  6, -8 )*{} = "lr";
{ \ar@{-} "m"; "ll" };
{ \ar@{-} "m"; "lm" };
{ \ar@{-} "m"; "lr" };
( 9, -8 )*{} = "ml";
(  12, -8 )*{} = "mm";
(  15, -8 )*{} = "mr";
{ \ar@{-} "r"; "mr" };
{ \ar@{-} "r"; "mm" };
{ \ar@{-} "r"; "ml" };
(  9, -12 )*{} = "mll";
(  12, -12 )*{} = "mlm";
(  15, -12 )*{} = "mlr";
{ \ar@{-} "mm"; "mll" };
{ \ar@{-} "mm"; "mlm" };
{ \ar@{-} "mm"; "mlr" };
\end{xy}
}
\\[1mm]
& {}
\bigplus
\adjustbox{valign=m}{
\begin{xy}
(  3,   0 )*{} = "root";
(  -3,  -4 )*{} = "l";
(  3,  -4 )*{} = "m";
(  9,  -4 )*{} = "r";
{ \ar@{-} "root"; "l" };
{ \ar@{-} "root"; "m" };
{ \ar@{-} "root"; "r" };
( 6, -8 )*{} = "rl";
(  9, -8 )*{} = "rm";
(  12, -8 )*{} = "rr";
{ \ar@{-} "r"; "rl" };
{ \ar@{-} "r"; "rm" };
{ \ar@{-} "r"; "rr" };
( 9, -12 )*{} = "rrr";
(  12, -12 )*{} = "rrm";
(  15, -12 )*{} = "rrl";
{ \ar@{-} "rr"; "rrr" };
{ \ar@{-} "rr"; "rrm" };
{ \ar@{-} "rr"; "rrl" };
(  9, -16 )*{} = "rrml";
(  12, -16 )*{} = "rrmm";
(  15, -16 )*{} = "rrmr";
{ \ar@{-} "rrm"; "rrml" };
{ \ar@{-} "rrm"; "rrmm" };
{ \ar@{-} "rrm"; "rrmr" };
\end{xy}
}
\bigplus
\adjustbox{valign=m}{
\begin{xy}
(  3,   0 )*{} = "root";
(  -3,  -4 )*{} = "l";
(  3,  -4 )*{} = "m";
(  12,  -4 )*{} = "r";
{ \ar@{-} "root"; "l" };
{ \ar@{-} "root"; "m" };
{ \ar@{-} "root"; "r" };
( 0, -8 )*{} = "ml";
(  3, -8 )*{} = "mm";
(  6, -8 )*{} = "mr";
{ \ar@{-} "m"; "ml" };
{ \ar@{-} "m"; "mm" };
{ \ar@{-} "m"; "mr" };
( 9, -8 )*{} = "ml";
(  12, -8 )*{} = "mm";
(  15, -8 )*{} = "mr";
{ \ar@{-} "r"; "mr" };
{ \ar@{-} "r"; "mm" };
{ \ar@{-} "r"; "ml" };
(  12, -12 )*{} = "mll";
(  15, -12 )*{} = "mlm";
(  18, -12 )*{} = "mlr";
{ \ar@{-} "mr"; "mll" };
{ \ar@{-} "mr"; "mlm" };
{ \ar@{-} "mr"; "mlr" };
\end{xy}
}
\bigplus
\adjustbox{valign=m}{
\begin{xy}
(  3,   0 )*{} = "root";
(  0,  -4 )*{} = "l";
(  3,  -4 )*{} = "m";
(  6,  -4 )*{} = "r";
{ \ar@{-} "root"; "l" };
{ \ar@{-} "root"; "m" };
{ \ar@{-} "root"; "r" };
( 3, -8 )*{} = "rl";
(  6, -8 )*{} = "rm";
(  9, -8 )*{} = "rr";
{ \ar@{-} "r"; "rl" };
{ \ar@{-} "r"; "rm" };
{ \ar@{-} "r"; "rr" };
( 6, -12 )*{} = "rrl";
(  9, -12 )*{} = "rrm";
(  12, -12 )*{} = "rrr";
{ \ar@{-} "rr"; "rrr" };
{ \ar@{-} "rr"; "rrm" };
{ \ar@{-} "rr"; "rrl" };
(  9, -16 )*{} = "rrml";
(  12, -16 )*{} = "rrmm";
(  15, -16 )*{} = "rrmr";
{ \ar@{-} "rrr"; "rrml" };
{ \ar@{-} "rrr"; "rrmm" };
{ \ar@{-} "rrr"; "rrmr" };
\end{xy}
}
\end{align*}
Terms 1, 3 and terms 2, 5 cancel; no further reduction is possible, producing $\delta$.
\end{proof}


\begin{lemma}
\label{degree9case3}
Identifying the first $t$ of $\ell m(\alpha) = t \circ_1 t$
with the second $t$ of $\ell m(\beta) = t \circ_2 (t  \circ_3  t)$
we obtain the S-polynomial $\epsilon$ and
$\{ \alpha, \beta, \epsilon \}$ is self-reduced:
\begin{align*}
\epsilon
&=
\adjustbox{valign=m}{
\begin{xy}
(  3,   0 )*{\bullet} = "root";
( -3,  -4 )*{} = "l";
(  3,  -4 )*{} = "m";
(  12,  -4 )*{} = "r";
{ \ar@{-} "root"; "l" };
{ \ar@{-} "root"; "m" };
{ \ar@{-} "root"; "r" };
( 0, -8 )*{} = "ll";
(  3, -8 )*{} = "lm";
(  6, -8 )*{} = "lr";
{ \ar@{-} "m"; "ll" };
{ \ar@{-} "m"; "lm" };
{ \ar@{-} "m"; "lr" };
( 9, -8 )*{} = "ml";
(  12, -8 )*{} = "mm";
(  15, -8 )*{} = "mr";
{ \ar@{-} "r"; "mr" };
{ \ar@{-} "r"; "mm" };
{ \ar@{-} "r"; "ml" };
(  9, -12 )*{} = "mll";
(  12, -12 )*{} = "mlm";
(  15, -12 )*{} = "mlr";
{ \ar@{-} "mm"; "mll" };
{ \ar@{-} "mm"; "mlm" };
{ \ar@{-} "mm"; "mlr" };
\end{xy}
}
\bigplus
\adjustbox{valign=m}{
\begin{xy}
(  3,   0 )*{} = "root";
(  -3,  -4 )*{} = "l";
(  3,  -4 )*{} = "m";
(  12,  -4 )*{} = "r";
{ \ar@{-} "root"; "l" };
{ \ar@{-} "root"; "m" };
{ \ar@{-} "root"; "r" };
( 0, -8 )*{} = "ml";
(  3, -8 )*{} = "mm";
(  6, -8 )*{} = "mr";
{ \ar@{-} "m"; "ml" };
{ \ar@{-} "m"; "mm" };
{ \ar@{-} "m"; "mr" };
( 9, -8 )*{} = "ml";
(  12, -8 )*{} = "mm";
(  15, -8 )*{} = "mr";
{ \ar@{-} "r"; "mr" };
{ \ar@{-} "r"; "mm" };
{ \ar@{-} "r"; "ml" };
(  12, -12 )*{} = "mll";
(  15, -12 )*{} = "mlm";
(  18, -12 )*{} = "mlr";
{ \ar@{-} "mr"; "mll" };
{ \ar@{-} "mr"; "mlm" };
{ \ar@{-} "mr"; "mlr" };
\end{xy}
}
\bigminus
\adjustbox{valign=m}{
\begin{xy}
(  3,   0 )*{} = "root";
(  0,  -4 )*{} = "l";
(  3,  -4 )*{} = "m";
(  6,  -4 )*{} = "r";
{ \ar@{-} "root"; "l" };
{ \ar@{-} "root"; "m" };
{ \ar@{-} "root"; "r" };
(  0, -8 )*{} = "rl";
(  6, -8 )*{} = "rm";
(  15, -8 )*{} = "rr";
{ \ar@{-} "r"; "rl" };
{ \ar@{-} "r"; "rm" };
{ \ar@{-} "r"; "rr" };
( 12, -12 )*{} = "rrl";
(  15, -12 )*{} = "rrm";
(  18, -12 )*{} = "rrr";
{ \ar@{-} "rr"; "rrl" };
{ \ar@{-} "rr"; "rrm" };
{ \ar@{-} "rr"; "rrr" };
(  3, -12 )*{} = "rll";
(  6, -12 )*{} = "rlm";
(  9, -12 )*{} = "rlr";
{ \ar@{-} "rm"; "rll" };
{ \ar@{-} "rm"; "rlm" };
{ \ar@{-} "rm"; "rlr" };
\end{xy}
}
\bigminus
\adjustbox{valign=m}{
\begin{xy}
(  3,   0 )*{} = "root";
(  0,  -4 )*{} = "l";
(  3,  -4 )*{} = "m";
(  6,  -4 )*{} = "r";
{ \ar@{-} "root"; "l" };
{ \ar@{-} "root"; "m" };
{ \ar@{-} "root"; "r" };
( 3, -8 )*{} = "rl";
(  6, -8 )*{} = "rm";
(  9, -8 )*{} = "rr";
{ \ar@{-} "r"; "rl" };
{ \ar@{-} "r"; "rm" };
{ \ar@{-} "r"; "rr" };
( 6, -12 )*{} = "rrl";
(  9, -12 )*{} = "rrm";
(  12, -12 )*{} = "rrr";
{ \ar@{-} "rr"; "rrr" };
{ \ar@{-} "rr"; "rrm" };
{ \ar@{-} "rr"; "rrl" };
(  9, -16 )*{} = "rrml";
(  12, -16 )*{} = "rrmm";
(  15, -16 )*{} = "rrmr";
{ \ar@{-} "rrr"; "rrml" };
{ \ar@{-} "rrr"; "rrmm" };
{ \ar@{-} "rrr"; "rrmr" };
\end{xy}
}
\\
&=
( ( t \circ_3 t ) \circ_2 t ) \circ_6 t
+
( ( t \circ_3 t ) \circ_2 t ) \circ_7 t
-
( t \circ_3 ( t \circ_3 t ) ) \circ_4 t
-
( t \circ_3 ( t \circ_3 t ) ) \circ_7 t.
\end{align*}
\end{lemma}

\begin{proof}
We have the equations
\[
t \circ_2 ( \ell m(\alpha) \circ_5 t )
=
t \circ_2 ( ( t \circ_1 t ) \circ_5 t )
=
\!\!\!\!
\adjustbox{valign=m}{
\begin{xy}
(  3,   0 )*{} = "root";
( -3,  -4 )*{} = "l";
(  3,  -4 )*{} = "m";
(  9,  -4 )*{} = "r";
{ \ar@{-} "root"; "l" };
{ \ar@{-} "root"; "m" };
{ \ar@{-} "root"; "r" };
( -3, -8 )*{} = "ml";
(  3, -8 )*{} = "mm";
(  9, -8 )*{} = "mr";
{ \ar@{-} "m"; "mr" };
{ \ar@{-} "m"; "mm" };
{ \ar@{-} "m"; "ml" };
(  6, -12 )*{} = "mrl";
(  9, -12 )*{} = "mrm";
(  12, -12 )*{} = "mrr";
{ \ar@{-} "mr"; "mrl" };
{ \ar@{-} "mr"; "mrm" };
{ \ar@{-} "mr"; "mrr" };
( -6, -12 )*{} = "mll";
(  -3, -12 )*{} = "mlm";
(  0, -12 )*{} = "mlr";
{ \ar@{-} "ml"; "mll" };
{ \ar@{-} "ml"; "mlm" };
{ \ar@{-} "ml"; "mlr" };
\end{xy}
}
\!\!\!\!
=
( t \circ_2 (t  \circ_3  t)) \circ_2 t
=
\ell m(\beta) \circ_2  t.
\]
The resulting S-polynomial $t \circ_2 ( \alpha \circ_5 t ) - \beta \circ_2 t$ is
\begin{align*}
&
   t \circ_2  (( t \circ_1 t ) \circ_5  t )
+  t \circ_2 (( t \circ_2 t ) \circ_5  t )
+  t \circ_2 (( t \circ_3 t )\circ_5  t )
\\
& {}
- ( t \circ_2 ( t \circ_3 t ) ) \circ_2 t
- ( t \circ_3 ( t \circ_2 t ) ) \circ_2 t
- ( t \circ_3 ( t \circ_3 t ) ) \circ_2 t.
\end{align*}
Terms 1, 4 cancel, leaving
\begin{align*}
&
  t \circ_2 (( t \circ_2 t ) \circ_5  t )
+ t \circ_2 (( t \circ_3 t )\circ_5  t )
- ( t \circ_3 ( t \circ_2 t ) ) \circ_2 t
- ( t \circ_3 ( t \circ_3 t ) ) \circ_2 t
\\
&=
\adjustbox{valign=m}{
\begin{xy}
(  3,   0 )*{\bullet} = "root";
( -3,  -4 )*{} = "l";
(  3,  -4 )*{} = "m";
(  9,  -4 )*{} = "r";
{ \ar@{-} "root"; "l" };
{ \ar@{-} "root"; "m" };
{ \ar@{-} "root"; "r" };
( -3, -8 )*{} = "ml";
(  3, -8 )*{} = "mm";
(  12, -8 )*{} = "mr";
{ \ar@{-} "m"; "ml" };
{ \ar@{-} "m"; "mm" };
{ \ar@{-} "m"; "mr" };
( 9, -12 )*{} = "mrl";
(  12, -12 )*{} = "mrm";
(  15, -12 )*{} = "mrr";
{ \ar@{-} "mr"; "mrl" };
{ \ar@{-} "mr"; "mrm" };
{ \ar@{-} "mr"; "mrr" };
(  0, -12 )*{} = "mml";
(  3, -12 )*{} = "mmm";
(  6, -12 )*{} = "mmr";
{ \ar@{-} "mm"; "mml" };
{ \ar@{-} "mm"; "mmm" };
{ \ar@{-} "mm"; "mmr" };
\end{xy}
}
\bigplus
\adjustbox{valign=m}{
\begin{xy}
(  3,   0 )*{} = "root";
(  0,  -4 )*{} = "l";
(  3,  -4 )*{} = "m";
(  6,  -4 )*{} = "r";
{ \ar@{-} "root"; "l" };
{ \ar@{-} "root"; "m" };
{ \ar@{-} "root"; "r" };
(  0, -8 )*{} = "ml";
(  3, -8 )*{} = "mm";
(  6, -8 )*{} = "mr";
{ \ar@{-} "m"; "ml" };
{ \ar@{-} "m"; "mm" };
{ \ar@{-} "m"; "mr" };
( 3, -12 )*{} = "mrl";
(  6, -12 )*{} = "mrm";
(  9, -12 )*{} = "mrr";
{ \ar@{-} "mr"; "mrl" };
{ \ar@{-} "mr"; "mrm" };
{ \ar@{-} "mr"; "mrr" };
(  6, -16 )*{} = "mrrl";
(  9, -16 )*{} = "mrrm";
(  12, -16 )*{} = "mrrr";
{ \ar@{-} "mrr"; "mrrl" };
{ \ar@{-} "mrr"; "mrrm" };
{ \ar@{-} "mrr"; "mrrr" };
\end{xy}
}
\bigminus
\adjustbox{valign=m}{
\begin{xy}
(  3,   0 )*{} = "root";
( -3,  -4 )*{} = "l";
(  3,  -4 )*{} = "m";
(  12,  -4 )*{} = "r";
{ \ar@{-} "root"; "l" };
{ \ar@{-} "root"; "m" };
{ \ar@{-} "root"; "r" };
( 0, -8 )*{} = "ll";
(  3, -8 )*{} = "lm";
(  6, -8 )*{} = "lr";
{ \ar@{-} "m"; "ll" };
{ \ar@{-} "m"; "lm" };
{ \ar@{-} "m"; "lr" };
( 9, -8 )*{} = "ml";
(  12, -8 )*{} = "mm";
(  15, -8 )*{} = "mr";
{ \ar@{-} "r"; "mr" };
{ \ar@{-} "r"; "mm" };
{ \ar@{-} "r"; "ml" };
(  9, -12 )*{} = "mll";
(  12, -12 )*{} = "mlm";
(  15, -12 )*{} = "mlr";
{ \ar@{-} "mm"; "mll" };
{ \ar@{-} "mm"; "mlm" };
{ \ar@{-} "mm"; "mlr" };
\end{xy}
}
\bigminus
\adjustbox{valign=m}{
\begin{xy}
(  3,   0 )*{} = "root";
(  -3,  -4 )*{} = "l";
(  3,  -4 )*{} = "m";
(  12,  -4 )*{} = "r";
{ \ar@{-} "root"; "l" };
{ \ar@{-} "root"; "m" };
{ \ar@{-} "root"; "r" };
( 0, -8 )*{} = "ml";
(  3, -8 )*{} = "mm";
(  6, -8 )*{} = "mr";
{ \ar@{-} "m"; "ml" };
{ \ar@{-} "m"; "mm" };
{ \ar@{-} "m"; "mr" };
( 9, -8 )*{} = "ml";
(  12, -8 )*{} = "mm";
(  15, -8 )*{} = "mr";
{ \ar@{-} "r"; "mr" };
{ \ar@{-} "r"; "mm" };
{ \ar@{-} "r"; "ml" };
(  12, -12 )*{} = "mll";
(  15, -12 )*{} = "mlm";
(  18, -12 )*{} = "mlr";
{ \ar@{-} "mr"; "mll" };
{ \ar@{-} "mr"; "mlm" };
{ \ar@{-} "mr"; "mlr" };
\end{xy}
}
\end{align*}
We reduce terms 1, 2 using $\beta$:
\begin{align*}
& {}
\bigminus
\adjustbox{valign=m}{
\begin{xy}
(  3,   0 )*{} = "root";
(  0,  -4 )*{} = "l";
(  3,  -4 )*{} = "m";
(  6,  -4 )*{} = "r";
{ \ar@{-} "root"; "l" };
{ \ar@{-} "root"; "m" };
{ \ar@{-} "root"; "r" };
( -3, -8 )*{} = "rl";
(  6, -8 )*{} = "rm";
(  9, -8 )*{} = "rr";
{ \ar@{-} "r"; "rl" };
{ \ar@{-} "r"; "rm" };
{ \ar@{-} "r"; "rr" };
(  3, -12 )*{} = "rml";
(  6, -12 )*{} = "rmm";
(  9, -12 )*{} = "rmr";
{ \ar@{-} "rm"; "rml" };
{ \ar@{-} "rm"; "rmm" };
{ \ar@{-} "rm"; "rmr" };
( -6, -12 )*{} = "rll";
( -3, -12 )*{} = "rlm";
(  0, -12 )*{} = "rlr";
{ \ar@{-} "rl"; "rll" };
{ \ar@{-} "rl"; "rlm" };
{ \ar@{-} "rl"; "rlr" };
\end{xy}
}
 \bigminus
\adjustbox{valign=m}{
\begin{xy}
(  3,   0 )*{} = "root";
(  0,  -4 )*{} = "l";
(  3,  -4 )*{} = "m";
(  6,  -4 )*{} = "r";
{ \ar@{-} "root"; "l" };
{ \ar@{-} "root"; "m" };
{ \ar@{-} "root"; "r" };
(  1.5, -8 )*{} = "rl";
(  6, -8 )*{} = "rm";
( 10.5, -8 )*{} = "rr";
{ \ar@{-} "r"; "rl" };
{ \ar@{-} "r"; "rm" };
{ \ar@{-} "r"; "rr" };
(   7.5, -12 )*{} = "rrl";
(  10.5, -12 )*{} = "rrm";
(  13.5, -12 )*{} = "rrr";
{ \ar@{-} "rr"; "rrl" };
{ \ar@{-} "rr"; "rrm" };
{ \ar@{-} "rr"; "rrr" };
( -1.5, -12 )*{} = "rll";
(  1.5, -12 )*{} = "rlm";
(  4.5, -12 )*{} = "rlr";
{ \ar@{-} "rl"; "rll" };
{ \ar@{-} "rl"; "rlm" };
{ \ar@{-} "rl"; "rlr" };
\end{xy}
}
 \bigminus
\adjustbox{valign=m}{
\begin{xy}
(  3,   0 )*{} = "root";
(  0,  -4 )*{} = "l";
(  3,  -4 )*{} = "m";
(  6,  -4 )*{} = "r";
{ \ar@{-} "root"; "l" };
{ \ar@{-} "root"; "m" };
{ \ar@{-} "root"; "r" };
(  3, -8 )*{} = "ml";
(  6, -8 )*{} = "mm";
(  9, -8 )*{} = "mr";
{ \ar@{-} "r"; "ml" };
{ \ar@{-} "r"; "mm" };
{ \ar@{-} "r"; "mr" };
( 3, -12 )*{} = "mrl";
(  6, -12 )*{} = "mrm";
(  9, -12 )*{} = "mrr";
{ \ar@{-} "mm"; "mrl" };
{ \ar@{-} "mm"; "mrm" };
{ \ar@{-} "mm"; "mrr" };
(  6, -16 )*{} = "mrrl";
(  9, -16 )*{} = "mrrm";
(  12, -16 )*{} = "mrrr";
{ \ar@{-} "mrr"; "mrrl" };
{ \ar@{-} "mrr"; "mrrm" };
{ \ar@{-} "mrr"; "mrrr" };
\end{xy}
}
\bigminus
\adjustbox{valign=m}{
\begin{xy}
(  3,   0 )*{} = "root";
(  0,  -4 )*{} = "l";
(  3,  -4 )*{} = "m";
(  6,  -4 )*{} = "r";
{ \ar@{-} "root"; "l" };
{ \ar@{-} "root"; "m" };
{ \ar@{-} "root"; "r" };
(  3, -8 )*{} = "rl";
(  6, -8 )*{} = "rm";
(  9, -8 )*{} = "rr";
{ \ar@{-} "r"; "rl" };
{ \ar@{-} "r"; "rm" };
{ \ar@{-} "r"; "rr" };
(  6, -12 )*{} = "rrl";
(  9, -12 )*{} = "rrm";
( 12, -12 )*{} = "rrr";
{ \ar@{-} "rr"; "rrl" };
{ \ar@{-} "rr"; "rrm" };
{ \ar@{-} "rr"; "rrr" };
(  6, -16 )*{} = "rrml";
(  9, -16 )*{} = "rrmm";
(  12, -16 )*{} = "rrmr";
{ \ar@{-} "rrm"; "rrml" };
{ \ar@{-} "rrm"; "rrmm" };
{ \ar@{-} "rrm"; "rrmr" };
\end{xy}
}
 \bigminus
\adjustbox{valign=m}{
\begin{xy}
(  3,   0 )*{\bullet} = "root";
(  0,  -4 )*{} = "l";
(  3,  -4 )*{} = "m";
( 12,  -4 )*{} = "r";
{ \ar@{-} "root"; "l" };
{ \ar@{-} "root"; "m" };
{ \ar@{-} "root"; "r" };
( 0, -8 )*{} = "ll";
(  3, -8 )*{} = "lm";
(  6, -8 )*{} = "lr";
{ \ar@{-} "m"; "ll" };
{ \ar@{-} "m"; "lm" };
{ \ar@{-} "m"; "lr" };
( 9, -8 )*{} = "ml";
(  12, -8 )*{} = "mm";
(  15, -8 )*{} = "mr";
{ \ar@{-} "r"; "mr" };
{ \ar@{-} "r"; "mm" };
{ \ar@{-} "r"; "ml" };
(  9, -12 )*{} = "mll";
(  12, -12 )*{} = "mlm";
(  15, -12 )*{} = "mlr";
{ \ar@{-} "mm"; "mll" };
{ \ar@{-} "mm"; "mlm" };
{ \ar@{-} "mm"; "mlr" };
\end{xy}
}
\bigminus
\adjustbox{valign=m}{
\begin{xy}
(  3,   0 )*{} = "root";
(  0,  -4 )*{} = "l";
(  3,  -4 )*{} = "m";
(  12,  -4 )*{} = "r";
{ \ar@{-} "root"; "l" };
{ \ar@{-} "root"; "m" };
{ \ar@{-} "root"; "r" };
( 0, -8 )*{} = "ml";
(  3, -8 )*{} = "mm";
(  6, -8 )*{} = "mr";
{ \ar@{-} "m"; "ml" };
{ \ar@{-} "m"; "mm" };
{ \ar@{-} "m"; "mr" };
( 9, -8 )*{} = "ml";
(  12, -8 )*{} = "mm";
(  15, -8 )*{} = "mr";
{ \ar@{-} "r"; "mr" };
{ \ar@{-} "r"; "mm" };
{ \ar@{-} "r"; "ml" };
(  12, -12 )*{} = "mll";
(  15, -12 )*{} = "mlm";
(  18, -12 )*{} = "mlr";
{ \ar@{-} "mr"; "mll" };
{ \ar@{-} "mr"; "mlm" };
{ \ar@{-} "mr"; "mlr" };
\end{xy}
}
\end{align*}
Reducing terms 1, 2 using $\alpha$ gives
\begin{align*}
&
\adjustbox{valign=m}{
\begin{xy}
(  3,   0 )*{} = "root";
(  0,  -4 )*{} = "l";
(  3,  -4 )*{} = "m";
(  6,  -4 )*{} = "r";
{ \ar@{-} "root"; "l" };
{ \ar@{-} "root"; "m" };
{ \ar@{-} "root"; "r" };
(  3, -8 )*{} = "rl";
(  6, -8 )*{} = "rm";
(  9, -8 )*{} = "rr";
{ \ar@{-} "r"; "rl" };
{ \ar@{-} "r"; "rm" };
{ \ar@{-} "r"; "rr" };
( 3, -12 )*{} = "mrl";
(  6, -12 )*{} = "mrm";
(  9, -12 )*{} = "mrr";
{ \ar@{-} "rm"; "mrl" };
{ \ar@{-} "rm"; "mrm" };
{ \ar@{-} "rm"; "mrr" };
(  6, -16 )*{} = "mrrl";
(  9, -16 )*{} = "mrrm";
(  12, -16 )*{} = "mrrr";
{ \ar@{-} "mrr"; "mrrl" };
{ \ar@{-} "mrr"; "mrrm" };
{ \ar@{-} "mrr"; "mrrr" };
\end{xy}
}
\bigplus
\adjustbox{valign=m}{
\begin{xy}
(  3,   0 )*{} = "root";
(  0,  -4 )*{} = "l";
(  3,  -4 )*{} = "m";
(  6,  -4 )*{} = "r";
{ \ar@{-} "root"; "l" };
{ \ar@{-} "root"; "m" };
{ \ar@{-} "root"; "r" };
( 3, -8 )*{} = "rl";
(  6, -8 )*{} = "rm";
(  9, -8 )*{} = "rr";
{ \ar@{-} "r"; "rl" };
{ \ar@{-} "r"; "rm" };
{ \ar@{-} "r"; "rr" };
( 6, -12 )*{} = "rrr";
(  9, -12 )*{} = "rrm";
(  12, -12 )*{} = "rrl";
{ \ar@{-} "rr"; "rrr" };
{ \ar@{-} "rr"; "rrm" };
{ \ar@{-} "rr"; "rrl" };
(  6, -16 )*{} = "rrml";
(  9, -16 )*{} = "rrmm";
(  12, -16 )*{} = "rrmr";
{ \ar@{-} "rrm"; "rrml" };
{ \ar@{-} "rrm"; "rrmm" };
{ \ar@{-} "rrm"; "rrmr" };
\end{xy}
}
\bigplus
\adjustbox{valign=m}{
\begin{xy}
(  3,   0 )*{} = "root";
(  0,  -4 )*{} = "l";
(  3,  -4 )*{} = "m";
(  6,  -4 )*{} = "r";
{ \ar@{-} "root"; "l" };
{ \ar@{-} "root"; "m" };
{ \ar@{-} "root"; "r" };
(  0, -8 )*{} = "rl";
(  6, -8 )*{} = "rm";
(  15, -8 )*{} = "rr";
{ \ar@{-} "r"; "rl" };
{ \ar@{-} "r"; "rm" };
{ \ar@{-} "r"; "rr" };
( 12, -12 )*{} = "rrl";
(  15, -12 )*{} = "rrm";
(  18, -12 )*{} = "rrr";
{ \ar@{-} "rr"; "rrl" };
{ \ar@{-} "rr"; "rrm" };
{ \ar@{-} "rr"; "rrr" };
(  3, -12 )*{} = "rll";
(  6, -12 )*{} = "rlm";
(  9, -12 )*{} = "rlr";
{ \ar@{-} "rm"; "rll" };
{ \ar@{-} "rm"; "rlm" };
{ \ar@{-} "rm"; "rlr" };
\end{xy}
}
\bigplus
\adjustbox{valign=m}{
\begin{xy}
(  3,   0 )*{} = "root";
(  0,  -4 )*{} = "l";
(  3,  -4 )*{} = "m";
(  6,  -4 )*{} = "r";
{ \ar@{-} "root"; "l" };
{ \ar@{-} "root"; "m" };
{ \ar@{-} "root"; "r" };
( 3, -8 )*{} = "rl";
(  6, -8 )*{} = "rm";
(  9, -8 )*{} = "rr";
{ \ar@{-} "r"; "rl" };
{ \ar@{-} "r"; "rm" };
{ \ar@{-} "r"; "rr" };
( 6, -12 )*{} = "rrl";
(  9, -12 )*{} = "rrm";
(  12, -12 )*{} = "rrr";
{ \ar@{-} "rr"; "rrr" };
{ \ar@{-} "rr"; "rrm" };
{ \ar@{-} "rr"; "rrl" };
(  9, -16 )*{} = "rrml";
(  12, -16 )*{} = "rrmm";
(  15, -16 )*{} = "rrmr";
{ \ar@{-} "rrr"; "rrml" };
{ \ar@{-} "rrr"; "rrmm" };
{ \ar@{-} "rrr"; "rrmr" };
\end{xy}
}
\\
&
\bigminus
\adjustbox{valign=m}{
\begin{xy}
(  3,   0 )*{} = "root";
(  0,  -4 )*{} = "l";
(  3,  -4 )*{} = "m";
(  6,  -4 )*{} = "r";
{ \ar@{-} "root"; "l" };
{ \ar@{-} "root"; "m" };
{ \ar@{-} "root"; "r" };
(  3, -8 )*{} = "rl";
(  6, -8 )*{} = "rm";
(  9, -8 )*{} = "rr";
{ \ar@{-} "r"; "rl" };
{ \ar@{-} "r"; "rm" };
{ \ar@{-} "r"; "rr" };
(  6, -12 )*{} = "rrl";
(  9, -12 )*{} = "rrm";
( 12, -12 )*{} = "rrr";
{ \ar@{-} "rr"; "rrl" };
{ \ar@{-} "rr"; "rrm" };
{ \ar@{-} "rr"; "rrr" };
(  6, -16 )*{} = "rrml";
(  9, -16 )*{} = "rrmm";
(  12, -16 )*{} = "rrmr";
{ \ar@{-} "rrm"; "rrml" };
{ \ar@{-} "rrm"; "rrmm" };
{ \ar@{-} "rrm"; "rrmr" };
\end{xy}
}
\bigminus
\adjustbox{valign=m}{
\begin{xy}
(  3,   0 )*{} = "root";
(  0,  -4 )*{} = "l";
(  3,  -4 )*{} = "m";
(  6,  -4 )*{} = "r";
{ \ar@{-} "root"; "l" };
{ \ar@{-} "root"; "m" };
{ \ar@{-} "root"; "r" };
(  3, -8 )*{} = "rl";
(  6, -8 )*{} = "rm";
(  9, -8 )*{} = "rr";
{ \ar@{-} "r"; "rl" };
{ \ar@{-} "r"; "rm" };
{ \ar@{-} "r"; "rr" };
( 3, -12 )*{} = "mrl";
(  6, -12 )*{} = "mrm";
(  9, -12 )*{} = "mrr";
{ \ar@{-} "rm"; "mrl" };
{ \ar@{-} "rm"; "mrm" };
{ \ar@{-} "rm"; "mrr" };
(  6, -16 )*{} = "mrrl";
(  9, -16 )*{} = "mrrm";
(  12, -16 )*{} = "mrrr";
{ \ar@{-} "mrr"; "mrrl" };
{ \ar@{-} "mrr"; "mrrm" };
{ \ar@{-} "mrr"; "mrrr" };
\end{xy}
}
\bigminus
\adjustbox{valign=m}{
\begin{xy}
(  3,   0 )*{\bullet} = "root";
(  0,  -4 )*{} = "l";
(  3,  -4 )*{} = "m";
(  12,  -4 )*{} = "r";
{ \ar@{-} "root"; "l" };
{ \ar@{-} "root"; "m" };
{ \ar@{-} "root"; "r" };
( 0, -8 )*{} = "ll";
(  3, -8 )*{} = "lm";
(  6, -8 )*{} = "lr";
{ \ar@{-} "m"; "ll" };
{ \ar@{-} "m"; "lm" };
{ \ar@{-} "m"; "lr" };
( 9, -8 )*{} = "ml";
(  12, -8 )*{} = "mm";
(  15, -8 )*{} = "mr";
{ \ar@{-} "r"; "mr" };
{ \ar@{-} "r"; "mm" };
{ \ar@{-} "r"; "ml" };
(  9, -12 )*{} = "mll";
(  12, -12 )*{} = "mlm";
(  15, -12 )*{} = "mlr";
{ \ar@{-} "mm"; "mll" };
{ \ar@{-} "mm"; "mlm" };
{ \ar@{-} "mm"; "mlr" };
\end{xy}
}
\bigminus
\adjustbox{valign=m}{
\begin{xy}
(  3,   0 )*{} = "root";
(  -3,  -4 )*{} = "l";
(  3,  -4 )*{} = "m";
(  12,  -4 )*{} = "r";
{ \ar@{-} "root"; "l" };
{ \ar@{-} "root"; "m" };
{ \ar@{-} "root"; "r" };
( 0, -8 )*{} = "ml";
(  3, -8 )*{} = "mm";
(  6, -8 )*{} = "mr";
{ \ar@{-} "m"; "ml" };
{ \ar@{-} "m"; "mm" };
{ \ar@{-} "m"; "mr" };
( 9, -8 )*{} = "ml";
(  12, -8 )*{} = "mm";
(  15, -8 )*{} = "mr";
{ \ar@{-} "r"; "mr" };
{ \ar@{-} "r"; "mm" };
{ \ar@{-} "r"; "ml" };
(  12, -12 )*{} = "mll";
(  15, -12 )*{} = "mlm";
(  18, -12 )*{} = "mlr";
{ \ar@{-} "mr"; "mll" };
{ \ar@{-} "mr"; "mlm" };
{ \ar@{-} "mr"; "mlr" };
\end{xy}
}
\end{align*}
Terms 1, 6 and terms 2, 5 cancel.
No further reduction is possible, giving $-\epsilon$.
\end{proof}


\begin{lemma}
\label{degree9case4}
Identifying the first $t$ of $\ell m(\alpha) = t \circ_1 t$
with the third $t$ of $\ell m(\beta) = t \circ_2 (t  \circ_3  t)$
we obtain new S-polynomial $\zeta$, and
$\{ \alpha, \beta, \zeta \}$ is self-reduced:
\[
\zeta
=
\adjustbox{valign=m}{
\begin{xy}
(  3,   0 )*{\bullet} = "root";
(  0,  -4 )*{} = "l";
(  3,  -4 )*{} = "m";
(  6,  -4 )*{} = "r";
{ \ar@{-} "root"; "l" };
{ \ar@{-} "root"; "m" };
{ \ar@{-} "root"; "r" };
(  0, -8 )*{} = "rl";
(  6, -8 )*{} = "rm";
(  15, -8 )*{} = "rr";
{ \ar@{-} "r"; "rl" };
{ \ar@{-} "r"; "rm" };
{ \ar@{-} "r"; "rr" };
( 12, -12 )*{} = "rrl";
(  15, -12 )*{} = "rrm";
(  18, -12 )*{} = "rrr";
{ \ar@{-} "rr"; "rrl" };
{ \ar@{-} "rr"; "rrm" };
{ \ar@{-} "rr"; "rrr" };
(  3, -12 )*{} = "rml";
(  6, -12 )*{} = "rmm";
(  9, -12 )*{} = "rmr";
{ \ar@{-} "rm"; "rml" };
{ \ar@{-} "rm"; "rmm" };
{ \ar@{-} "rm"; "rmr" };
\end{xy}
}
\bigminus
\adjustbox{valign=m}{
\begin{xy}
(  3,   0 )*{} = "root";
(  0,  -4 )*{} = "l";
(  3,  -4 )*{} = "m";
(  6,  -4 )*{} = "r";
{ \ar@{-} "root"; "l" };
{ \ar@{-} "root"; "m" };
{ \ar@{-} "root"; "r" };
( 3, -8 )*{} = "rl";
(  6, -8 )*{} = "rm";
(  9, -8 )*{} = "rr";
{ \ar@{-} "r"; "rl" };
{ \ar@{-} "r"; "rm" };
{ \ar@{-} "r"; "rr" };
( 6, -12 )*{} = "rrl";
(  9, -12 )*{} = "rrm";
(  12, -12 )*{} = "rrr";
{ \ar@{-} "rr"; "rrr" };
{ \ar@{-} "rr"; "rrm" };
{ \ar@{-} "rr"; "rrl" };
(  9, -16 )*{} = "rrml";
(  12, -16 )*{} = "rrmm";
(  15, -16 )*{} = "rrmr";
{ \ar@{-} "rrr"; "rrml" };
{ \ar@{-} "rrr"; "rrmm" };
{ \ar@{-} "rrr"; "rrmr" };
\end{xy}
}
=
t \circ_3 ( ( t \circ_2 t ) \circ_5 t )
-
t \circ_3 ( t \circ_3 ( t \circ_3 t ) ).
\]
\end{lemma}

\begin{proof}
We have the equations
\[
( t \circ_2 t )  \circ_4  \ell m(\alpha)
=
( t \circ_2 t ) \circ_4 (t \circ_1  t)
=
\adjustbox{valign=m}{
\begin{xy}
(  3,   0 )*{} = "root";
(  0,  -4 )*{} = "l";
(  3,  -4 )*{} = "m";
(  6,  -4 )*{} = "r";
{ \ar@{-} "root"; "l" };
{ \ar@{-} "root"; "m" };
{ \ar@{-} "root"; "r" };
( 0, -8 )*{} = "ml";
(  3, -8 )*{} = "mm";
(  6, -8 )*{} = "mr";
{ \ar@{-} "m"; "mr" };
{ \ar@{-} "m"; "mm" };
{ \ar@{-} "m"; "ml" };
(  3, -12 )*{} = "mll";
(  6, -12 )*{} = "mlm";
(  9, -12 )*{} = "mlr";
{ \ar@{-} "mr"; "mll" };
{ \ar@{-} "mr"; "mlm" };
{ \ar@{-} "mr"; "mlr" };
(  0, -16 )*{} = "mlll";
(  3, -16 )*{} = "mllm";
(  6, -16 )*{} = "mllr";
{ \ar@{-} "mll"; "mlll" };
{ \ar@{-} "mll"; "mllm" };
{ \ar@{-} "mll"; "mllr" };
\end{xy}
}
=
( t \circ_2 (t  \circ_3  t)) \circ_4 t
=
\ell m(\beta) \circ_4  t.
\]
The resulting S-polynomial $( t \circ_2 t ) \circ_4 \alpha - \beta \circ_4  t$ is
\begin{align*}
&
   ( t \circ_2 t )  \circ_4 ( t \circ_1 t )
+  ( t \circ_2 t )  \circ_4 ( t \circ_2 t )
+  ( t \circ_2 t )  \circ_4 ( t \circ_3 t )
\\
& {}
- ( t \circ_2 ( t \circ_3 t ) ) \circ_4 t
-  ( t \circ_3 ( t \circ_2 t ) ) \circ_4 t
- ( t \circ_3 ( t \circ_3 t ) ) \circ_4 t.
\end{align*}
Terms 1, 4 cancel, leaving
\begin{align*}
&
  ( t \circ_2 t )  \circ_4 ( t \circ_2 t )
+ ( t \circ_2 t )  \circ_4 ( t \circ_3 t )
- ( t \circ_3 ( t \circ_2 t ) ) \circ_4 t
- ( t \circ_3 ( t \circ_3 t ) ) \circ_4 t
\\
&=
\adjustbox{valign=m}{
\begin{xy}
(  3,   0 )*{\bullet} = "root";
(  0,  -4 )*{} = "l";
(  3,  -4 )*{} = "m";
(  6,  -4 )*{} = "r";
{ \ar@{-} "root"; "l" };
{ \ar@{-} "root"; "m" };
{ \ar@{-} "root"; "r" };
(  0, -8 )*{} = "ml";
(  3, -8 )*{} = "mm";
(  6, -8 )*{} = "mr";
{ \ar@{-} "m"; "ml" };
{ \ar@{-} "m"; "mm" };
{ \ar@{-} "m"; "mr" };
(  3, -12 )*{} = "mrl";
(  6, -12 )*{} = "mrm";
(  9, -12 )*{} = "mrr";
{ \ar@{-} "mr"; "mrl" };
{ \ar@{-} "mr"; "mrm" };
{ \ar@{-} "mr"; "mrr" };
(  3, -16 )*{} = "mrml";
(  6, -16 )*{} = "mrmm";
(  9, -16 )*{} = "mrmr";
{ \ar@{-} "mrm"; "mrml" };
{ \ar@{-} "mrm"; "mrmm" };
{ \ar@{-} "mrm"; "mrmr" };
\end{xy}
}
\bigplus
\adjustbox{valign=m}{
\begin{xy}
(  3,   0 )*{} = "root";
(  0,  -4 )*{} = "l";
(  3,  -4 )*{} = "m";
(  6,  -4 )*{} = "r";
{ \ar@{-} "root"; "l" };
{ \ar@{-} "root"; "m" };
{ \ar@{-} "root"; "r" };
(  0, -8 )*{} = "ml";
(  3, -8 )*{} = "mm";
(  6, -8 )*{} = "mr";
{ \ar@{-} "m"; "ml" };
{ \ar@{-} "m"; "mm" };
{ \ar@{-} "m"; "mr" };
(  3, -12 )*{} = "mrl";
(  6, -12 )*{} = "mrm";
(  9, -12 )*{} = "mrr";
{ \ar@{-} "mr"; "mrl" };
{ \ar@{-} "mr"; "mrm" };
{ \ar@{-} "mr"; "mrr" };
(  6, -16 )*{} = "mrrl";
(  9, -16 )*{} = "mrrm";
( 12, -16 )*{} = "mrrr";
{ \ar@{-} "mrr"; "mrrl" };
{ \ar@{-} "mrr"; "mrrm" };
{ \ar@{-} "mrr"; "mrrr" };
\end{xy}
}
\bigminus
\adjustbox{valign=m}{
\begin{xy}
(  3,   0 )*{} = "root";
(  0,  -4 )*{} = "l";
(  3,  -4 )*{} = "m";
(  6,  -4 )*{} = "r";
{ \ar@{-} "root"; "l" };
{ \ar@{-} "root"; "m" };
{ \ar@{-} "root"; "r" };
(  3, -8 )*{} = "rl";
(  6, -8 )*{} = "rm";
(  9, -8 )*{} = "rr";
{ \ar@{-} "r"; "rl" };
{ \ar@{-} "r"; "rm" };
{ \ar@{-} "r"; "rr" };
( 3, -12 )*{} = "rml";
(  6, -12 )*{} = "rmm";
(  9, -12 )*{} = "rmr";
{ \ar@{-} "rm"; "rml" };
{ \ar@{-} "rm"; "rmm" };
{ \ar@{-} "rm"; "rmr" };
(  0, -16 )*{} = "rmll";
(  3, -16 )*{} = "rmlm";
(  6, -16 )*{} = "rmlr";
{ \ar@{-} "rml"; "rmll" };
{ \ar@{-} "rml"; "rmlm" };
{ \ar@{-} "rml"; "rmlr" };
\end{xy}
}
\bigminus
\adjustbox{valign=m}{
\begin{xy}
(  3,   0 )*{} = "root";
(  0,  -4 )*{} = "l";
(  3,  -4 )*{} = "m";
(  6,  -4 )*{} = "r";
{ \ar@{-} "root"; "l" };
{ \ar@{-} "root"; "m" };
{ \ar@{-} "root"; "r" };
(  0, -8 )*{} = "rl";
(  6, -8 )*{} = "rm";
(  15, -8 )*{} = "rr";
{ \ar@{-} "r"; "rl" };
{ \ar@{-} "r"; "rm" };
{ \ar@{-} "r"; "rr" };
( 12, -12 )*{} = "rrl";
(  15, -12 )*{} = "rrm";
(  18, -12 )*{} = "rrr";
{ \ar@{-} "rr"; "rrl" };
{ \ar@{-} "rr"; "rrm" };
{ \ar@{-} "rr"; "rrr" };
(  3, -12 )*{} = "rml";
(  6, -12 )*{} = "rmm";
(  9, -12 )*{} = "rmr";
{ \ar@{-} "rm"; "rml" };
{ \ar@{-} "rm"; "rmm" };
{ \ar@{-} "rm"; "rmr" };
\end{xy}
}
\end{align*}
We use $\beta$ to reduce terms 1, 2:
\begin{align*}
\bigminus
\adjustbox{valign=m}{
\begin{xy}
(  3,   0 )*{} = "root";
(  0,  -4 )*{} = "l";
(  3,  -4 )*{} = "m";
(  6,  -4 )*{} = "r";
{ \ar@{-} "root"; "l" };
{ \ar@{-} "root"; "m" };
{ \ar@{-} "root"; "r" };
(  3, -8 )*{} = "ml";
(  6, -8 )*{} = "mm";
(  9, -8 )*{} = "mr";
{ \ar@{-} "r"; "ml" };
{ \ar@{-} "r"; "mm" };
{ \ar@{-} "r"; "mr" };
( 3, -12 )*{} = "rml";
(  6, -12 )*{} = "rmm";
(  9, -12 )*{} = "rmr";
{ \ar@{-} "mm"; "rml" };
{ \ar@{-} "mm"; "rmm" };
{ \ar@{-} "mm"; "rmr" };
(  3, -16 )*{} = "rmrl";
(  6, -16 )*{} = "rmrm";
(  9, -16 )*{} = "rmrr";
{ \ar@{-} "rmm"; "rmrl" };
{ \ar@{-} "rmm"; "rmrm" };
{ \ar@{-} "rmm"; "rmrr" };
\end{xy}
}
\bigminus
\adjustbox{valign=m}{
\begin{xy}
(  3,   0 )*{} = "root";
(  0,  -4 )*{} = "l";
(  3,  -4 )*{} = "m";
(  6,  -4 )*{} = "r";
{ \ar@{-} "root"; "l" };
{ \ar@{-} "root"; "m" };
{ \ar@{-} "root"; "r" };
(  3, -8 )*{} = "rl";
(  6, -8 )*{} = "rm";
(  9, -8 )*{} = "rr";
{ \ar@{-} "r"; "rl" };
{ \ar@{-} "r"; "rm" };
{ \ar@{-} "r"; "rr" };
( 6, -12 )*{} = "rrl";
(  9, -12 )*{} = "rrm";
(  12, -12 )*{} = "rrr";
{ \ar@{-} "rr"; "rrl" };
{ \ar@{-} "rr"; "rrm" };
{ \ar@{-} "rr"; "rrr" };
(  3, -16 )*{} = "rrll";
(  6, -16 )*{} = "rrlm";
(  9, -16 )*{} = "rrlr";
{ \ar@{-} "rrl"; "rrll" };
{ \ar@{-} "rrl"; "rrlm" };
{ \ar@{-} "rrl"; "rrlr" };
\end{xy}
}
\bigminus
\adjustbox{valign=m}{
\begin{xy}
(  3,   0 )*{} = "root";
(  0,  -4 )*{} = "l";
(  3,  -4 )*{} = "m";
(  6,  -4 )*{} = "r";
{ \ar@{-} "root"; "l" };
{ \ar@{-} "root"; "m" };
{ \ar@{-} "root"; "r" };
(  3, -8 )*{} = "ml";
(  6, -8 )*{} = "mm";
(  9, -8 )*{} = "mr";
{ \ar@{-} "r"; "ml" };
{ \ar@{-} "r"; "mm" };
{ \ar@{-} "r"; "mr" };
( 3, -12 )*{} = "mrl";
(  6, -12 )*{} = "mrm";
(  9, -12 )*{} = "mrr";
{ \ar@{-} "mm"; "mrl" };
{ \ar@{-} "mm"; "mrm" };
{ \ar@{-} "mm"; "mrr" };
(  6, -16 )*{} = "mrrl";
(  9, -16 )*{} = "mrrm";
(  12, -16 )*{} = "mrrr";
{ \ar@{-} "mrr"; "mrrl" };
{ \ar@{-} "mrr"; "mrrm" };
{ \ar@{-} "mrr"; "mrrr" };
\end{xy}
}
\bigminus
\adjustbox{valign=m}{
\begin{xy}
(  3,   0 )*{} = "root";
(  0,  -4 )*{} = "l";
(  3,  -4 )*{} = "m";
(  6,  -4 )*{} = "r";
{ \ar@{-} "root"; "l" };
{ \ar@{-} "root"; "m" };
{ \ar@{-} "root"; "r" };
(  3, -8 )*{} = "rl";
(  6, -8 )*{} = "rm";
(  9, -8 )*{} = "rr";
{ \ar@{-} "r"; "rl" };
{ \ar@{-} "r"; "rm" };
{ \ar@{-} "r"; "rr" };
( 6, -12 )*{} = "rrl";
(  9, -12 )*{} = "rrm";
(  12, -12 )*{} = "rrr";
{ \ar@{-} "rr"; "rrl" };
{ \ar@{-} "rr"; "rrm" };
{ \ar@{-} "rr"; "rrr" };
(  6, -16 )*{} = "rrml";
(  9, -16 )*{} = "rrmm";
(  12, -16 )*{} = "rrmr";
{ \ar@{-} "rrm"; "rrml" };
{ \ar@{-} "rrm"; "rrmm" };
{ \ar@{-} "rrm"; "rrmr" };
\end{xy}
}
\bigminus
\adjustbox{valign=m}{
\begin{xy}
(  3,   0 )*{\bullet} = "root";
(  0,  -4 )*{} = "l";
(  3,  -4 )*{} = "m";
(  6,  -4 )*{} = "r";
{ \ar@{-} "root"; "l" };
{ \ar@{-} "root"; "m" };
{ \ar@{-} "root"; "r" };
(  3, -8 )*{} = "rl";
(  6, -8 )*{} = "rm";
(  9, -8 )*{} = "rr";
{ \ar@{-} "r"; "rl" };
{ \ar@{-} "r"; "rm" };
{ \ar@{-} "r"; "rr" };
( 3, -12 )*{} = "rml";
(  6, -12 )*{} = "rmm";
(  9, -12 )*{} = "rmr";
{ \ar@{-} "rm"; "rml" };
{ \ar@{-} "rm"; "rmm" };
{ \ar@{-} "rm"; "rmr" };
(  0, -16 )*{} = "rmll";
(  3, -16 )*{} = "rmlm";
(  6, -16 )*{} = "rmlr";
{ \ar@{-} "rml"; "rmll" };
{ \ar@{-} "rml"; "rmlm" };
{ \ar@{-} "rml"; "rmlr" };
\end{xy}
}
\bigminus
\adjustbox{valign=m}{
\begin{xy}
(  3,   0 )*{} = "root";
(  0,  -4 )*{} = "l";
(  3,  -4 )*{} = "m";
(  6,  -4 )*{} = "r";
{ \ar@{-} "root"; "l" };
{ \ar@{-} "root"; "m" };
{ \ar@{-} "root"; "r" };
(  0, -8 )*{} = "rl";
(  6, -8 )*{} = "rm";
(  15, -8 )*{} = "rr";
{ \ar@{-} "r"; "rl" };
{ \ar@{-} "r"; "rm" };
{ \ar@{-} "r"; "rr" };
( 12, -12 )*{} = "rrl";
(  15, -12 )*{} = "rrm";
(  18, -12 )*{} = "rrr";
{ \ar@{-} "rr"; "rrl" };
{ \ar@{-} "rr"; "rrm" };
{ \ar@{-} "rr"; "rrr" };
(  3, -12 )*{} = "rml";
(  6, -12 )*{} = "rmm";
(  9, -12 )*{} = "rmr";
{ \ar@{-} "rm"; "rml" };
{ \ar@{-} "rm"; "rmm" };
{ \ar@{-} "rm"; "rmr" };
\end{xy}
}
\end{align*}
We use $\alpha$ to reduce terms 2, 5:
\begin{align*}
&
\bigminus
\adjustbox{valign=m}{
\begin{xy}
(  3,   0 )*{} = "root";
(  0,  -4 )*{} = "l";
(  3,  -4 )*{} = "m";
(  6,  -4 )*{} = "r";
{ \ar@{-} "root"; "l" };
{ \ar@{-} "root"; "m" };
{ \ar@{-} "root"; "r" };
(  3, -8 )*{} = "ml";
(  6, -8 )*{} = "mm";
(  9, -8 )*{} = "mr";
{ \ar@{-} "r"; "ml" };
{ \ar@{-} "r"; "mm" };
{ \ar@{-} "r"; "mr" };
( 3, -12 )*{} = "rml";
(  6, -12 )*{} = "rmm";
(  9, -12 )*{} = "rmr";
{ \ar@{-} "mm"; "rml" };
{ \ar@{-} "mm"; "rmm" };
{ \ar@{-} "mm"; "rmr" };
(  3, -16 )*{} = "rmrl";
(  6, -16 )*{} = "rmrm";
(  9, -16 )*{} = "rmrr";
{ \ar@{-} "rmm"; "rmrl" };
{ \ar@{-} "rmm"; "rmrm" };
{ \ar@{-} "rmm"; "rmrr" };
\end{xy}
}
\bigplus
\adjustbox{valign=m}{
\begin{xy}
(  3,   0 )*{} = "root";
(  0,  -4 )*{} = "l";
(  3,  -4 )*{} = "m";
(  6,  -4 )*{} = "r";
{ \ar@{-} "root"; "l" };
{ \ar@{-} "root"; "m" };
{ \ar@{-} "root"; "r" };
(  3, -8 )*{} = "rl";
(  6, -8 )*{} = "rm";
(  9, -8 )*{} = "rr";
{ \ar@{-} "r"; "rl" };
{ \ar@{-} "r"; "rm" };
{ \ar@{-} "r"; "rr" };
( 6, -12 )*{} = "rrl";
(  9, -12 )*{} = "rrm";
(  12, -12 )*{} = "rrr";
{ \ar@{-} "rr"; "rrl" };
{ \ar@{-} "rr"; "rrm" };
{ \ar@{-} "rr"; "rrr" };
(  6, -16 )*{} = "rrml";
(  9, -16 )*{} = "rrmm";
(  12, -16 )*{} = "rrmr";
{ \ar@{-} "rrm"; "rrml" };
{ \ar@{-} "rrm"; "rrmm" };
{ \ar@{-} "rrm"; "rrmr" };
\end{xy}
}
\bigplus
\adjustbox{valign=m}{
\begin{xy}
(  3,   0 )*{} = "root";
(  0,  -4 )*{} = "l";
(  3,  -4 )*{} = "m";
(  6,  -4 )*{} = "r";
{ \ar@{-} "root"; "l" };
{ \ar@{-} "root"; "m" };
{ \ar@{-} "root"; "r" };
( 3, -8 )*{} = "rl";
(  6, -8 )*{} = "rm";
(  9, -8 )*{} = "rr";
{ \ar@{-} "r"; "rl" };
{ \ar@{-} "r"; "rm" };
{ \ar@{-} "r"; "rr" };
( 6, -12 )*{} = "rrl";
(  9, -12 )*{} = "rrm";
(  12, -12 )*{} = "rrr";
{ \ar@{-} "rr"; "rrr" };
{ \ar@{-} "rr"; "rrm" };
{ \ar@{-} "rr"; "rrl" };
(  9, -16 )*{} = "rrml";
(  12, -16 )*{} = "rrmm";
(  15, -16 )*{} = "rrmr";
{ \ar@{-} "rrr"; "rrml" };
{ \ar@{-} "rrr"; "rrmm" };
{ \ar@{-} "rrr"; "rrmr" };
\end{xy}
}
\bigminus
\adjustbox{valign=m}{
\begin{xy}
(  3,   0 )*{} = "root";
(  0,  -4 )*{} = "l";
(  3,  -4 )*{} = "m";
(  6,  -4 )*{} = "r";
{ \ar@{-} "root"; "l" };
{ \ar@{-} "root"; "m" };
{ \ar@{-} "root"; "r" };
(  3, -8 )*{} = "ml";
(  6, -8 )*{} = "mm";
(  9, -8 )*{} = "mr";
{ \ar@{-} "r"; "ml" };
{ \ar@{-} "r"; "mm" };
{ \ar@{-} "r"; "mr" };
( 3, -12 )*{} = "mrl";
(  6, -12 )*{} = "mrm";
(  9, -12 )*{} = "mrr";
{ \ar@{-} "mm"; "mrl" };
{ \ar@{-} "mm"; "mrm" };
{ \ar@{-} "mm"; "mrr" };
(  6, -16 )*{} = "mrrl";
(  9, -16 )*{} = "mrrm";
(  12, -16 )*{} = "mrrr";
{ \ar@{-} "mrr"; "mrrl" };
{ \ar@{-} "mrr"; "mrrm" };
{ \ar@{-} "mrr"; "mrrr" };
\end{xy}
}
\\
&
\bigminus
\adjustbox{valign=m}{
\begin{xy}
(  3,   0 )*{} = "root";
(  0,  -4 )*{} = "l";
(  3,  -4 )*{} = "m";
(  6,  -4 )*{} = "r";
{ \ar@{-} "root"; "l" };
{ \ar@{-} "root"; "m" };
{ \ar@{-} "root"; "r" };
(  3, -8 )*{} = "rl";
(  6, -8 )*{} = "rm";
(  9, -8 )*{} = "rr";
{ \ar@{-} "r"; "rl" };
{ \ar@{-} "r"; "rm" };
{ \ar@{-} "r"; "rr" };
( 6, -12 )*{} = "rrl";
(  9, -12 )*{} = "rrm";
(  12, -12 )*{} = "rrr";
{ \ar@{-} "rr"; "rrl" };
{ \ar@{-} "rr"; "rrm" };
{ \ar@{-} "rr"; "rrr" };
(  6, -16 )*{} = "rrml";
(  9, -16 )*{} = "rrmm";
(  12, -16 )*{} = "rrmr";
{ \ar@{-} "rrm"; "rrml" };
{ \ar@{-} "rrm"; "rrmm" };
{ \ar@{-} "rrm"; "rrmr" };
\end{xy}
}
\bigplus
\adjustbox{valign=m}{
\begin{xy}
(  3,   0 )*{} = "root";
(  0,  -4 )*{} = "l";
(  3,  -4 )*{} = "m";
(  6,  -4 )*{} = "r";
{ \ar@{-} "root"; "l" };
{ \ar@{-} "root"; "m" };
{ \ar@{-} "root"; "r" };
(  3, -8 )*{} = "ml";
(  6, -8 )*{} = "mm";
(  9, -8 )*{} = "mr";
{ \ar@{-} "r"; "ml" };
{ \ar@{-} "r"; "mm" };
{ \ar@{-} "r"; "mr" };
( 3, -12 )*{} = "rml";
(  6, -12 )*{} = "rmm";
(  9, -12 )*{} = "rmr";
{ \ar@{-} "mm"; "rml" };
{ \ar@{-} "mm"; "rmm" };
{ \ar@{-} "mm"; "rmr" };
(  3, -16 )*{} = "rmrl";
(  6, -16 )*{} = "rmrm";
(  9, -16 )*{} = "rmrr";
{ \ar@{-} "rmm"; "rmrl" };
{ \ar@{-} "rmm"; "rmrm" };
{ \ar@{-} "rmm"; "rmrr" };
\end{xy}
}
\bigplus
\adjustbox{valign=m}{
\begin{xy}
(  3,   0 )*{} = "root";
(  0,  -4 )*{} = "l";
(  3,  -4 )*{} = "m";
(  6,  -4 )*{} = "r";
{ \ar@{-} "root"; "l" };
{ \ar@{-} "root"; "m" };
{ \ar@{-} "root"; "r" };
(  3, -8 )*{} = "ml";
(  6, -8 )*{} = "mm";
(  9, -8 )*{} = "mr";
{ \ar@{-} "r"; "ml" };
{ \ar@{-} "r"; "mm" };
{ \ar@{-} "r"; "mr" };
( 3, -12 )*{} = "mrl";
(  6, -12 )*{} = "mrm";
(  9, -12 )*{} = "mrr";
{ \ar@{-} "mm"; "mrl" };
{ \ar@{-} "mm"; "mrm" };
{ \ar@{-} "mm"; "mrr" };
(  6, -16 )*{} = "mrrl";
(  9, -16 )*{} = "mrrm";
(  12, -16 )*{} = "mrrr";
{ \ar@{-} "mrr"; "mrrl" };
{ \ar@{-} "mrr"; "mrrm" };
{ \ar@{-} "mrr"; "mrrr" };
\end{xy}
}
\bigminus
\adjustbox{valign=m}{
\begin{xy}
(  3,   0 )*{} = "root";
(  0,  -4 )*{} = "l";
(  3,  -4 )*{} = "m";
(  6,  -4 )*{} = "r";
{ \ar@{-} "root"; "l" };
{ \ar@{-} "root"; "m" };
{ \ar@{-} "root"; "r" };
(  0, -8 )*{} = "rl";
(  6, -8 )*{} = "rm";
(  15, -8 )*{} = "rr";
{ \ar@{-} "r"; "rl" };
{ \ar@{-} "r"; "rm" };
{ \ar@{-} "r"; "rr" };
( 12, -12 )*{} = "rrl";
(  15, -12 )*{} = "rrm";
(  18, -12 )*{} = "rrr";
{ \ar@{-} "rr"; "rrl" };
{ \ar@{-} "rr"; "rrm" };
{ \ar@{-} "rr"; "rrr" };
(  3, -12 )*{} = "rml";
(  6, -12 )*{} = "rmm";
(  9, -12 )*{} = "rmr";
{ \ar@{-} "rm"; "rml" };
{ \ar@{-} "rm"; "rmm" };
{ \ar@{-} "rm"; "rmr" };
\end{xy}
}
\end{align*}
Terms 1, 6 and 2, 5 and 4, 7 cancel.
No further reduction is possible, giving $-\zeta$.
\end{proof}


\begin{lemma}
\label{degree9simplify}
The polynomials $\gamma$, $\delta$, $\epsilon$, $\zeta$ 
span a subspace with basis $\eta$, $\theta$, $\nu$.
\end{lemma}

\begin{proof}
It is easy to see that
\[
\eta
=
\frac13 \big( \gamma + \delta + 2 \epsilon \big),
\qquad
\theta
=
\frac13 \big( 2 \gamma - \delta + \epsilon \big),
\qquad
\nu
=
-\frac13 \big( \gamma - 2 \delta + 2 \epsilon \big),
\]
and that these three polynomials form a basis of
$\mathrm{span}(\gamma,\delta,\epsilon,\zeta)$.
\end{proof}

\begin{lemma}
\label{degree9finished}
Every S-polynomial obtained from $\alpha$, $\beta$, $\eta$, $\theta$, $\nu$ reduces to 0.
\end{lemma}

\begin{proof}
If either $f$ or $g$ is a monomial then clearly every S-polynomial obtained from $f$ and $g$ reduces to 0.
We have already considered S-polynomials from $\alpha$ and $\beta$;
the other cases are $\alpha$, $\eta$ and $\beta$, $\eta$ and $\eta$, $\eta$
with many subcases.
We give details for the most difficult subcase and leave the others as exercises.
These calculations can be simplified using the triangle lemma for nonsymmetric operads
\cite[Prop.~3.5.3.2]{BD}.

We identify the second $t$ of $\ell m ( \alpha )$ with the first $t$ of $\ell m ( \eta )$
and obtain this SCM:
\[
\ell m ( \alpha ) = t \circ_1 t,
\quad
\ell m ( \eta ) = ( t \circ_2 t ) \circ_5 ( t \circ_2 t ),
\quad
( \ell m ( \alpha ) \circ_2 t ) \circ_5 ( t \circ_2 t )
=
t \circ_1 \ell m ( \eta ).
\]
To save space we switch to nonassociative notation.
We obtain the S-polynomial
\[
\begin{array}{l}
( \alpha \circ_2 t ) \circ_5 ( t \circ_2 t )
-
t \circ_1 \eta 
=
\\[1mm]
({\ast}(({\ast}{\ast}{\ast})({\ast}({\ast}{\ast}{\ast}){\ast}){\ast}){\ast}) +
({\ast}({\ast}{\ast}{\ast})(({\ast}({\ast}{\ast}{\ast}){\ast}){\ast}{\ast})) -
(({\ast}({\ast}{\ast}{\ast})({\ast}{\ast}({\ast}{\ast}{\ast})){\ast}{\ast}).
\end{array}
\]
Rewrite rules \eqref{frr} and \eqref{mybeta2} have this form;
the letters represent submonomials:
\begin{align}
\label{anew}
((vwx)yz) \; &\longmapsto \; {} - (v(wxy)z) - (vw(xyz)),
\\
\label{bnew}
(t(uv(wxy))z) \; &\longmapsto \; {} - (tu(v(wxy)z)) - (tu(vw(xyz))).
\end{align}
When we apply \eqref{anew} or \eqref{bnew} we use bars to indicate the submonomials.
To begin we reduce all three monomials in the S-polynomial using $\alpha$ and obtain
\begin{align*}
&
({\ast}((\bar{\ast}\bar{\ast}\bar{\ast})\overline{({\ast}({\ast}{\ast}{\ast}){\ast})}\bar{\ast}){\ast}) +
({\ast}({\ast}{\ast}{\ast})((\bar{\ast}\overline{({\ast}{\ast}{\ast})}\bar{\ast})\bar{\ast}\bar{\ast})) -
((\bar{\ast}\overline{({\ast}{\ast}{\ast})}\,\overline{({\ast}{\ast}({\ast}{\ast}{\ast}))}\bar{\ast}\bar{\ast})
=
\\
&
- ({\ast}({\ast}({\ast}{\ast}({\ast}({\ast}{\ast}{\ast}){\ast})){\ast}){\ast})
- ({\ast}({\ast}{\ast}({\ast}({\ast}({\ast}{\ast}{\ast}){\ast}){\ast})){\ast})
- ({\ast}({\ast}{\ast}{\ast})({\ast}((\bar{\ast}\bar{\ast}\bar{\ast})\bar{\ast}\bar{\ast}){\ast}))
\\
&
- ({\ast}({\ast}{\ast}{\ast})({\ast}({\ast}{\ast}{\ast})({\ast}{\ast}{\ast})))
+ ({\ast}((\bar{\ast}\bar{\ast}\bar{\ast})\overline{({\ast}{\ast}({\ast}{\ast}{\ast}))}\bar{\ast}){\ast})
+ ({\ast}({\ast}{\ast}{\ast})((\bar{\ast}\bar{\ast}\overline{({\ast}{\ast}{\ast})})\bar{\ast}\bar{\ast})).
\end{align*}
Terms 3, 5, 6 reduce using $\alpha$ as indicated; term 4 is $\theta \circ_2 t$ and reduces to 0:
\begin{align*}
&{}
-
({\ast}(\bar{\ast}(\bar{\ast}\bar{\ast}(\bar{\ast}(\overline{{\ast}{\ast}{\ast}})\bar{\ast})\bar{\ast}){\ast})
-
(\bar{\ast}(\bar{\ast}\bar{\ast}(\bar{\ast}(\overline{{\ast}({\ast}{\ast}{\ast}){\ast}})\bar{\ast}))\bar{\ast})
+
({\ast}({\ast}{\ast}{\ast})({\ast}({\ast}({\ast}{\ast}{\ast}){\ast}){\ast}))
\\
&{}
+
({\ast}({\ast}{\ast}{\ast})(\bar{\ast}(\bar{\ast}\bar{\ast}(\bar{\ast}\bar{\ast}\bar{\ast}))\bar{\ast}))
-
({\ast}(\bar{\ast}(\bar{\ast}\bar{\ast}(\bar{\ast}\bar{\ast}(\overline{{\ast}{\ast}{\ast}}))\bar{\ast}){\ast})
-
(\bar{\ast}(\bar{\ast}\bar{\ast}(\bar{\ast}(\overline{{\ast}{\ast}({\ast}{\ast}{\ast})})\bar{\ast})\bar{\ast})
\\
&{}
-
({\ast}({\ast}{\ast}{\ast})({\ast}({\ast}({\ast}{\ast}{\ast}){\ast}){\ast}))
-
({\ast}({\ast}{\ast}{\ast})({\ast}{\ast}({\ast}{\ast}{\ast}){\ast}){\ast})).
\end{align*}
Terms 3, 7 cancel, and terms 1, 2, 4, 5, 6 reduce using $\beta$ as indicated:
\begin{align*}
&
(\bar{\ast}(\bar{\ast}\bar{\ast}(\bar{\ast}(\overline{{\ast}({\ast}{\ast}{\ast}){\ast}}))\bar{\ast})\bar{\ast})
+
(\bar{\ast}(\bar{\ast}\bar{\ast}(\bar{\ast}\bar{\ast}(\overline{({\ast}{\ast}{\ast}){\ast}{\ast}}))){\bar\ast})
+
({\ast}{\ast}({\ast}({\ast}({\ast}({\ast}{\ast}{\ast}){\ast}){\ast}){\ast}))
\\
&{}
+
({\ast}{\ast}({\ast}{\ast}(({\ast}({\ast}{\ast}{\ast}){\ast}){\ast}{\ast}))
-
({\ast}({\ast}{\ast}{\ast})({\ast}{\ast}({\ast}({\ast}{\ast}{\ast}){\ast})))
-
({\ast}({\ast}{\ast}{\ast})({\ast}{\ast}({\ast}{\ast}({\ast}{\ast}{\ast}))))
\\
&{}
+
(\bar{\ast}(\bar{\ast}\bar{\ast}(\bar{\ast}(\overline{{\ast}{\ast}({\ast}{\ast}{\ast})})\bar{\ast}))\bar{\ast})
+
(\bar{\ast}(\bar{\ast}\bar{\ast}(\bar{\ast}\bar{\ast}(\overline{{\ast}({\ast}{\ast}{\ast}){\ast}}))\bar{\ast})
+
({\ast}{\ast}({\ast}({\ast}({\ast}{\ast}({\ast}{\ast}{\ast})){\ast})){\ast})
\\
&{}
+
({\ast}{\ast}({\ast}{\ast}(({\ast}{\ast}({\ast}{\ast}{\ast})){\ast}{\ast})))
-
({\ast}({\ast}{\ast}{\ast})({\ast}{\ast}(({\ast}{\ast}{\ast}){\ast}{\ast}))).
\end{align*}
Terms 1, 2, 7, 8 reduce using $\beta$ as indicated; term 6 is $\nu \circ_2 t$ and reduces to 0;
omitting terms which cancel, we obtain
\begin{align*}
&
({\ast}{\ast}({\ast}({\ast}{\ast}((\bar{\ast}\bar{\ast}\bar{\ast}){\bar\ast}\bar{\ast})){\ast}))
+
({\ast}{\ast}({\ast}{\ast}({\ast}((\bar{\ast}\bar{\ast}\bar{\ast})\bar{\ast}\bar{\ast}){\ast})))
-
({\ast}({\ast}{\ast}{\ast})({\ast}{\ast}({\ast}({\ast}{\ast}{\ast}){\ast})))
\\
&{}
-
({\ast}{\ast}({\ast}({\ast}({\ast}{\ast}({\ast}{\ast}{\ast})){\ast}){\ast}))
-
({\ast}{\ast}({\ast}({\ast}{\ast}({\ast}({\ast}{\ast}{\ast}){\ast})){\ast}))
-
({\ast}{\ast}({\ast}{\ast}({\ast}({\ast}({\ast}{\ast}{\ast})){\ast}){\ast}))
\\
&{}
+
({\ast}{\ast}({\ast}{\ast}(({\ast}{\ast}({\ast}{\ast}{\ast})){\ast}{\ast})))
-
({\ast}({\ast}{\ast}{\ast})({\ast}{\ast}((\bar{\ast}\bar{\ast}\bar{\ast})\bar{\ast}\bar{\ast}))).
\end{align*}
Terms 1, 2, 8 reduce using $\alpha$ as indicated; omitting terms which cancel, we obtain:
\begin{align*}
&{}
-2
({\ast}{\ast}(\bar{\ast}(\bar{\ast}\bar{\ast}(\bar{\ast}(\overline{{\ast}{\ast}{\ast}}){\bar\ast})){\bar\ast}))
-2
({\ast}{\ast}({\ast}{\ast}({\ast}({\ast}({\ast}{\ast}{\ast}){\ast}){\ast})))
-
({\ast}{\ast}(\bar{\ast}(\bar{\ast}\bar{\ast}(\bar{\ast}\bar{\ast}(\overline{{\ast}{\ast}{\ast}}))){\bar\ast}))
\\
&{}
-
({\ast}{\ast}({\ast}{\ast}(\bar{\ast}(\bar{\ast}\bar{\ast}(\bar{\ast}\bar{\ast}\bar{\ast}))\bar{\ast}))).
\end{align*}
Terms 1, 3, 4 reduce using $\beta$ as indicated.
Some terms cancel, and others reduce to 0 using $\nu$, leaving the single term
$({\ast}{\ast}({\ast}{\ast}(\bar{\ast}(\bar{\ast}\bar{\ast}(\bar{\ast}\bar{\ast}\bar{\ast}))\bar{\ast})))$.
We reduce using $\beta$ and then both terms reduce to 0 using $\nu$.
This completes the proof of Theorem \ref{eventheorem}.
\end{proof}


We use Theorem \ref{eventheorem} to calculate the dimensions of
the homogeneous components of the ternary partially associative operad
$\tpa = \freeternary / \langle \alpha \rangle$
with an operation of even (homological) degree.
Theorem \ref{evendimensionformula} below implies the conditional result of Goze \& Remm
\cite[Theorem 15]{GR}; our proof using Gr\"obner bases is much simpler.
For a more general conjecture, see \cite[Conjecture 10.3.2.6, case 6]{BD}.

\begin{lemma}
\label{arity1357}
For $n = 1, 3, 5, 7$ we have
\[
\dim\tpa(1) = \dim\tpa(3) = 1,
\qquad
\dim\tpa(5) = 2,
\qquad
\dim\tpa(7) = 4.
\]
For $\tpa(5)$ a monomial basis in increasing path-lex order is
\[
T_1 =
\adjustbox{valign=m}{
\begin{xy}
(  0,  0 )*{} = "root";
( -3, -4 )*{} = "l";
(  0, -4 )*{} = "m";
(  3, -4 )*{} = "r";
{ \ar@{-} "root"; "l" };
{ \ar@{-} "root"; "m" };
{ \ar@{-} "root"; "r" };
(  0, -8 )*{} = "rl";
(  3, -8 )*{} = "rm";
(  6, -8 )*{} = "rr";
{ \ar@{-} "r"; "rl" };
{ \ar@{-} "r"; "rm" };
{ \ar@{-} "r"; "rr" };
\end{xy}
}
\qquad
,
\qquad
T_2 =
\adjustbox{valign=m}{
\begin{xy}
(  0,  0 )*{} = "root";
( -3, -4 )*{} = "l";
(  0, -4 )*{} = "m";
(  3, -4 )*{} = "r";
{ \ar@{-} "root"; "l" };
{ \ar@{-} "root"; "m" };
{ \ar@{-} "root"; "r" };
( -3, -8 )*{} = "ml";
(  0, -8 )*{} = "mm";
(  3, -8 )*{} = "mr";
{ \ar@{-} "m"; "ml" };
{ \ar@{-} "m"; "mm" };
{ \ar@{-} "m"; "mr" };
\end{xy}
}
\]
For $\tpa(7)$ a monomial basis in increasing path-lex order is
\[
T_3 =
\adjustbox{valign=m}{
\begin{xy}
(  0,  0 )*{} = "root";
( -3, -4 )*{} = "l";
(  0, -4 )*{} = "m";
(  3, -4 )*{} = "r";
{ \ar@{-} "root"; "l" };
{ \ar@{-} "root"; "m" };
{ \ar@{-} "root"; "r" };
(  0, -8 )*{} = "rl";
(  3, -8 )*{} = "rm";
(  6, -8 )*{} = "rr";
{ \ar@{-} "r"; "rl" };
{ \ar@{-} "r"; "rm" };
{ \ar@{-} "r"; "rr" };
(  3, -12 )*{} = "rrl";
(  6, -12 )*{} = "rrm";
(  9, -12 )*{} = "rrr";
{ \ar@{-} "rr"; "rrl" };
{ \ar@{-} "rr"; "rrm" };
{ \ar@{-} "rr"; "rrr" };
\end{xy}
}
\quad
,
\quad
T_4 =
\adjustbox{valign=m}{
\begin{xy}
(  0,  0 )*{} = "root";
( -3, -4 )*{} = "l";
(  0, -4 )*{} = "m";
(  3, -4 )*{} = "r";
{ \ar@{-} "root"; "l" };
{ \ar@{-} "root"; "m" };
{ \ar@{-} "root"; "r" };
(  0, -8 )*{} = "rl";
(  3, -8 )*{} = "rm";
(  6, -8 )*{} = "rr";
{ \ar@{-} "r"; "rl" };
{ \ar@{-} "r"; "rm" };
{ \ar@{-} "r"; "rr" };
(  0, -12 )*{} = "rml";
(  3, -12 )*{} = "rmm";
(  6, -12 )*{} = "rmr";
{ \ar@{-} "rm"; "rml" };
{ \ar@{-} "rm"; "rmm" };
{ \ar@{-} "rm"; "rmr" };
\end{xy}
}
\quad
,
\quad
T_5 =
\adjustbox{valign=m}{
\begin{xy}
(  0,  0 )*{} = "root";
( -9, -4 )*{} = "l";
(  0, -4 )*{} = "m";
(  9, -4 )*{} = "r";
{ \ar@{-} "root"; "l" };
{ \ar@{-} "root"; "m" };
{ \ar@{-} "root"; "r" };
( -3, -8 )*{} = "ml";
(  0, -8 )*{} = "mm";
(  3, -8 )*{} = "mr";
{ \ar@{-} "m"; "ml" };
{ \ar@{-} "m"; "mm" };
{ \ar@{-} "m"; "mr" };
(  6, -8 )*{} = "rl";
(  9, -8 )*{} = "rm";
( 12, -8 )*{} = "rr";
{ \ar@{-} "r"; "rl" };
{ \ar@{-} "r"; "rm" };
{ \ar@{-} "r"; "rr" };
\end{xy}
}
\quad
,
\quad
T_6 =
\adjustbox{valign=m}{
\begin{xy}
(  0,  0 )*{} = "root";
( -3, -4 )*{} = "l";
(  0, -4 )*{} = "m";
(  3, -4 )*{} = "r";
{ \ar@{-} "root"; "l" };
{ \ar@{-} "root"; "m" };
{ \ar@{-} "root"; "r" };
( -3, -8 )*{} = "ml";
(  0, -8 )*{} = "mm";
(  3, -8 )*{} = "mr";
{ \ar@{-} "m"; "ml" };
{ \ar@{-} "m"; "mm" };
{ \ar@{-} "m"; "mr" };
( -3, -12 )*{} = "mml";
(  0, -12 )*{} = "mmm";
(  3, -12 )*{} = "mmr";
{ \ar@{-} "mm"; "mml" };
{ \ar@{-} "mm"; "mmm" };
{ \ar@{-} "mm"; "mmr" };
\end{xy}
}
\]
\end{lemma}

\begin{proof}
The case $n = 1$ is trivial, and for $n = 3$ we have only the basic tree $t$.
For $n = 5$, the monomial $t \circ_1 t$ reduces by $\alpha$,
leaving only $T_1 = t \circ_2 t$ and $T_2 = t \circ_3 t$.
For $n = 7$, we have 
(i)
$T_1 \circ_i t$:
if $i = 1, 3$ the result reduces by $\alpha$,
and if $i = 2, 4, 5$ we obtain $T_5$, $T_4$, $T_3$;
(ii)
$t \circ_i T_1$:
if $i = 1, 2$ the result reduces by $\alpha$, $\beta$,
and if $i = 3$ we obtain $T_3$;
(iii)
$T_2 \circ_i t$:
if $i = 1, 2$ the result reduces by $\alpha$,
if $i = 3$ we obtain $T_6$,
if $i = 4$ the result reduces by $\beta$,
and if $i = 5$ we obtain $T_5$;
(iv)
$t \circ_i T_2$:
if $i = 1$ the result reduces by $\alpha$,
if $i = 2, 3$ we obtain $T_6$, $T_4$.
Clearly $T_3, \dots, T_6$ cannot be reduced using $\alpha$ or $\beta$,
which proves linear independence.
\end{proof}

\begin{theorem}
\label{evendimensionformula}
For weight $k \ge 3$ we have $\dim \tpa( 2k{+}1 ) = k{+}1$.
\end{theorem}

\begin{proof}
Let $M_0$ be the tree with only one vertex,
set $M_1 = t$, and for $\ell \ge 2$ define
\[
M_\ell = t \circ_2 ( t \circ_2 ( t \circ_2 \cdots ( t \circ_2 t ) \cdots ) )
\qquad
\text{($\ell$ copies of $t$)}.
\]
Consider $k{+}1$ monomials of weight $k$ in increasing path-lex order
where $\boxed{\ell} = M_\ell$:
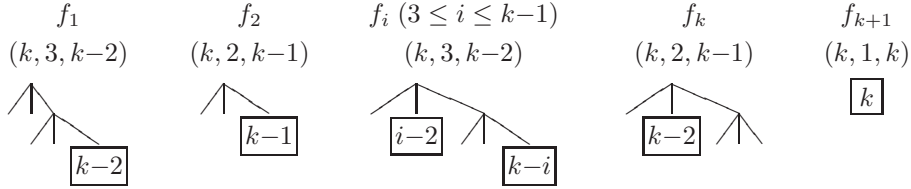
\[
\begin{array}{c@{\qquad}c@{\qquad}c@{\qquad}c@{\qquad}c}
f_1 & f_2 & f_i \; (3 \le i \le k{-}1) & f_k & f_{k+1}
\\[1mm]
(k,3,k{-}2) & (k,2,k{-}1) & (k,3,k{-}2) & (k,2,k{-}1) & (k,1,k)
\\[1mm]
\adjustbox{valign=t}{
\begin{xy}
(  0,  0 )*{} = "root";
( -3, -4 )*{} = "l";
(  0, -4 )*{} = "m";
(  3, -4 )*{} = "r";
{ \ar@{-} "root"; "l" };
{ \ar@{-} "root"; "m" };
{ \ar@{-} "root"; "r" };
(  0, -8 )*{} = "rl";
(  3, -8 )*{} = "rm";
(  9, -8 )*{} = "rr";
{ \ar@{-} "r"; "rl" };
{ \ar@{-} "r"; "rm" };
{ \ar@{-} "r"; "rr" };
(  9, -11 )*{\boxed{\!k{-}2\!}} = "*";
\end{xy}
}
&
\adjustbox{valign=t}{
\begin{xy}
(  0,  0 )*{} = "root";
( -3, -4 )*{} = "l";
(  0, -4 )*{} = "m";
(  6, -4 )*{} = "r";
{ \ar@{-} "root"; "l" };
{ \ar@{-} "root"; "m" };
{ \ar@{-} "root"; "r" };
(  6, -7 )*{\boxed{\!k{-}1\!}} = "*";
\end{xy}
}
&
\adjustbox{valign=t}{
\begin{xy}
(  0,  0 )*{} = "root";
( -6, -4 )*{} = "l";
(  0, -4 )*{} = "m";
(  9, -4 )*{} = "r";
(  0, -7 )*{\boxed{\!i{-}2\!}} = "*";
{ \ar@{-} "root"; "l" };
{ \ar@{-} "root"; "m" };
{ \ar@{-} "root"; "r" };
(  6, -8 )*{} = "rl";
(  9, -8 )*{} = "rm";
( 15, -8 )*{} = "rr";
{ \ar@{-} "r"; "rl" };
{ \ar@{-} "r"; "rm" };
{ \ar@{-} "r"; "rr" };
( 15, -11 )*{\boxed{\!k{-}i\!}} = "*";
\end{xy}
}
&
\adjustbox{valign=t}{
\begin{xy}
(  0,  0 )*{} = "root";
( -6, -4 )*{} = "l";
(  0, -4 )*{} = "m";
(  9, -4 )*{} = "r";
(  0, -7 )*{\boxed{\!k{-}2\!}} = "*";
{ \ar@{-} "root"; "l" };
{ \ar@{-} "root"; "m" };
{ \ar@{-} "root"; "r" };
(  6, -8 )*{} = "rl";
(  9, -8 )*{} = "rm";
( 12, -8 )*{} = "rr";
{ \ar@{-} "r"; "rl" };
{ \ar@{-} "r"; "rm" };
{ \ar@{-} "r"; "rr" };
\end{xy}
}
&
\boxed{k}
\end{array}
\]
We say a leaf is left (middle, right) if it is the left (middle, right) child of its parent.
The ordered triples above give the number of left (middle, right) leaves.
We have
\[
f_1 = t \circ_3 ( t \circ_3 M_{k-2} ),
\;\;
f_2 = t \circ_3 M_{k-1},
\;\;
f_i = ( t \circ_3 ( t \circ_3 M_{k-i} ) ) \circ_2 M_{i-2} \;\; ( 3 \le i \le k ).
\]
For $3 \le i \le k{-}1$, we obtain $f_{i+1}$ from $f_i$ by moving the bottom $t$
of the right-right subtree to the middle subtree.
We will show that $f_1$, $\dots$, $f_{k+1}$ form a basis of $\tpa(2k{+}1)$.
For linear independence, we simply observe that no $f_i$ ($1 \le i \le k$)
can be reduced using any Gr\"obner basis element $\alpha$, $\beta$, $\eta$, $\theta$, $\nu$.

To prove that $f_1$, $\dots$, $f_{k+1}$ span $\tpa(2k{+}1)$ we use induction on $k \ge 3$.
Basis: Lemma \ref{arity1357} gives
$f_1 = T_3$, $f_2 = T_4$, $f_3 = T_5$, $f_4 = T_6$.
Induction:
Assume that $f_1, \dots, f_{k+1}$ span $\tpa(2k{+}1)$ 
and write $f'_1, \dots, f'_{k+2}$ for the monomials of weight $k{+}1$.
For each $f_i$ in $\tpa(2k{+}1)$ we obtain monomials of weight $k{+}1$ in two ways:
\[
(1) \;\;
\text{$t \circ_j f_i$ for $j \in [3]$, $i \in [k{+}1]$};
\qquad
(2) \;\;
\text{$f_i \circ_j t$ for $i \in [k{+}1]$, $j \in [2k{+}1]$}.
\]

Case 1:
If $j = 1$ then $t \circ_1 f_i$ reduces by $\alpha$.
If $j = 2$ then $t \circ_2 f_i$ reduces by $\beta$ for $i \in [k]$,
and $t \circ_2 f_{k+1} = M_{k+1} = f'_{k+2}$.
If $j = 3$ then
$t \circ_3 f_1$ reduces using $\nu$,
$t \circ_3 f_2 = f'_1$,
$t \circ_3 f_i$ reduces using $\theta$ for $i \in [k]$, and
$t \circ_3 f_{k+1} = f'_2$.

Case 2 has three subcases depending on where we attach $t$.
If we attach to a left leaf of $f_i$ then the result reduces by $\alpha$.
If we attach to a right leaf then
for $f_1$ the result reduces by $\nu$ or $\beta$,
for $f_2, \dots, f_k$ the result reduces by $\beta$ or $\theta$,
and for $f_{k+1}$ either we obtain $f'_{k+1}$ or the result reduces by $\beta$.
If we attach to a middle leaf of $f_1$ then we obtain either $f'_3$ or $f'_1$ or the result
reduces by $\theta$.
If we attach to a middle leaf of $f_2$ then we obtain either $f'_3$ or $f'_2$.
If we attach to a middle leaf of $f_i$ for $3 \le i \le k$ then we obtain
$f'_j$ for $3 \le j \le k{+}1$ or the result reduces by $\theta$.
If we attach to the middle leaf of $f_{k+1}$ then we obtain $f'_{k+2}$.
\end{proof}

We now assume that the ternary operation $t$ has odd (homological) degree.
Thus every tree has even or odd parity depending the number of internal nodes.
We write $|f| \in \{0,1\}$ for the parity of $f$.
We must include Koszul signs in the relations for partial compositions:
transposing two odd elements introduces a minus sign.

\begin{lemma}
\label{signs}
\emph{\cite[Def.~1.1]{Markl}}
If $p, q, r \in \mathcal{T}$ then 
\[
( p \circ_i q ) \circ_j r
=
\begin{cases}
\; p \circ_i ( q \circ_{j-i+1} r ) &\quad i \le j \le i+\ell(q)-1
\\
\;
(-1)^{|q||r|} \,
( p \circ_{j-\ell(q)+1} r ) \circ_i q &\quad i+\ell(q) \le j \le \ell(p)+\ell(q)-1
\\
\;
(-1)^{|q||r|} \,
( p \circ_j r ) \circ_{i+\ell(r)-1} q &\quad 1 \le j \le i-1
\end{cases}
\]
\end{lemma}

\begin{theorem}
\label{oddtheorem}
The relation $\alpha$ is a Gr\"obner basis for $\langle \alpha \rangle$ 
in the free nonsymmetric operad with a ternary operation of odd homological degree.
\end{theorem}

\begin{proof}
The first few steps are identical to those for an even operation.
The leading monomial $\ell m(\alpha) = t \circ_1 t$ overlaps with itself in one way
to produce this SCM:
\[
\adjustbox{valign=m}{
\begin{xy}
(  0,  0 )*{} = "root";
( -3, -4 )*{} = "l";
(  0, -4 )*{} = "m";
(  3, -4 )*{} = "r";
{ \ar@{-} "root"; "l" };
{ \ar@{-} "root"; "m" };
{ \ar@{-} "root"; "r" };
( -6, -8 )*{} = "ll";
( -3, -8 )*{} = "lm";
(  0, -8 )*{} = "lr";
{ \ar@{-} "l"; "ll" };
{ \ar@{-} "l"; "lm" };
{ \ar@{-} "l"; "lr" };
( -9, -12 )*{} = "lll";
( -6, -12 )*{} = "llm";
( -3, -12 )*{} = "llr";
{ \ar@{-} "ll"; "lll" };
{ \ar@{-} "ll"; "llm" };
{ \ar@{-} "ll"; "llr" };
\end{xy}
}
= \;
\ell m(\alpha) \circ_1 t
\; = \;
t \circ_1 \ell m(\alpha).
\]
We apply the same partial compositions to $\alpha$ instead of $\ell m(\alpha)$:
\begin{align*}
\alpha \circ_1 t
&=
( t \circ_1 t ) \circ_1 t + ( t \circ_2 t ) \circ_1 t + ( t \circ_3 t ) \circ_1 t
=
\!\!\!\!\!\!\!\!
\adjustbox{valign=m}{
\begin{xy}
(  0,  0 )*{\bullet} = "root";
( -3, -4 )*{} = "l";
(  0, -4 )*{} = "m";
(  3, -4 )*{} = "r";
{ \ar@{-} "root"; "l" };
{ \ar@{-} "root"; "m" };
{ \ar@{-} "root"; "r" };
( -6, -8 )*{} = "ll";
( -3, -8 )*{} = "lm";
(  0, -8 )*{} = "lr";
{ \ar@{-} "l"; "ll" };
{ \ar@{-} "l"; "lm" };
{ \ar@{-} "l"; "lr" };
( -9, -12 )*{} = "lll";
( -6, -12 )*{} = "llm";
( -3, -12 )*{} = "llr";
{ \ar@{-} "ll"; "lll" };
{ \ar@{-} "ll"; "llm" };
{ \ar@{-} "ll"; "llr" };
\end{xy}
}
\bigplus
\!\!\!\!
\adjustbox{valign=m}{
\begin{xy}
(  0,  0 )*{} = "root";
( -9, -4 )*{} = "l";
(  0, -4 )*{} = "m";
(  3, -4 )*{} = "r";
{ \ar@{-} "root"; "l" };
{ \ar@{-} "root"; "m" };
{ \ar@{-} "root"; "r" };
( -12, -8 )*{} = "ll";
(  -9, -8 )*{} = "lm";
(  -6, -8 )*{} = "lr";
{ \ar@{-} "l"; "ll" };
{ \ar@{-} "l"; "lm" };
{ \ar@{-} "l"; "lr" };
( -3, -8 )*{} = "ml";
(  0, -8 )*{} = "mm";
(  3, -8 )*{} = "mr";
{ \ar@{-} "m"; "ml" };
{ \ar@{-} "m"; "mm" };
{ \ar@{-} "m"; "mr" };
\end{xy}
}
\bigplus
\!\!\!\!
\adjustbox{valign=m}{
\begin{xy}
(  0,  0 )*{} = "root";
( -4.5, -4 )*{} = "l";
(  0, -4 )*{} = "m";
(  4.5, -4 )*{} = "r";
{ \ar@{-} "root"; "l" };
{ \ar@{-} "root"; "m" };
{ \ar@{-} "root"; "r" };
( -7.5, -8 )*{} = "ll";
( -4.5, -8 )*{} = "lm";
( -1.5, -8 )*{} = "lr";
{ \ar@{-} "l"; "ll" };
{ \ar@{-} "l"; "lm" };
{ \ar@{-} "l"; "lr" };
(  1.5, -8 )*{} = "rl";
(  4.5, -8 )*{} = "rm";
(  7.5, -8 )*{} = "rr";
{ \ar@{-} "r"; "rl" };
{ \ar@{-} "r"; "rm" };
{ \ar@{-} "r"; "rr" };
\end{xy}
}
\\[1mm]
t \circ_1 \alpha
&=
t \circ_1 ( t \circ_1 t ) + t \circ_1 ( t \circ_2 t ) + t \circ_1 ( t \circ_3 t )
=
\!\!\!\!\!\!\!\!
\adjustbox{valign=m}{
\begin{xy}
(  0,  0 )*{\bullet} = "root";
( -3, -4 )*{} = "l";
(  0, -4 )*{} = "m";
(  3, -4 )*{} = "r";
{ \ar@{-} "root"; "l" };
{ \ar@{-} "root"; "m" };
{ \ar@{-} "root"; "r" };
( -6, -8 )*{} = "ll";
( -3, -8 )*{} = "lm";
(  0, -8 )*{} = "lr";
{ \ar@{-} "l"; "ll" };
{ \ar@{-} "l"; "lm" };
{ \ar@{-} "l"; "lr" };
( -9, -12 )*{} = "lll";
( -6, -12 )*{} = "llm";
( -3, -12 )*{} = "llr";
{ \ar@{-} "ll"; "lll" };
{ \ar@{-} "ll"; "llm" };
{ \ar@{-} "ll"; "llr" };
\end{xy}
}
\bigplus
\adjustbox{valign=m}{
\begin{xy}
(  0,  0 )*{} = "root";
( -3, -4 )*{} = "l";
(  0, -4 )*{} = "m";
(  3, -4 )*{} = "r";
{ \ar@{-} "root"; "l" };
{ \ar@{-} "root"; "m" };
{ \ar@{-} "root"; "r" };
( -6, -8 )*{} = "ll";
( -3, -8 )*{} = "lm";
(  0, -8 )*{} = "lr";
{ \ar@{-} "l"; "ll" };
{ \ar@{-} "l"; "lm" };
{ \ar@{-} "l"; "lr" };
( -6, -12 )*{} = "lml";
( -3, -12 )*{} = "lmm";
(  0, -12 )*{} = "lmr";
{ \ar@{-} "lm"; "lml" };
{ \ar@{-} "lm"; "lmm" };
{ \ar@{-} "lm"; "lmr" };
\end{xy}
}
\bigplus
\adjustbox{valign=m}{
\begin{xy}
(  0,  0 )*{} = "root";
( -3, -4 )*{} = "l";
(  0, -4 )*{} = "m";
(  3, -4 )*{} = "r";
{ \ar@{-} "root"; "l" };
{ \ar@{-} "root"; "m" };
{ \ar@{-} "root"; "r" };
( -6, -8 )*{} = "ll";
( -3, -8 )*{} = "lm";
(  0, -8 )*{} = "lr";
{ \ar@{-} "l"; "ll" };
{ \ar@{-} "l"; "lm" };
{ \ar@{-} "l"; "lr" };
( -3, -12 )*{} = "lrl";
(  0, -12 )*{} = "lrm";
(  3, -12 )*{} = "lrr";
{ \ar@{-} "lr"; "lrl" };
{ \ar@{-} "lr"; "lrm" };
{ \ar@{-} "lr"; "lrr" };
\end{xy}
}
\end{align*}
The difference is this (non-reduced) S-polynomial:
\begin{equation}
\label{oddtrees}
\begin{array}{l}
\alpha \circ_1 t - t \circ_1 \alpha
=
( t \circ_2 t ) \circ_1 t + ( t \circ_3 t ) \circ_1 t
- t \circ_1 ( t \circ_2 t ) - t \circ_1 ( t \circ_3 t )
=
\\
\adjustbox{valign=m}{
\begin{xy}
(  0,  0 )*{} = "root";
( -9, -4 )*{} = "l";
(  0, -4 )*{} = "m";
(  9, -4 )*{} = "r";
{ \ar@{-} "root"; "l" };
{ \ar@{-} "root"; "m" };
{ \ar@{-} "root"; "r" };
( -12, -8 )*{} = "ll";
(  -9, -8 )*{} = "lm";
(  -6, -8 )*{} = "lr";
{ \ar@{-} "l"; "ll" };
{ \ar@{-} "l"; "lm" };
{ \ar@{-} "l"; "lr" };
( -3, -8 )*{} = "ml";
(  0, -8 )*{} = "mm";
(  3, -8 )*{} = "mr";
{ \ar@{-} "m"; "ml" };
{ \ar@{-} "m"; "mm" };
{ \ar@{-} "m"; "mr" };
\end{xy}
}
\bigplus
\adjustbox{valign=m}{
\begin{xy}
(  0,  0 )*{} = "root";
( -6, -4 )*{} = "l";
(  0, -4 )*{} = "m";
(  6, -4 )*{} = "r";
{ \ar@{-} "root"; "l" };
{ \ar@{-} "root"; "m" };
{ \ar@{-} "root"; "r" };
( -9, -8 )*{} = "ll";
( -6, -8 )*{} = "lm";
( -3, -8 )*{} = "lr";
{ \ar@{-} "l"; "ll" };
{ \ar@{-} "l"; "lm" };
{ \ar@{-} "l"; "lr" };
(  3, -8 )*{} = "rl";
(  6, -8 )*{} = "rm";
(  9, -8 )*{} = "rr";
{ \ar@{-} "r"; "rl" };
{ \ar@{-} "r"; "rm" };
{ \ar@{-} "r"; "rr" };
\end{xy}
}
\bigminus
\adjustbox{valign=m}{
\begin{xy}
(  0,  0 )*{\bullet} = "root";
( -3, -4 )*{} = "l";
(  0, -4 )*{} = "m";
(  3, -4 )*{} = "r";
{ \ar@{-} "root"; "l" };
{ \ar@{-} "root"; "m" };
{ \ar@{-} "root"; "r" };
( -6, -8 )*{} = "ll";
( -3, -8 )*{} = "lm";
(  0, -8 )*{} = "lr";
{ \ar@{-} "l"; "ll" };
{ \ar@{-} "l"; "lm" };
{ \ar@{-} "l"; "lr" };
( -6, -12 )*{} = "lml";
( -3, -12 )*{} = "lmm";
(  0, -12 )*{} = "lmr";
{ \ar@{-} "lm"; "lml" };
{ \ar@{-} "lm"; "lmm" };
{ \ar@{-} "lm"; "lmr" };
\end{xy}
}
\bigminus
\adjustbox{valign=m}{
\begin{xy}
(  0,  0 )*{} = "root";
( -3, -4 )*{} = "l";
(  0, -4 )*{} = "m";
(  3, -4 )*{} = "r";
{ \ar@{-} "root"; "l" };
{ \ar@{-} "root"; "m" };
{ \ar@{-} "root"; "r" };
( -6, -8 )*{} = "ll";
( -3, -8 )*{} = "lm";
(  0, -8 )*{} = "lr";
{ \ar@{-} "l"; "ll" };
{ \ar@{-} "l"; "lm" };
{ \ar@{-} "l"; "lr" };
( -3, -12 )*{} = "lrl";
(  0, -12 )*{} = "lrm";
(  3, -12 )*{} = "lrr";
{ \ar@{-} "lr"; "lrl" };
{ \ar@{-} "lr"; "lrm" };
{ \ar@{-} "lr"; "lrr" };
\end{xy}
}
\end{array}
\end{equation}
Lemma \ref{signs} (case 3), $p = q = r = t$, with $i, j = 2, 1$ and $i, j = 3, 1$ gives
\[
( t \circ_2 t ) \circ_1 t = {} - ( t \circ_1 t ) \circ_4 t,
\qquad
( t \circ_3 t ) \circ_1 t = {} - ( t \circ_1 t ) \circ_5 t.
\]
Lemma \ref{signs} (case 1), $p = q = r = t$, with $i, j = 2, 2$ and $i, j = 1, 3$ gives
\[
{} - t \circ_1 ( t \circ_2 t ) = {} -
( t \circ_1 t ) \circ_2 t.
\qquad
{} - t \circ_1 ( t \circ_3 t ) = {} -
( t \circ_1 t ) \circ_3 t.
\]
Therefore \eqref{oddtrees} equals
\[
{}
- ( t \circ_1 t ) \circ_4 t
- ( t \circ_1 t ) \circ_5 t
- ( t \circ_1 t ) \circ_2 t
- ( t \circ_1 t ) \circ_3 t.
\]
We reduce each monomial using $\alpha$ and obtain
\begin{equation}
\label{oddreduce1}
\begin{array}{l}
( t \circ_2 t ) \circ_4 t
+ ( t \circ_3 t ) \circ_4 t
+ ( t \circ_2 t ) \circ_5 t
+ ( t \circ_3 t ) \circ_5 t
\\[1mm]
{}
+ ( t \circ_2 t ) \circ_2 t
+ ( t \circ_3 t ) \circ_2 t
+ ( t \circ_2 t ) \circ_3 t
+ ( t \circ_3 t ) \circ_3 t.
\end{array}
\end{equation}
Terms 3, 6 cancel by Lemma \ref{signs} (case 2),
$( t \circ_2 t ) \circ_5 t = {} - ( t \circ_3 t ) \circ_2 t$,
leaving
\begin{align*}
&
( t \circ_2 t ) \circ_4 t
+ ( t \circ_3 t ) \circ_4 t
+ ( t \circ_3 t ) \circ_5 t
+ ( t \circ_2 t ) \circ_2 t
+ ( t \circ_2 t ) \circ_3 t
+ ( t \circ_3 t ) \circ_3 t
=
\\
&
\adjustbox{valign=m}{
\begin{xy}
(  0,  0 )*{} = "root";
( -3, -4 )*{} = "l";
(  0, -4 )*{} = "m";
(  3, -4 )*{} = "r";
{ \ar@{-} "root"; "l" };
{ \ar@{-} "root"; "m" };
{ \ar@{-} "root"; "r" };
( -3, -8 )*{} = "ml";
(  0, -8 )*{} = "mm";
(  3, -8 )*{} = "mr";
{ \ar@{-} "m"; "ml" };
{ \ar@{-} "m"; "mm" };
{ \ar@{-} "m"; "mr" };
(  0, -12 )*{} = "mrl";
(  3, -12 )*{} = "mrm";
(  6, -12 )*{} = "mrr";
{ \ar@{-} "mr"; "mrl" };
{ \ar@{-} "mr"; "mrm" };
{ \ar@{-} "mr"; "mrr" };
\end{xy}
}
\bigplus
\adjustbox{valign=m}{
\begin{xy}
(  0,  0 )*{} = "root";
( -3, -4 )*{} = "l";
(  0, -4 )*{} = "m";
(  3, -4 )*{} = "r";
{ \ar@{-} "root"; "l" };
{ \ar@{-} "root"; "m" };
{ \ar@{-} "root"; "r" };
(  0, -8 )*{} = "rl";
(  3, -8 )*{} = "rm";
(  6, -8 )*{} = "rr";
{ \ar@{-} "r"; "rl" };
{ \ar@{-} "r"; "rm" };
{ \ar@{-} "r"; "rr" };
(  0, -12 )*{} = "rml";
(  3, -12 )*{} = "rmm";
(  6, -12 )*{} = "rmr";
{ \ar@{-} "rm"; "rml" };
{ \ar@{-} "rm"; "rmm" };
{ \ar@{-} "rm"; "rmr" };
\end{xy}
}
\bigplus
\adjustbox{valign=m}{
\begin{xy}
(  0,  0 )*{} = "root";
( -3, -4 )*{} = "l";
(  0, -4 )*{} = "m";
(  3, -4 )*{} = "r";
{ \ar@{-} "root"; "l" };
{ \ar@{-} "root"; "m" };
{ \ar@{-} "root"; "r" };
(  0, -8 )*{} = "rl";
(  3, -8 )*{} = "rm";
(  6, -8 )*{} = "rr";
{ \ar@{-} "r"; "rl" };
{ \ar@{-} "r"; "rm" };
{ \ar@{-} "r"; "rr" };
(  3, -12 )*{} = "rrl";
(  6, -12 )*{} = "rrm";
(  9, -12 )*{} = "rrr";
{ \ar@{-} "rr"; "rrl" };
{ \ar@{-} "rr"; "rrm" };
{ \ar@{-} "rr"; "rrr" };
\end{xy}
}
\bigplus
\adjustbox{valign=m}{
\begin{xy}
(  0,  0 )*{\bullet} = "root";
( -3, -4 )*{} = "l";
(  0, -4 )*{} = "m";
(  3, -4 )*{} = "r";
{ \ar@{-} "root"; "l" };
{ \ar@{-} "root"; "m" };
{ \ar@{-} "root"; "r" };
( -3, -8 )*{} = "ml";
(  0, -8 )*{} = "mm";
(  3, -8 )*{} = "mr";
{ \ar@{-} "m"; "ml" };
{ \ar@{-} "m"; "mm" };
{ \ar@{-} "m"; "mr" };
( -6, -12 )*{} = "mll";
( -3, -12 )*{} = "mlm";
(  0, -12 )*{} = "mlr";
{ \ar@{-} "ml"; "mll" };
{ \ar@{-} "ml"; "mlm" };
{ \ar@{-} "ml"; "mlr" };
\end{xy}
}
\bigplus
\adjustbox{valign=m}{
\begin{xy}
(  0,  0 )*{} = "root";
( -3, -4 )*{} = "l";
(  0, -4 )*{} = "m";
(  3, -4 )*{} = "r";
{ \ar@{-} "root"; "l" };
{ \ar@{-} "root"; "m" };
{ \ar@{-} "root"; "r" };
( -3, -8 )*{} = "ml";
(  0, -8 )*{} = "mm";
(  3, -8 )*{} = "mr";
{ \ar@{-} "m"; "ml" };
{ \ar@{-} "m"; "mm" };
{ \ar@{-} "m"; "mr" };
( -3, -12 )*{} = "mml";
(  0, -12 )*{} = "mmm";
(  3, -12 )*{} = "mmr";
{ \ar@{-} "mm"; "mml" };
{ \ar@{-} "mm"; "mmm" };
{ \ar@{-} "mm"; "mmr" };
\end{xy}
}
\bigplus
\adjustbox{valign=m}{
\begin{xy}
(  0,  0 )*{} = "root";
( -3, -4 )*{} = "l";
(  0, -4 )*{} = "m";
(  3, -4 )*{} = "r";
{ \ar@{-} "root"; "l" };
{ \ar@{-} "root"; "m" };
{ \ar@{-} "root"; "r" };
(  0, -8 )*{} = "rl";
(  3, -8 )*{} = "rm";
(  6, -8 )*{} = "rr";
{ \ar@{-} "r"; "rl" };
{ \ar@{-} "r"; "rm" };
{ \ar@{-} "r"; "rr" };
( -3, -12 )*{} = "rll";
(  0, -12 )*{} = "rlm";
(  3, -12 )*{} = "rlr";
{ \ar@{-} "rl"; "rll" };
{ \ar@{-} "rl"; "rlm" };
{ \ar@{-} "rl"; "rlr" };
\end{xy}
}
\end{align*}
We reduce terms 4, 6 using $\alpha$; this cancels terms 1, 5 and terms 2, 3.
\end{proof}

\begin{theorem}
For an odd operation, the dimension of the ternary partially associative operad 
the binary Catalan number (in the weight grading).
\end{theorem}

\begin{proof}
Relation $\alpha$ shows that any left subtree reduces,
so the dimension for weight $w$ is the number of binary trees of weight $w$.
\end{proof}



\end{document}